\documentclass[11pt,twoside]{preprint}
\usepackage[osf]{newtxtext}

\usepackage{mathtools}
\usepackage{amsthm}
\usepackage{amssymb}

\usepackage{ifthen}

\usepackage{mhequ}
\usepackage{mhfig}

\usepackage{accents}

\usepackage{tikz}
\usepackage{mathrsfs}
\usepackage{wrapfig}
\usepackage{makeidx} 
\usepackage{microtype} 

\usepackage{hyperref} 

\makeindex

\newlength{\dhatheight}

\def\restr{\mathbin{\upharpoonright}}

\let\MHIndex\theindex
\def\theindex{\MHIndex\markboth{\sc Index}{\sc Index}}

\makeatletter
\renewcommand{\ps@plain}{\pagestyle{MHheadings}}

\renewenvironment{proof}[1][\proofname]{\par
  \pushQED{\qed}%
  \vspace{6\p@\@plus6\p@}%
  \noindent
    \textit{#1\@addpunct{.}}\enskip\ignorespaces
  }{%
    \popQED
    \vspace{6\p@\@plus6\p@}\@endpefalse
  }

\newenvironment{claim}[1][MM]
               {\list{$\bullet$}{%
  \setbox\@tempboxa\hbox{#1}\@tempdima\wd\@tempboxa%
  \setlength{\labelwidth}{\@tempdima}
  \advance\@tempdima by 1em%
  \setlength{\leftmargin}{\@tempdima}
  \setlength{\parsep}{1mm}\setlength{\itemindent}{0mm}%
  \setlength{\labelsep}{2mm}\setlength{\itemsep}{0mm}%
  \setlength{\topsep}{1mm}%
}}{\endlist}

\makeatother

\let\symb\mathbf

\def\eref#1{(\ref{#1})}

\def\WL{W_{\!L}}
\def\WD{W_{\!\Delta}}
\def\HH{\mathring \CH}
\def\E{{\symb E}}

\def\P{{\symb P}}
\def\R{{\mathbf R}}
\def\Z{{\mathbf Z}}
\def\N{{\mathbf N}}
\def\C{{\mathbf C}}
\def\HS{{\mathrm{HS}}}
\def\TV{{\mathrm{TV}}}
\def\Lip{{\mathrm{Lip}}}
\def\cC{{\mathscr C}}
\def\cB{{\mathscr B}}
\def\cM{{\mathscr M}}
\def\cP{{\mathscr P}}
\def\B{{\mathscr B}}
\def\F{{\mathscr F}}

\def\T{\mathbf{T}}

\def\hf{\textstyle{1\over 2}}
\def\Fp{\F_{\mathrm{pr}}}
\def\Lipb{{\mathrm{Lip}_\beta}}
\def\*{|\!|\!|}
\def\Cyl{\mathrm{Cyl}}
\def\X{\mathcal{X}}

\let\d\partial

\def\CO{\mathcal{O}}
\def\CC{\mathcal{C}}
\def\CX{\mathcal{X}}

\def\CB{\mathcal{B}}
\def\CD{\mathcal{D}}
\def\CM{\mathcal{M}}
\def\CH{\mathcal{H}}
\def\CN{\mathcal{N}}
\def\CE{\mathcal{E}}
\def\CL{\mathcal{L}}
\def\CR{\mathcal{R}}
\def\CK{\mathcal{K}}
\def\CP{\mathcal{P}}
\def\CQ{\mathcal{Q}}
\def\CS{\mathcal{S}}

\let\eps\varepsilon

\newcommand\scal[2][ ]{\ifthenelse{\equal{#1}{ }}{\langle#2\rangle}{}
        \ifthenelse{\equal{#1}{b}}{\bigl\langle#2\bigr\rangle}{}
        \ifthenelse{\equal{#1}{B}}{\Bigl\langle#2\Bigr\rangle}{}
        \ifthenelse{\equal{#1}{bb}}{\biggl\langle#2\biggr\rangle}{}
        \ifthenelse{\equal{#1}{BB}}{\Biggl\langle#2\Biggr\rangle}{}}

\undef\div
\DeclareMathOperator{\tr}{tr}
\DeclareMathOperator{\supp}{supp}
\DeclareMathOperator{\diam}{diam}
\DeclareMathOperator{\range}{range}
\DeclareMathOperator{\div}{div}

\newcommand{\eqdef}{\stackrel{\mbox{\tiny def}}{=}}
\newcommand{\one}{\mathbf{1}}

\theoremstyle{plain}
\newtheorem{theorem}{Theorem}[section]
\newtheorem{lemma}[theorem]{Lemma}
\newtheorem{corollary}[theorem]{Corollary}
\newtheorem{proposition}[theorem]{Proposition}

\theoremstyle{definition}
\newtheorem{definition}[theorem]{Definition}

\newtheorem{remark}[theorem]{Remark}
\newtheorem{example}[theorem]{Example}
\newtheorem{exercise}[theorem]{Exercise}

\let\phi\varphi
\let\f\frac

\DeclareSymbolFont{timesoperators}{T1}{ptm}{m}{n}
\SetSymbolFont{timesoperators}{bold}{T1}{ptm}{b}{n}

\makeatletter
\renewcommand{\operator@font}{\mathgroup\symtimesoperators}
\makeatother

\usetikzlibrary{decorations}

\begin{document}

\title{An Introduction to Stochastic PDEs}
\author{Martin Hairer}
\institute{EPFL / Imperial College London}
\maketitle

\setcounter{tocdepth}{2}

\tableofcontents

\section{Introduction}

These notes are based on a series of lectures of various lengths given at the University of Warwick, the Courant
Institute, Imperial College London, and EPFL.
It is an attempt to give a reasonably self-contained presentation of the basic theory of stochastic partial differential
equations, taking for granted basic measure theory, functional analysis and probability theory, but nothing else. 
Since the aim was to present most of the material covered in these notes
during a 30-hours series of postgraduate lectures, such an attempt is doomed to failure unless drastic choices are made.
This is why many important facets of the theory of stochastic PDEs are missing from these notes. In particular,
we do \textit{not} treat equations with multiplicative noise, we do \textit{not} treat equations driven L\'evy noise, we do \textit{not} consider equations with ``rough'' (that is not locally Lipschitz, even in a suitable space) nonlinearities, we do \textit{not}
treat measure-valued processes, we do \textit{not} consider hyperbolic or elliptic problems, we do \textit{not} cover Malliavin calculus
and densities of solutions, etc. The reader who is interested in a more detailed exposition of these more technically subtle parts of the theory
might be advised to read the excellent works \cite{DPZ92,DPZ96,PesZab,Roeck,SanzSole}.

Instead, the approach taken in these notes is to focus on semilinear \textit{parabolic} problems driven by \textit{additive} noise.
These can be treated as stochastic evolution equations in some infinite-dimen\-sional Banach or Hilbert 
space that usually have nice regularising properties
and they already form (in my humble opinion) a very rich class of problems with many interesting properties. Furthermore, this
class of problems has the advantage of allowing to completely pass under silence many subtle problems
arising from stochastic integration in infinite-dimensional spaces.

\subsection{Acknowledgements}

These notes would never have been completed, were it not for the enthusiasm of the attendants of the course.
Hundreds of typos and mistakes were spotted and corrected. I am particularly indebted to David Epstein and Jochen Vo\ss\  
who carefully worked their way through these notes when they were still in a state of wilderness.
Special thanks are also due to Pavel Bubak who was running the tutorials for the course given in Warwick.

\section{Some Motivating Examples}

\subsection{A model for a random string (polymer)}

Take $N+1$ particles with positions $u_n$ immersed in a fluid and assume that nearest-neighbours
are connected by harmonic springs. If the particles are furthermore subject to an external forcing $F$,
the equations of motion (in the overdamped regime where the forces acting on the particle are more important than
inertia, which can also formally be seen as the limit where the masses of the particles go to zero) would be given by
\begin{equs}
{du_0 \over dt} &= k(u_1 - u_0) + F(u_0)\;,\\
{du_n \over dt} &= k(u_{n+1} + u_{n-1} - 2u_n) + F(u_n)\;, \quad n = 1,\ldots,N-1\;,\\
{du_N \over dt} &= k(u_{N-1} - u_N) + F(u_N)\;.
\end{equs}
This is a very primitive model for a polymer chain consisting of $N+1$ monomers and without self-interaction.
(In particular, our string is allowed to self-intersect...)
It does however not take into account the effect of the molecules of water that would randomly ``kick'' the 
particles that make up our string. Assuming that these kicks occur randomly and independently
at high rate, this effect can be modelled in first instance by independent white noises acting on 
all degrees of freedom of our model. We thus obtain a system of coupled stochastic differential equations:
\begin{equs}
{du_0 } &= k(u_1 - u_0)\,dt + F(u_0)\,dt + \sigma\,dw_0(t)\;,\\
{du_n} &= k(u_{n+1} + u_{n-1} - 2u_n)\,dt + F(u_n)\,dt + \sigma\,dw_n(t)\;, \quad n = 1,\ldots,N-1\;,\\
{du_N} &= k(u_{N-1} - u_N)\,dt + F(u_N)\,dt + \sigma\,dw_N(t)\;.
\end{equs}
Formally taking the continuum limit (with the scalings $k \approx \nu N^2$ and $\sigma \approx \sqrt N$), we can infer that if $N$ is very large, this system is well-described
by the solution to a stochastic \textit{partial differential equation}
\begin{equ}
du(x,t) = \nu \d_x^2 u(x,t)\, dt + F(u(x,t))\, dt + dW(x,t)\;,
\end{equ}
endowed with the boundary conditions $\d_x u(0,t) = \d_x u(1,t) = 0$. 

The reason for the choice of scaling $k \approx \nu N^2$ should be clear: if we want to obtain a continuum
limit $u$ such that $u_n(t) \approx u(n/N,t)$, then this is the choice such that
$k(u_{n+1} + u_{n-1} - 2u_n) \approx \nu \d_x^2 u$.
It is a bit less obvious \textit{a priori} what the correct scaling for $\sigma$ should be.
Recall at this stage that white noise can be viewed as a random distribution that is ``delta-correlated'',
i.e.\ setting $\xi_n = \sigma\,dw_n/dt$ one has at a formal level
\begin{equ}
\E \xi_n(s)\xi_n(t) = \sigma^2 \delta(t-s)\;,
\end{equ}
where $\delta$ denotes the ``delta function''. This is of course not an actual function but
a distribution, so that the $\xi_n$'s should really be interpreted as random distributions such that,
for any two test functions $\phi$ and $\psi$, one has
\begin{equ}
\E \xi_n(\phi)\xi_m(\psi) =  \sigma^2 \delta_{n,m} \int_{\R} \phi(t)\psi(t)\,dt\;.
\end{equ}
It is now natural to look for a scaling such that the collection of $\xi_n$'s converges to a
random space-time distribution. For this, consider now test functions $\phi$ and $\psi$ depending on
both space and time and note that the natural way of testing the $\xi_n$'s against then is given by
the distribution $\xi^{(N)}$ such that
\begin{equ}
\xi^{(N)}(\phi) = {1\over N} \sum_{n=0}^N \xi_n(\phi(n/N,\cdot))\;.
\end{equ}
In particular, its covariance is then given by
\begin{equs}
\E\xi^{(N)}(\phi)\xi^{(N)}(\psi) &= {\sigma^2\over N^2} \sum_{n=0}^N \E \xi_n(\phi(n/N,\cdot))\xi_n(\psi(n/N,\cdot)) \\
&={\sigma^2\over N^2} \sum_{n=0}^N \int_\R \phi(n/N,t)\psi(n/N,t)\,dt
\approx {\sigma^2\over N} \int_0^1 \int_\R \phi(x,t)\psi(x,t)\,dt\,dx\;.
\end{equs}
If we want this to have a non-trivial limit as $N\to\infty$, we should indeed choose $\sigma$
of order $\sqrt N$. Furthermore, we see that the limiting random distribution $\xi$ formally satisfies
\begin{equ}
\E \xi(y,s)\xi(x,t) = \delta(t-s)\delta(x-y)\;,
\end{equ}
which is called \emph{space-time white noise}.

\subsection{The stochastic Navier--Stokes equations}

The Navier--Stokes\index{stochastic!Navier--Stokes equations} equations describing the evolution of the velocity field $u(x,t)$ of an incompressible viscous fluid
are given by
\begin{equ}[e:NS]
{du \over dt} = \nu \Delta u - \bigl(u\cdot \nabla\bigr) u - \nabla p  + f\;,
\end{equ}
complemented with the (algebraic) incompressibility condition $\div u = 0$. Here, $f$ denotes some external force 
acting on the fluid, whereas the pressure $p$ is given implicitly by the requirement that $\div u = 0$ at all times.

While it is not too difficult in general to show that solutions to \eref{e:NS} exist in some weak sense,
in the case where $x \in \R^d$ with $d \ge 3$, their \textit{uniqueness for all times} is an open problem with a \$1,000,000 
prize. We will of course not attempt to solve this long-standing problem, so we are going to restrict ourselves to the case
$d = 2$. (The case $d = 1$ makes no sense since there the condition $\div u = 0$ would imply that $u$ is constant. However,
one could also consider the Burger's equation  which has similar features to the Navier--Stokes equations.)

For simplicity, we consider solutions that are periodic in space, so that we view $u$ as a function from
$\T^2 \times \R_+$ to $\R^2$. In the absence of external forcing $f$, one can use the incompressibility assumption
to see that
\begin{equ}
{d\over dt} \int_{\T^2} |u(x,t)|^2\, dx = -2\nu \int_{\T^2} \tr D u(x,t)^* D u(x,t) \, dx \le -2\nu \int_{\T^2} |u(x,t)|^2 \, dx\;,
\end{equ}
where we used the Poincar\'e inequality in the last line (assuming that $\int_{\T^2} u(x,t)\, dx = 0$).
Therefore, by Gronwall's inequality, the solutions decay to $0$ exponentially fast. This shows that energy needs to
be pumped into the system continuously if one wishes to maintain an interesting regime.

One way to achieve this from a mathematical point of view is to add a force $f$ that is randomly fluctuating.
We are going to show that if one takes a random force that is Gaussian and such that
\begin{equ}
\E f(x,t) f(y,s) = \delta(t-s) C(x-y)\;,
\end{equ}
for some correlation function $C$ then, provided that $C$ is sufficiently regular, one can show that
\eref{e:NS}  has solutions for all times. 
Furthermore, these solutions do not blow up in the sense that one can find
a constant $K$ such that, for any solution to \eref{e:NS}, one has
\begin{equ}
\limsup_{t \to \infty} \E \|u(t)\|^2 \le K\;,
\end{equ}
for some suitable norm $\|\cdot\|$. This allows to provide a construction of a model for homogeneous
turbulence which is amenable to mathematical analysis.

\subsection{The stochastic heat equation}

\index{stochastic!heat equation}
In this section, we focus on the particular example of the stochastic heat equation. We will perform
a number of calculations that give us a feeling for what the solutions to this equation look like. These
calculations will not be completely rigorous but could be made so with some extra effort. Most tools required
to make them rigorous will be introduced later in the course.

\subsubsection{Setup}

Recall that the \textit{heat equation} is the partial differential equation:
\begin{equ}[e:heat]
\d_t u = \Delta u\;,\quad u\colon \R_+ \times \R^n \to \R\;.
\end{equ}
Given any bounded continuous initial condition $u_0 \colon \R^n \to \R$, there exists
a unique solution $u$ to \eref{e:heat} which is continuous on $\R_+ \times \R^n$ and such 
that $u(0, x) = u_0(x)$ for every $x \in \R^n$.

This solution is given by the formula
\begin{equ}
u(t,x) = {1\over (4\pi t)^{n/2}} \int_{\R^n} e^{-{|x-y|^2 \over 4t}} u_0(y)\, dy\;.
\end{equ}
We will denote this by the shorthand $u(t,\cdot\,) = e^{\Delta t} u_0$ by analogy with the solution
to an $\R^d$-valued linear equation of the type $\d_t u = A u$.

Let us now go one level up in difficulty by considering \eref{e:heat} with an additional
`forcing term' $f$:
\begin{equ}[e:heatf]
\d_t u = \Delta u + f\;,\quad u\colon \R_+ \times \R^n \to \R\;.
\end{equ}
From the variations of constants formula, we obtain that the solution to \eref{e:heatf}
is given by
\begin{equ}[e:solHeat]
u(t, \cdot\,) = e^{\Delta t} u_0 + \int_0^t e^{\Delta (t-s)} f(s, \cdot\,)\, ds\;.
\end{equ}
Since the kernel defining $e^{\Delta t}$ is very smooth, this expression actually makes sense for
a large class of distributions $f$. Suppose now that $f$ is ``space-time white noise''. We do not 
define this rigorously for the moment, but characterise it as a (distribution-valued) 
centred  Gaussian process
$\xi$ such that $\E \xi(s,x) \xi(t,y) = \delta(t-s)\delta(x-y)$.

The stochastic heat equation  is then the stochastic partial differential equation
\begin{equ}[e:stochHeat]
\d_t u = \Delta u + \xi\;,\quad u\colon \R_+ \times \R^n \to \R\;.
\end{equ}
Consider the simplest case $u_0 = 0$, so that its solution is given by
\begin{equ}[e:heatWhite]
u(t,x) =  \int_0^t {1\over (4\pi |t-s|)^{n/2}} \int_{\R^n} e^{-{|x-y|^2 \over 4(t-s)}} \xi(s,y)\, dy\, ds
\end{equ}
This is again a centred Gaussian process, but its covariance function is more complicated.
The aim of this section is to get some idea about the space-time
regularity properties of \eref{e:heatWhite}. While the solutions to ordinary stochastic differential
equations are in general $\alpha$-H\"older continuous (in time) for every $\alpha < 1/2$ but not
for $\alpha = 1/2$, we will see that in dimension $n = 1$, $u$ as given by \eref{e:heatWhite} is only 
`almost' $1/4$-H\"older continuous in time and ``almost'' $1/2$-H\"older continuous in space. In higher dimensions,
it is not even function-valued...
The reason for this lower time-regularity is that the driving space-time white noise is not only very singular as a 
function of time, but also as a function of space. Therefore, some of the regularising effect of the heat equation is required to turn
it into a continuous function in space.  

Heuristically, the appearance of the H\"older exponents $1/2$ for space and $1/4$ for time in dimension $n=1$ can be understood 
by the following argument. If we were to remove the term $\d_t u$ in \eref{e:stochHeat}, then $u$ would have the same time-regularity 
as $\xi$, but two more derivatives of space regularity. If on the other hand we were to remove the term $\Delta u$, then $u$ would have the
sample space regularity as $\xi$, but one more derivative of time regularity. The consequence of keeping both terms is that we can
`trade' space-regularity against time-regularity at a cost of one time derivative for two space derivatives. Now we know that
white noise (that is the centred Gaussian process $\eta$ with $\E \eta(t)\eta(s) = \delta(t-s)$) is the time derivative of Brownian motion,
which itself is ``almost'' $1/2$-H\"older continuous. Therefore, the regularity of $\eta$ requires ``a bit more than half a derivative'' 
of improvement if we wish to obtain a continuous function.

Turning back to $\xi$, we see that it is expected to behave like $\eta$ both in the space direction and in the time direction. 
So, in order to turn it into a continuous function of time, roughly half of a time derivative is required. This leaves over half of a time
derivative, which we trade against one spatial derivative, thus concluding that for fixed time, $u$ will be almost $1/2$-H\"older continuous
in space. Concerning the time regularity, we note that half of a space derivative is required to turn $\xi$ into a continuous function of space,
thus leaving one and a half space derivative. These can be traded against $3/4$ of a time derivative, thus explaining the $1/4$-H\"older
continuity in time.

In Section~\ref{sec:regularity}, we are going to see more precisely  how the space-regularity
and the time-regularity interplay in the solutions to linear SPDEs, thus allowing us to justify rigorously this type
of heuristic arguments. For the moment, let us justify it by a calculation in the particular case of the stochastic heat equation.

\subsubsection{A formal calculation}

Define the covariance for the solution to the stochastic heat equation by
\begin{equ}[e:defCov]
C(s,t,x,y) = \E u(s,x) u(t,y)\;,
\end{equ}
where $u$ is given by \eref{e:heatWhite}. 

By (statistical) translation invariance, it is clear that
$C(s,t,x,y) = C(s,t,0,x-y)$. Using \eref{e:heatWhite} and the expression for the covariance of $\xi$, one has
\begin{equs}
C(s,t&,0,x) \\
&= {1\over (4\pi)^n} \E \int_0^t \!\!\int_0^s \!\!\int_{\R^n}\!\!\int_{\R^n} {1\over  |s-r'|^{n/2} |t-r|^{n/2}}  e^{-{|x-y|^2 \over 4(t-r)} - {|y'|^2 \over 4(s-r')}} \xi(r,y)  \xi(r',y')\, dy\, dy'\, dr' \, dr \\
&=  {1\over (4\pi)^n} \int_0^{s\wedge t} \int_{\R^n} {1\over  |s-r|^{n/2} |t-r|^{n/2}}  e^{-{|x-y|^2 \over 4(t-r)} - {|y|^2 \over 4(s-r)}}\,dy \, dr \\
& = {1\over (4\pi)^n} \int_0^{s\wedge t} \int_{\R^n} {1\over  |s-r|^{n/2} |t-r|^{n/2}} \\
&\qquad \times \exp \Bigl(-{|x|^2 \over 4(t-r)} - {\scal{x,y}\over 2(t-r)}- {|y|^2 \over 4(s-r)}- {|y|^2 \over 4(t-r)}\Bigr)\,dy \, dr \;.
\end{equs}
(Here, we used the shorthand notation $a\wedge b \eqdef \min\{a,b\}$.)
The integral over $y$ can be performed explicitly by ``completing the square'' and one obtains
\begin{equs}
C(s,t,0,x) &=  2^{-n}\int_0^{s\wedge t}  {(s+t-2r)^{-n/2}}  \exp \Bigl(-{|x|^2 \over 4(s+t-2r)}\Bigr) \, dr \\
&= 2^{-(n+1)} \int_{|s-t|}^{s+ t}  {\ell^{-n/2}}  \exp \Bigl(-{|x|^2 \over 4\ell}\Bigr) \, d\ell\;. \label{e:singSHE}
\end{equs}
We notice that the singularity at $\ell = 0$ is integrable if and only if $n < 2$, so that
$C(t,t,0,0)$ is finite only in the one-dimensional case. We therefore limit ourselves to this case in the sequel.

\begin{remark}
Even though the random variable $u$ defined by \eref{e:heatWhite} is not function-valued in dimension $2$, it is ``almost'' the case
since the singularity in \eref{e:singSHE} diverges only logarithmically. The stationary solution to \eref{e:stochHeat} is
called the \textit{Gaussian free field} and has been the object of intense studies over the last few years, especially in dimension $2$. 
Its interest stems from the
fact that many of its features are conformally invariant (as a consequence of the conformal invariance of the Laplacian), 
thus linking probability theory to quantum field theory on one hand and
to complex geometry on the other hand. The Gaussian free field also relates directly to the Schramm--Loewner evolutions (SLEs) for the study of which Werner was awarded the Fields medal in 2006, see \cite{IntroSLE,SS06}. For more information on the Gaussian
free field, see for example the review article by Sheffield \cite{GFF}.
\end{remark}

The regularity of $u$ is determined by the behaviour of
$C$ near the ``diagonal'' $s = t$, $x = y$. We first consider the time regularity. We therefore set $x = 0$ and compute
\begin{equ}
C(s,t,0,0) =  {1\over 4}\int_{|s-t|}^{s+ t}  {\ell^{-1/2}} \, d\ell = \hf \bigl(|s+t|^{{1\over 2}}- |s-t|^{{1\over 2}}\bigr)\;.
\end{equ}
This shows that, in the case $n = 1$ and for $s \approx t$, one has the asymptotic behaviour
\begin{equ}
\E |u(s,0) - u(t,0)|^2 \approx |t-s|^{{1\over 2}}\;.
\end{equ}
Comparing this with the standard Brownian motion  for which $\E |B(s) - B(t)|^2 = |t-s|$, we conclude that
the process $t \mapsto u(t,x)$ is, for fixed $x$, almost surely $\alpha$-H\"older continuous for any exponent
$\alpha < 1/4$ (but actually \textit{not} for $\alpha = 1/4$). This argument is a special case of Kolmogorov's celebrated continuity test,
of which we will see a version adapted to Gaussian measures in Section~\ref{sec:boundsGauss}.

If, on the other hand, we fix $s=t$, we obtain (still in the case $n = 1$) via the change of variables $z = |x|^2 / 4\ell$,
the expression
\begin{equ}
C(t,t,0,x) =  { |x| \over 8} \int_{|x|^2 \over 8t}^{\infty}  {z^{-{3\over 2}}}  e^{-z} \, dz\;.
\end{equ} 
Integrating by parts, we get
\begin{equ}
C(t,t,0,x) =  {\sqrt{t}\over 4} e^{- {|x|^2 \over 8t}} + {|x|\over 4} \int_{|x|^2 \over 8t}^{\infty}  {z^{-{1\over 2}}}  e^{-z} \, dz\;,
\end{equ}
So that to leading order we have for small values of $x$: 
\begin{equ}
C(t,t,0,x) \approx {\sqrt{t}\over 4} + {|x|\over 4} \int_0^\infty z^{-{1\over 2}} e^{-z}\, dz = \sqrt t + {\sqrt \pi  |x| \over 4}+ \CO \bigl(|x|^2 / 8\sqrt t\bigr)\;.
\end{equ}
This shows that, at any fixed instant $t$, the solution to \eref{e:stochHeat} looks like a Brownian motion in space
over lengthscales of order $t^{1/2}$. Note that over such a lengthscale the Brownian motion fluctuates by about $t^{1/4}$, which is 
precisely the order of magnitude of $\E |u(t,x)|$.

\subsection{What have we learned?}

\begin{enumerate}
\item At a ``hand-waving'' level, we have forced our equation with a term that has a temporal
evolution resembling white noise, so that one would  naively expect its solutions to have a temporal regularity resembling
Brownian motion.
However, for any fixed location in space, the solution to the stochastic heat equation has a time-regularity which is
only almost ${1\over4}$-H\"older continuous, as opposed to the almost ${1\over 2}$-H\"older 
time-regularity of Brownian motion.
\item Unlike the solutions to an ordinary parabolic PDE, the solutions to a stochastic PDE tend to be spatially ``rough''. It is therefore
not obvious \textit{a priori} how the formal expression that we obtained is to be related to the original equation \eref{e:stochHeat},
since even for positive times, the map $x \mapsto u(t,x)$ is certainly not twice differentiable.
\item Even though the deterministic heat equation has the property that $e^{\Delta t} u \to 0$ as $t \to \infty$ for every $u \in L^2$, the solution
to the stochastic heat equation has the property that $\E |u(x,t)|^2 \to \infty$ for fixed $x$ as $t \to \infty$. This shows that 
in this particular case, the stochastic forcing term pumps energy into the system faster than the deterministic evolution
can dissipate it.
\end{enumerate}

\begin{exercise}
Perform the same calculation, but for the equation
\begin{equ}
\d_t u = \Delta u - a u + \xi\;,\quad u\colon \R_+ \times \R \to \R\;.
\end{equ}
Show that as $t \to \infty$, the law of its solution converges to the law of an Ornstein-Uhlenbeck process (if the space
variable is viewed as ``time''):
\begin{equ}
\lim_{t \to\infty} \E u(t,x)u(t,y) = C e^{-c|x-y|}\;.
\end{equ}
Compute the constants $C$ and $c$ as functions of the parameter $a$.
\textbf{Hint:} Show first that, similarly to \eqref{e:singSHE}, the above limit is given by
\begin{equ}
G(|x-y|) \eqdef \int_0^\infty e^{-\frac{|x-y|^2}{4r} - ar} \frac{dr}{4\sqrt r}\;. 
\end{equ}
Show then that, away from the origin, $G$ satisfies an ODE of the form $G' = \pm c G$. 
For this, a change of variables of the form $s = k/r$ for a suitable choice of $k$ might come in handy.
\end{exercise}

One much more robust way of deriving the exponents $1/2$ and $1/4$ appearing above is given 
by the following exercises.
Here, we take it for granted that the law of $\xi$ is characterised by the fact that,
for any two test functions $\phi$ and $\psi$, one has
\begin{equ}
\E \xi(\phi)\xi(\psi) = \scal{\phi,\psi} \eqdef \int_{\R\times \R^d} \phi(t,x)\psi(t,x)\,dt\,dx\;.
\end{equ}
(This will be justified in Section~\ref{sec:Gaussian}.)

\begin{exercise}\label{ex:scale}
Given an exponent $\alpha \in \R$ a scaling factor $\lambda > 0$, $h = (h_t, h_x) \in \R^{d+1}$, and a function $f \colon \R^{d+1} \to \R$, set
\begin{equ}[e:scaling]
(S_\lambda^\alpha f)(t,x) = \lambda^{-\alpha}f(\lambda^2 t,\lambda x)\;,\qquad
(T_h f)(t,x) = f(t - h_t,x - h_x)\;. 
\end{equ}
This is extended to distributions on $\R^{d+1}$ by
\begin{equ}
\bigl(S_\lambda^\alpha \xi\bigr)(\phi) = \xi\big(S_{1/\lambda}^{-d-2-\alpha} \phi\big)\;,\qquad
\bigl(T_h \xi\bigr)(\phi) = \xi\big(T_{-h} \phi\big)\;.
\end{equ}
Show that this definition is consistent with the canonical embedding of smooth functions into
the space of distributions. 

Let now $\xi$ denote space-time white noise. Show that, for any $\lambda$ and $h$ as above, and for
$\alpha = -\frac{d+2}{2}$, the laws of $S_\lambda^\alpha \xi$ and of $T_h \xi$ equal that of $\xi$.
We say that $\xi$ is \textit{stationary} and \textit{scale invariant} with scaling exponent $\alpha$.
Finally, show that if $u$ is a random function / distribution that is stationary and scale invariant with
exponent $\beta$, then $(\d_t - \Delta)u$ is stationary and scale invariant with exponent $\beta - 2$.
\end{exercise}

\begin{exercise}\label{ex:Holder}
Let $\alpha \in (0,1)$ and, given a continuous 
function $f\colon \R \to \R$, define $S_\lambda^\alpha f$ as above,
but ignoring the $t$ variable. Show that $f$ is uniformly $\alpha$-Hölder continuous\footnote{i.e.\ there exists $K$ such that $|f(x) - f(y)| \le K |x-y|^\alpha$ uniformly over all pairs $x,y$ with $|x-y| \le 1$.} if and only if
the quantity $\scal{S_\lambda^\alpha T_h f,\phi}$ is bounded, uniformly over all $h \in \R$,
all $\lambda \in (0,1]$, and all functions $\phi$ that are bounded by $1$, are supported in the 
centred ball of radius $1$, and are such that $\int \phi(x)\,dx = 0$. 
\textbf{Hint:} Consider successive approximations of $f(z)$ (with $z \in \{x,y\}$)
by testing $f$ against a suitable multiple of the indicator
function of an interval of length $2^{-n}|x-y|$ centred around $z$.

Show that the same statement also holds for functions on $\R^{d+1}$ with $S_\lambda^\alpha$
as in \eqref{e:scaling}, provided that Hölder continuity is measured with respect to the distance function
$d((t,x), (t',x')) = \sqrt{|t-t'|} + |x-x'|$. Note then that $\alpha$-Hölder continuity with 
respect to $d$ implies $\alpha$-Hölder continuity in the $x$ variable and $\alpha/2$-Hölder 
continuity in the $t$ variable.  
\end{exercise}

Exercise~\ref{ex:scale} makes it at least plausible that the solution to the stochastic heat equation
is translation and scale invariant with exponent $2 -  \frac{d+2}{2} = \frac{2-d}{2}$, which is 
$1/2$ in the special case $d=1$. Exercise~\ref{ex:Holder} then suggests that this implies the
claimed Hölder continuity. The reason why one actually looses a little bit is that  
translation and scale invariance \textit{in law} does not actually imply the uniform boundedness
required in Exercise~\ref{ex:Holder}.

\begin{remark}
It is in fact note quite true that $u$ is translation and scale invariant in the above sense,
but it is the case if one quotients it by constant functions. Since the above criterion for
Hölder regularity only requires test functions that annihilate constants, this does not affect
the argument.
\end{remark}

\section{Probability Measures on Polish Spaces}

While most of the probability theory appearing in these notes is constructed from Gaussian measures
(but on infinite-dimensional spaces!), we start with a foundational section on the
general theory of probability measures on Polish (that is complete, separable, metrisable) 
spaces.\index{Polish space}

The typical example of a Polish space one should have in mind is that of a separable Banach (or Hilbert) space,
like for example the Lebesgue spaces $L^p(\R^n)$ (with $n \neq \infty$) or the Sobolev spaces $H^s(\R^n)$.
More interestingly, we will see that, when equipped with the topology of weak convergence, the 
space of probability measures on a Polish space is itself Polish! An example of a ``large'' Polish
space which is very far from linear is the Gromov--Hausdorff space of 
\textit{all} isometry classes of compact metric spaces. 

This theory is of course much too vast to be done any sort of justice in these few short pages. We therefore refer the 
interested reader to the excellent and quite extensive treatise by Bogachev \cite{BogMT} and to the much shorter but maybe more
readily accessible book by Billingsley \cite{Bill68} which still covers a large part of the material required for the basic study of stochastic PDEs. 

\subsection{Convergence of probability measures}

The focus of this section will be mainly on the question of convergence of probability measures on Polish spaces. We will introduce
a number of different topologies on the space of probability measures on an arbitrary Polish space and we will discuss the relations
between these topologies and the metrics that generate them.

Recall that the Borel $\sigma$-algebra $\B(\CX)$ of a Polish space $\CX$ is the smallest collection of sets containing all the
open sets of $\CX$ that is furthermore closed under countable unions  and taking complements. A probability measure
$\mu$ on $\CX$ is then a map $\mu \colon \B(\CX) \to [0,1]$ such that $\mu(\CX) = 1$ and $\mu(\bigcup A_i) = \sum \mu(A_i)$
for any countable collection of mutually disjoint sets $A_i \in \B(\CX)$. We denote by $\cP(\CX)$ the set of all (Borel) probability
measures on $\CX$. In order of increasing strength, here are the two main notions of convergence that will be used in these notes:
\begin{claim}
\item \textbf{Weak convergence:}\index{weak convergence} A sequence of probability measures $\mu_n$ converges weakly to a 
limiting probability measure $\mu$ if  
$\lim_{n \to \infty}\int \phi(x)\,\mu_n(dx) = \int \phi(x)\,\mu(dx)$ for every bounded continuous function $\phi \colon \CX \to \R$.
\item \textbf{Total variation convergence:}\index{total variation distance} This notion of convergence is defined by the total variation metric given by 
\begin{equ}[e:defTV]
\|\mu- \nu\|_\TV = \sup_{\|\phi\|_\infty \le 1} \Bigl|\int \phi(x) \,\mu(dx) - \int \phi(x)\, \nu(dx)\Bigr|\;.
\end{equ}
Here, $\|\phi\|_\infty$ denotes the supremum norm of $\phi$.
\end{claim}

\begin{exercise}
Show that the ``total variation metric'' defined above is indeed a metric.
\end{exercise}

\begin{remark}
Another widespread notion of convergence is that of \textit{strong convergence}: a sequence $\mu_n$ converges strongly to a limiting measure $\mu$ if  
$\lim_{n \to \infty}\mu_n(A) = \mu(A)$ for every  $A \in \B(\CX)$. However, this notion of convergence
does have somewhat more pathological properties. For example, it is possible to find several \textit{non-equivalent} topologies on the space of signed measures on
$\CX$ giving rise to this notion of convergence for countable sequences, see \cite[Section~4.7(v)]{BogMT}.

Furthermore, these topologies are not ``nice''. For example, if we consider the set $\cM(\CX)$ of finite signed measures on $\CX$ endowed with the
total variation norm, then this is a Banach space (call it $\cM_\TV$). 
It is then a very unfortunate fact that one of the topologies on $\cM(\CX)$ giving 
rise to the notion of strong convergence 
for sequences is actually the \textit{weak} (in the usual sense of functional analysis) topology on $\cM_\TV$.
Since, just like $\cM_\TV$ itself, the dual of $\cM_\TV$ is not separable (unless $\CX$ is finite of course), 
it follows from general principles (see for example \cite{FabHab:01,Zizler}) that this topology is not 
metrisable (not even on bounded sets), which greatly limits its use in practice.
\end{remark}

\begin{remark}
If $\CX$ is a compact metric space, then the Riesz--Markov theorem \cite{Rudin} tells us that the dual of $\CC_b(\CX)$, the space of bounded continuous functions,
is precisely given by the space $\cM(\CX)$ of finite signed measures on $\CX$. Furthermore, in this language, the topology of weak convergence 
is nothing but the weak-* topology on $\cM(\CX)$, viewed as the dual of $\CC_b(\CX)$. 
Since the dual of an infinite-dimensional Banach space is never metrisable for the weak-* topology \cite{FabHab:01,Zizler}, one 
may think then that the notion of weak convergence for probability measure suffers from the same problems
as those pointed out in the previous remark for the strong convergence. Fortunately, it turns out that \textit{bounded subsets}
of the dual of a separable Banach space are weak-* metrisable, which is sufficient for our purpose since we are mostly interested
in probability measures. 
\end{remark}

We will sometimes use slightly different notions of convergence, but they will be only minor variations on the general themes given here.
The above notions of convergence give rise to two non-equivalent metrisable topologies on $\cP(\CX)$. 
Metrisability is obvious for the
notion of total variation convergence, but much less so for the notion of weak convergence. Before we turn to the construction of metrics
for weak convergence, let us give a few classical examples illustrating the differences between them.

\begin{example}
Let $\CX = \R$ and let $\mu_n$ be the Dirac measure located at ${1/n}$. Then it converges weakly, but neither strongly nor in total variation 
to the Dirac measure located at the origin.
\end{example}

\begin{example}
Let $\mu_n$ be the measure on $[0,\pi]$ given by $\mu_n(dx) = {2\over \pi} \sin^2(nx)\,dx$. Then, $\mu_n$ converges to the normalised Lebesgue
measure both in the weak and the strong sense. However, the total variation distance between $\mu_n$ and its limit  ${dx\over \pi}$
is  equal to ${2\over \pi}$ for every $n$.
\end{example}

Note that the theory of probability measures on arbitrary metric (or metrisable) spaces
is much more pathological. This can be seen by the following example which shows in particular that
the rule of thumb that ``every map you can define unambiguously is measurable'' can be broken in that case.

\begin{exercise}
Let $H\colon \R\to \{0,1\}$ denote the Heaviside function and write $H_t(x) = H(x-t)$. 
Show that the map $\Theta \colon \R \to L^\infty(\R)$ given by $t \mapsto H_t$ is
not even Borel measurable. (You may take it for granted that there do exist subsets of the real line
that are not Borel.) \textbf{Hint:} Use the fact that an arbitrary union of open sets is open to show
that given \textit{any} subset $A \subset \R$ there exists an open set $A^* \subset L^\infty$ 
with $\Theta^{-1}(A^*) = A$.
\end{exercise}

\begin{exercise}\label{ex:HB}
Show that for subsets $K\subset \CX$ with $\CX$ a Polish space, the following
four characterisations of compactness are equivalent:
\begin{enumerate}
\item Every open cover of $K$ admits a finite subcover. ($K$ is compact.)
\item $K$ is closed and, given any metric $d$ generating the topology of $\CX$,
for every $\eps > 0$, $K$ can be covered with finitely many balls of radius $\eps$. 
($K$ is closed and totally bounded.)
\item Every sequence $(x_n)$ of elements of $K$ admits a subsequence that converges
to a limit $x \in K$. ($K$ is sequentially compact.)
\end{enumerate}
\textbf{Hint:} The hard part is to show that 3 implies 1. For this, show first that if 3 holds and
we have an open cover $\{V_\alpha\}$ of $K$, then we can find a countable subcover. This can be seen
by fixing some dense countable subset $\{y_n\}_{n \in \N}$ of $K$ and, for each $m,n \in \N$,
choose exactly one $V_\alpha$ containing the open ball of radius $1/m$ around $y_n$ (if such a $V_\alpha$ 
exists). Writing $V_n$ the countable cover of $K$ obtained in this way, we define a decreasing family of 
closed sets by $F_n = K \setminus \bigcup_{j \le n} V_j$. Show that either one of the $F_n$'s must be empty
(thus showing that we have a finite cover of $K$), or $\bigcap_{n > 0} F_n \neq \emptyset$,
thus contradicting the fact that the $V_n$'s are a cover of $K$. 
\end{exercise}

\subsection{Total variation convergence}

\begin{wrapfigure}{r}{4.5cm}
\vspace{-7mm}
\begin{center}
\mhpastefig{TV}
\end{center}
\vspace{-8mm}
\end{wrapfigure}
The total variation distance between two probability measures $\mu$ and $\nu$
is relatively straightforward to comprehend: it consists of the total amount of mass that doesn't overlap between $\mu$ and $\nu$.
Since a picture is worth a thousand words, we illustrate this by
 the figure shown on the right: the total variation distance between $\mu$ and $\nu$ is given by the dark gray area.

If $\mu$ and $\nu$ have densities $\CD_\mu$ and $\CD_\nu$ with respect to some
common positive reference measure $\zeta$ (by the Radon--Nikodym theorem,
it is always possible to take $\zeta = {1\over 2}(\mu + \nu)$ for example), then one has
the identity
\begin{equ}[e:defTV2]
\|\mu-\nu\|_{\TV} = \int_\CX |\CD_\mu(x) - \CD_\nu(x)|\, \zeta(dx)\;,
\end{equ}
which is also sometimes taken as the definition of the total variation distance.

\begin{exercise}
Show that the characterisation \eref{e:defTV2} of the total variation distance does not depend on the particular choice of a reference measure $\zeta$
and that it does indeed agree with the definition previously given in \eref{e:defTV}. \textbf{Hint:} Consider the test function $\phi(x)$ such that
$\phi(x) = 1$ if $\CD_\mu(x) > \CD_\nu(x)$ and $\phi(x) = -1$ otherwise. 
\end{exercise}

As a consequence of the characterisation \eref{e:defTV2}, we have the following very important fact:

\begin{corollary}
For $\mu$ and $\nu$ two probability measures, one has $\|\mu - \nu\|_\TV = 2$ if and only if $\mu$ and $\nu$ are mutually singular
and $\|\mu - \nu\|_\TV < 2$ otherwise.
\end{corollary}

\begin{proof}
Let $\mu$ and $\nu$ be mutually singular, denoted by $\mu \perp \nu$.\index{mutually singular measures} Then there exists a set $A$ such that $\mu(A) = 1$ and $\nu(A) = 0$. Setting $\phi(x) = 2\,\one_A(x) - 1$, it follows
from \eref{e:defTV} that $\|\mu-\nu\|_\TV \ge 2$. On the other hand, one has $\|\mu-\nu\|_\TV \le 2$ as a consequence of the
definition, so that the first implication follows. 

To show the converse,
assume that $\mu$ and $\nu$ are not mutually singular and denote by $\CD_\mu$ and $\CD_\nu$ their densities with respect to a common reference measure $\zeta$.
Let $A_\mu = \{x \,:\, \CD_\mu(x) > 0\}$ and similarly for $A_\nu$, and set $A = A_\mu \cap A_\nu$. With this notation, one must have $\zeta(A) > 0$, for otherwise
$\mu \perp \nu$. Since for two positive numbers $a$ and $b$ we have the identity $|a-b| = a+b - 2(a\wedge b)$, it follows from \eref{e:defTV2} that
\begin{equ}
\|\mu-\nu\|_{\TV} = 2 - 2\int_A \bigl(\CD_\mu(x) \wedge \CD_\nu(x)\bigr)\, \zeta(dx)\;.
\end{equ}
Since $\zeta(A) > 0$ and $\CD_\mu(x) \wedge \CD_\nu(x) > 0$ for $x \in A$ by definition, the claim follows.
\end{proof}

There is a third very useful (and more probabilistically appealing) interpretation of the total variation distance between two probability measures. 
Indeed, the total variation
distance between two probability measures $\mu$ and $\nu$ on a Polish space $\CX$ is given by
\begin{equ}[e:TVCoupl]
\|\mu-\nu\|_{\TV} = 2\inf_{\pi \in \cC(\mu,\nu)} \pi(\{x \neq y\})\;,
\end{equ}
where the infimum runs over the set $\cC(\mu,\nu)$ of all probability measures $\pi$ on $\CX \times \CX$ with marginals $\mu$ and $\nu$. 
(This set is also called the set of all \textit{\index{coupling}{couplings}} of $\mu$ and $\nu$.)
In other words, if the total variation distance between $\mu$ and $\nu$ is smaller than $2\eps$, then it is possible to construct
$\CX$-valued random variables $X$ and $Y$ with respective laws $\mu$ and $\nu$ such that $X = Y$ with probability
$1-\eps$. This gives a straightforward probabilistic interpretation of the total variation distance as twice the probability that a random sample drawn from
$\mu$ can be distinguished from a sample drawn at random from $\nu$.

\begin{exercise}
Show that the identity \eref{e:TVCoupl} holds and that the infimum is attained. \textbf{Hint:} The optimal coupling can be constructed explicitly by considering a combination
of the measure $\bigl(\CD_\mu(x) \wedge \CD_\nu(x)\bigr)\,\zeta(dx)$ on the diagonal $(x,x)$ and the measure $\bigl(\CD_\mu(x) - \CD_\nu(x)\bigr)_+\bigl(\CD_\nu(y) - \CD_\mu(y)\bigr)_+\,\zeta(dx)\,\zeta(dy)$ off the diagonal.
\end{exercise}

\begin{remark}
Some authors define the total variation distance between measures 
as the expression \eref{e:TVCoupl}, but without the factor $2$. Being aware of this helps to navigate
a literature that could otherwise cause some confusion.
\end{remark}

\begin{exercise}\label{ex:TVWeight}
Given a function $V \colon \CX \to \R_+$, we can also define a weighted total variation distance by
\begin{equ}
\|\mu-\nu\|_{\TV,V} = \int_\CX \bigl(1+V(x)\bigr)|\CD_\mu(x) - \CD_\nu(x)|\, \zeta(dx)\;.
\end{equ}
By \eref{e:defTV2}, we recover the usual total variation distance as a special case when $V = 0$. This distance is only defined
on the subspace of finite signed measures that integrate $V$ and turns this subspace into a Banach space. Show that one does have, similarly to
\eref{e:TVCoupl}, the characterisation
\begin{equ}[e:TVCouplV]
\|\mu-\nu\|_{\TV,V} = \inf_{\pi \in \cC(\mu,\nu)} \int_{x \neq y} \bigl(2 + V(x) + V(y)\bigr)\,\pi(dx,dy)\;.
\end{equ} 
\end{exercise}

\subsection{Weak convergence}

It is obvious that if $\CX$ is uncountable, then $\cP(\CX)$ endowed with the total variation metric is not 
itself a separable space. Indeed, the collection
$\{\delta_x\}_{x\in \CX}$ yields an uncountable set of elements that are all at distance $2$ of each other. On the other hand, if we endow
$\cP(\CX)$ with the topology of weak convergence, then it turns out to be separable. Even better, it is actually itself a Polish space and one can 
construct a number of natural distance functions that generate its topology. In this section, we collect 
a few important results about
the properties of the weak convergence topology. For a more detailed account, including complete proofs, we refer for example
to \cite{Bill68,BogMT,Vil03}.

If one had to choose one, the single most import result in the theory of weak convergence of probability measures would probably be 
Prohorov's characterisation of those subsets of the set of probability measures that are
precompact for the weak convergence topology. Before we state this theorem, we introduce the 
concept of a \textit{tight} family of probability measures, which will is a fundamental concept in the theory
of weak convergence:

\begin{definition}
Given a collection $M \subset \cP(\CX)$ of probability measures on a Polish space $\CX$,
we say that $M$ is \textit{tight}\index{tightness} if, for every $\eps > 0$, there exists a compact set $K\subset \CX$ such that $\mu(K) \ge 1-\eps$ for every $\mu \in M$. 
\end{definition}

In other words, $M$ is tight if its elements are uniformly concentrated on compact sets. It turns out that sets comprising of a single measure (or finitely many measures) are always tight.
While this is obvious if the space $\CX$ can be covered by a countable collection of compact sets
(like it is the case for $\R^n$ for example), it is not so obvious if $\CX$ is an infinite-dimensional space.
Using the Heine--Borel theorem, it is however not to difficult to prove it, and this is the content of the next lemma:

\begin{lemma}
Let $\mu \in \cP(\CX)$ for a Polish space $\CX$. Then the singleton $\{\mu\}$ is tight.
\end{lemma}

\begin{proof}
Fix $\eps > 0$. Since $\CX$ is separable it can be covered by countably many balls of fixed, but arbitrary, radius.
Therefore, for every $n > 0$, one can find a set $K_n$ consisting of finitely many 
closed balls of radius $1/n$ and such that $\mu(K_n) > 1-2^{-n}\eps$. Setting
$K = \bigcap_{n > 0} K_n$, it follows that $\mu(K) \ge 1-\eps$, which concludes the proof since
$K$ is closed and \index{totally bounded sets}totally bounded, and therefore compact by 
Exercise~\ref{ex:HB}.
\end{proof}

Another interesting fact is that tightness follows from the following property which may appear
weaker at first sight:

\begin{lemma}\label{lem:asymTight}
Let $M \subset \cP(\CX)$ for $\CX$ a Polish space width metric $d$. Assume that, for every $\eps > 0$, there exists
$K \subset \CX$ compact such that, for every $\mu \in M$, $\mu(K^\eps) \ge 1-\eps$, where $K^\eps$ is the $\eps$-fattening
of $K$: $K^\eps = \{x \in \CX\,:\, d(x,K) \le \eps\}$. Then $M$ is tight.
\end{lemma}

\begin{proof}
The proof works as before since, for every $n$, we can find $K_n$ such that $\mu(K_n^{\eps 2^{-n}}) \ge 1- \eps 2^{-n}$
by assumption, and then set $K = \bigcap_{n > 0} K_n^{\eps 2^{-n}}$ similarly to before.
To show that $K$ is totally bounded, given $\delta > 0$ arbitrary, take $n$ large enough so that
 $\eps 2^{-n} < \delta / 2$. Since $K_n$ is compact, it can be covered by finitely many balls of 
radius $\delta/ 2$ and, since $\eps 2^{-n} < \delta / 2$, the balls of radius $\delta$ and same radii
do cover $K_n^{\eps 2^{-n}}$, and therefore $K$. 
\end{proof}

So why is tightness so important? The following theorem due to Prohorov \cite{Prohorov} shows that
tight families of probability measure coincide with precompact subsets of $\cP(\CX)$ in the topology
of weak convergence:

\begin{theorem}[Prohorov]\index{Prohorov's theorem}
A subset $M \subset \cP(\CX)$ is precompact for the topology of weak convergence if and only if it is tight.
\end{theorem}

Before we give the proof of this theorem, let us prove the following particular case:

\begin{lemma}\label{lem:compactMeas}
If $\CX$ is compact, then both $\cP(\CX)$ and the unit ball in $\cM(\CX)$ are compact for the 
topology of weak convergence.
\end{lemma}

\begin{proof}
Take any sequence $\mu_n$ of uniformly bounded measures on $\CX$.
Since $\CX$ is compact, the space $\CC(\CX)$ is separable, so we can find a dense countable subset $\phi_n$
of the unit ball in $\CC(\CX)$.
A simple diagonal extraction argument, combined with the fact that there exists $C$ such that
$\mu_n(\phi_m) \le C$ by assumption, allows to extract a subsequence $\mu_{n_k}$ such that
$\mu_{n_k}(\phi_n) \to c_n$ for every $n$. By density of the $\phi_n$ and boundedness of the $\mu_n$,
it follows that there exists a continuous linear functional $\mu$ on $\CC(\CX)$ such that $\mu_{n_k}(\phi) \to \mu(\phi)$
for every $\phi \in \CC(\CX)$. The conclusion then follows from the Riesz--Markov theorem that identifies
the dual of $\CC(\CX)$ with $\cM(\CX)$.

Note that in particular the set of probability measures is compact since, by testing against the constant function $1$
and positive functions, we conclude that the limit for any converging subsequence is again a probability measure.
\end{proof}

\begin{remark}
The above proof is nothing but a special case of the Banach--Alaoglu theorem \cite[Thm~3.15]{Rudin},
but its proof is sufficiently short and elementary so that we reproduced it here. It works for the unit ball 
of the dual space of any separable Banach space, endowed with the weak-* topology.
\end{remark}

We are now ready to give the proof of Prohorov's theorem, which follows rather closely the exposition
given in \cite[Theorem~8.6.2]{BogMT}.
The original proof can be found in  \cite{Prohorov}, but see also \cite{Bill68} for a clean proof in the
special case $\CX = \R$.

\begin{proof}[Proof of Prohorov's theorem]
We first show that tightness is sufficient by extracting a weakly convergent subsequence from $M$
under the assumption that $M$ is tight. By assumption, we can find an increasing sequence of compact sets $K_n \subset \CX$
such that $\mu(\CX \setminus K_n) \ge  1 - 2^{-n}$ for every $\mu \in M$. Using Lemma~\ref{lem:compactMeas} and 
a diagonal extraction argument, we can find a sequence $\{\mu_n\}$ in $M$ such that, for every $m > 0$,
the restricted sequence $\{\mu_n \restr K_m\}$ converges to some element $\hat \mu_m$, which is 
a positive measure on $K_m$ with $|\hat\mu_m(K_m) - 1| \le 2^{-m}$. 
This is furthermore an increasing and bounded sequence, which therefore has a limit $\mu$ with $\mu|_{K_m} = \hat\mu_m$. 
We conclude that,  for every continuous function $\phi$ bounded by $1$ and every $m \ge 1$, 
we have
\begin{equ}
|\mu_n(\phi) - \mu(\phi)| \le  \bigl|\bigl(\mu_n\restr K_m\bigr)(\phi) - \hat \mu_m(\phi)\bigr| + 
\bigl|\bigl(\mu_n\restr K_m\bigr)(\phi) - \mu_n(\phi)\bigr| + \bigl|\hat \mu_m(\phi) - \mu(\phi)\bigr|\;.
\end{equ}
The first term converges to $0$ and the other two terms are each bounded by $2^{-m}$. Since $m$ was arbitrary,
this shows that $\{\mu_n\} \to \mu$ as required.

We now show the converse statement, namely that if $M$ is not tight, then it cannot be precompact. 
Assuming that $M$ is not tight, we use the contrapositive of Lemma~\ref{lem:asymTight} to conclude that 
there exists a fixed $\eps > 0$ such that, for every compact $K\subset \CX$, there is an element $\mu_K \in M$
such that $\mu_K(\CX \setminus K^\eps) \ge \eps$. We now fix a value $\delta > 0$ (think of $\delta$
as being much smaller than the $\eps$ that we just found),
and we construct a sequence of measures $\mu_n \in M$
and two sequences of compact sets $(A_n,K_n)$ recursively in the following way:
\begin{claim}
\item Choose for $\mu_0$ any element of $M$ (it has to contain infinitely many elements since it is not tight),
choose $K_0$ such that $\mu_0(K_0) \ge 1- \delta$, and set $A_0 = K_0$.
\item Given data up to the $n$th index, choose
$\mu_{n+1}\in M$ such that $\mu_{n+1}(\CX \setminus A_n^\eps) \ge \eps$, which is possible by the lack of tightness of $M$. 
Then, choose a compact set $K \subset \CX$
such that $\mu_{n+1}(K) \ge 1-\delta$ and set $A_{n+1} = A_n \cup K$ and $K_{n+1} = K \setminus A_{n}^\eps$.
\end{claim}
Actually, the only properties of this construction that we are going to use are that  
$\mu_n(K_n) \ge \eps - \delta$, $\mu_n(\bigcup_{k > n}K_k) \le \delta$, and $K_n^{\eps/2} \cap K_m^{\eps / 2} = \emptyset$
for every $n \neq m$.

Our aim is to show that the sequence $\mu_n$ constructed in this way contains no convergent subsequence. 
If $\bar \mu_n = \mu_{k_n}$ is an arbitrary subsequence then, by setting
$\bar K_n =  K_{k_n}$ and $\bar A_n = A_{k_n}$, the sequence $(\bar \mu_n, \bar A_n, \bar K_n)$ 
also has the properties mentioned in the last paragraph,
so that it suffices to show that any sequence $\{\mu_n\}$ with these properties cannot be convergent.

We do this by exhibiting a continuous test function $\phi$ such that $\mu_n(\phi)$ does not converge.
Define first continuous functions $\phi_n$ by $\phi_n(x) = 1$ if $x\in K_n$, $\phi_n(x) = 0$ if $x \not \in K_n^{\eps/2}$,
and $\phi_n(x) = 1 - 2\eps^{-1} d(x,K_n)$ otherwise. Note that since these functions all have disjoint supports and are all Lipschitz continuous
with the same Lipschitz constant, the function 
\begin{equ}
\phi_\lambda(x) \eqdef \sum_{n>0} \lambda_n \phi_n(x)\;,
\end{equ}
is continuous and bounded for every bounded sequence $\lambda$. Given the sequence $\mu_n$, we 
now construct in a recursive way
a sequence $\lambda$ with $|\lambda_n| \le 1$ for every $n$ and such that $|\mu_n(\phi_\lambda) - \mu_{n+1}(\phi_\lambda)| \ge c$
for some fixed $c > 0$ and for every $n$.

Choose first $\lambda_0 = 0$, say. For arbitrary $n \ge 0$,
once $\lambda_0,\ldots,\lambda_n$ are given,
it follows from the property $\mu_n(\bigcup_{k>n}K_k) \le \delta$ that
 $\mu_n(\phi_\lambda)$ is determined to within an error of at most $\delta$ by $\lambda_0,\ldots,\lambda_n$.
On the other hand, we have $\mu_{n+1}(K_{n+1}) \ge \eps - \delta$ so that, by adjusting $\lambda_{n+1} \in [-1,1]$, we can
cover a range of values of width at least $2(\eps-\delta)$ for $\mu_{n+1}(\phi_\lambda)$.
This guarantees that we can find $\lambda_{n+1}$ in such a way that $|\mu_{n+1}(\phi_{\bar\lambda}) - \mu_n(\phi_{\bar\lambda})| \ge \eps - 3\delta$
for \textit{every} sequence $\bar \lambda$ such that $\bar \lambda_k = \lambda_k$ for $k \le n+1$. 
Since $\eps$ was fixed but $\delta$ was arbitrary in this construction, the claim follows by choosing $\delta$ sufficiently small.
\end{proof}

\subsection{Wasserstein distance}

Now that we have some understanding how compact sets look like in $\cP(\CX)$, we turn to the construction of a family of metrics that generate 
this topology.
Given any bounded lower semicontinuous metric $d$ on $\CX$ (note that $d$ does not necessarily need to generate the topology of $\CX$!),
we can ``lift'' it to the space of probability measures on $\CX$ in a natural way by setting:
\begin{equ}[e:defW1]
d(\mu,\nu) = \inf_{\pi \in \cC(\mu,\nu)} \int_{\CX}\int_{\CX}d(x,y)\,\pi(dx,dy)\;.
\end{equ}
This distance is called the \textit{$1$-Wasserstein distance}\footnote{This is really a misnomer since these distances were introduced by Kantorovich and the special case $p=1$ was already studied by Monge. However, the name ``Wasserstein distance'' is now being used in most of the literature on the subject so we'll stick with it.}\index{Wasserstein distance} for $d$ on $\cP(\CX)$. (The $p$-Wasserstein distances
can be defined similarly for every $p \ge 1$ by setting their $p$th power equal to the right hand side of \eref{e:defW1} with $d$ replaced by $d^p$.)
The reason why we assumed that $d$ is lower semicontinuous is the following:

\begin{exercise}
Show that the infimum in \eref{e:defW1} is achieved. \textbf{Hint:} Use the fact that single measures are tight to conclude that the set
$\cC(\mu,\nu)$ is compact for any two probability measures $\mu$ and $\nu$. Then use the lower semicontinuity of $d$ to show that any
accumulation point of approximate minimisers must be a minimiser. 
\end{exercise}

\begin{exercise}
Use disintegration of measures to show that \eqref{e:defW1} does indeed determine a distance function, i.e.\ 
that it satisfies the triangle inequality.
\end{exercise}

\begin{theorem}
If the metric $d$ is bounded and generates the topology of $\CX$, then its $p$-Wasserstein lift to $\cP(\CX)$ generates the topology of weak
convergence.
\end{theorem}

\begin{proof}
We only consider the case $p=1$, the general case is similar. Consider a sequence 
$\mu_n \in \cP(\CX)$ such that $d(\mu,\mu_n) \to 0$ for  some $\mu \in \cP(\CX)$, so that 
one can find $\pi_n \in \cC(\mu,\mu_n)$ such that 
$\int d(x,y)\,\pi_n(dx,dy) \to 0$. Let now $f$ be an arbitrary continuous function on $\CX$
with absolute value bounded by $1$, 
fix $\eps > 0$ arbitrary, and let $K \subset \CX$ compact be such that $\mu(K) > 1-\eps$.
Let furthermore $\delta > 0$ be such that, for every $x \in K$,
and every $y \in \CX$ with $d(x,y) \le \delta$, we have $|f(y) - f(x)| < \eps$.
(Such a $\delta$ exists, otherwise we could find sequences $x_n \in K$ and $y_n \in \CX$ with $d(x_n, y_n) \to 0$
and $|f(x_n) - f(y_n)| \ge \eps$. This would contradicts continuity of $f$ at accumulation points of the
sequence $x_n$, which exist by compactness of $K$.)

Setting $\Delta_\delta = \{(x,y) \in \CX^2\,:\, d(x,y) \le \delta\}$, we can find $N$ large enough 
such that $\pi_n(\Delta_\delta) > 1-\eps$ for every $n > N$. It follows that
\begin{equs}
\Big|\int f\,d\mu_n - \int f\,d\mu\Big| &= \Big|\int (f(x) - f(y))\,\pi_n(dx,dy)\Big|
\le \int_{K \times \CX} |f(x) - f(y)|\,\pi_n(dx,dy) + 2\eps \\
&\le \int_{(K \times \CX) \cap \Delta_\delta} |f(x) - f(y)|\,\pi_n(dx,dy) +  4 \eps
\le 5 \eps\;.
\end{equs}
Since $\eps$ was arbitrary, this shows that $\int f\,d\mu_n \to \int f\,d\mu$ and, since $f$ was
also arbitrary, the desired weak convergence follows.

Conversely, let $\mu_n \to \mu$ weakly, let $\eps > 0$, and let $K$ be as above.
By compactness of $K$, it is straightforward to find a collection of continuous functions
$\psi_k \colon \CX \to [0,1]$ with $k \ge 1$ such that only finitely many of the $\psi_k$'s 
are non-zero, $\sum_{k \ge 1} \psi_k(x) = 1$ for every $x \in K$,
and $\diam\supp\psi_k \le \eps$ for every $k$. For every $k \ge 1$, we also choose 
some $x_k \in \supp\psi_k$ and we write $\psi_0 = 1-\sum_{k\ge 1} \psi_k$.
Given any probability measure $\nu$, we then set
\begin{equ}[e:defFmu]
F(\nu) = \psi_0\nu + \sum_{k\ge 1} \int \psi_k\,d\nu\,\delta_{x_k}\;,\qquad \pi_{\nu} = \Delta^* (\psi_0\nu) + \sum_{k\ge 1} (\psi_k \nu)\otimes \delta_{x_k} \;,
\end{equ}
where $\Delta \colon x \mapsto (x,x)$ is the diagonal map. 
Note that $F(\nu)$ is again a probability measure and that $\pi_\nu \in \cC(F(\nu),\nu)$.

Since $d(x_k,y) \le \eps$ 
for every $k \ge 1$ and every $y \in \supp \psi_k$, and assuming without loss of generality that the distance $d$ is
bounded by $1$, it furthermore follows that 
\begin{equ}[e:boundCoupl]
\int d(x,y)\, \pi_{\nu}(dx,dy) \le \eps + \int \psi_0\,d\nu\;.
\end{equ}
We then estimate
\begin{equ}
d(\mu_n,\mu) \le d(\mu_n,F(\mu_n)) + d(F(\mu_n),F(\mu)) + d(F(\mu),\mu)\;.
\end{equ}
Since $\psi_0$ is supported outside of $K$, it follows from \eqref{e:boundCoupl}
that $d(F(\mu),\mu) \le 2\eps$ and, since $\pi_0$ is bounded and continuous, 
the weak convergence of the $\mu_n$ implies that $d(\mu_n,F(\mu_n)) \le 3\eps$ (say)
for all $n$ sufficiently large. For the remaining term, we write
\begin{equ}
a_n^{(k)} = \int \psi_k\,d\mu_n \wedge \int \psi_k\,d\mu\;,\qquad \delta_n = 1-\sum_{k\ge 1} a_n^{(k)}\;,
\end{equ}
so that, since $\lim_{n \to \infty} a_n^{(k)} = \int \psi_k\,d\mu$, one has
$\lim_{n \to \infty}\delta_n \le \eps$. Furthermore, one has decompositions
\begin{equ}
F(\mu_n) = \sum_{k \ge 1} a_n^{(k)}\delta_{x_k} + \delta_n \eta_n\;,\quad
F(\mu) = \sum_{k \ge 1} a_n^{(k)}\delta_{x_k} + \delta_n \hat \eta_n\;,
\end{equ}
for $\eta_n$, $\hat \eta_n$ some probability measures.
It follows immediately that 
\begin{equ}
\pi_n = \sum_{k \ge 1} a_n^{(k)}\delta_{(x_k,x_k)} + \delta_n \eta_n\otimes \hat \eta_n
\in \cC(F(\mu_n),F(\mu))\;,
\end{equ}
so that $d(F(\mu_n),F(\mu)) \le \delta_n$, and the proof that $d(\mu_n,\mu) \to 0$ is complete.
\end{proof}

A very useful feature of the Wasserstein-$1$ distances is that they can also be viewed as the dual norm to the Lipschitz norm on functions.
This is the content of the celebrated Monge--Kantorovich--Rubinstein duality theorem
(see for example \cite{Vil07}) which we state here without proof.

\begin{theorem}
For $d$ any lower semicontinuous metric on $\CX$, the identity
\begin{equ}[e:dual]
d(\mu,\nu) = \sup_{\Lip_d(\phi) \le 1} \Bigl(\int_\CX \phi(x)\,\mu(dx) - \int_\CX \phi(y)\,\mu(dy)\Bigr)
\end{equ} 
holds for all pairs $(\mu,\nu)$ of probability measures. Here, $\Lip_d(\phi)$ denotes the best Lipschitz constant for $\phi$
with respect to the metric $d$.
\end{theorem}

\begin{remark}
The metric $d$ does not need to be bounded in general, so there might be pairs of probability measures for which $d(\mu,\nu)$ is infinite.
\end{remark}

\begin{remark}
There exists a generalisation of the duality \eref{e:dual} that holds also if $d$ is not a distance function (and therefore also for the $p$-Wasserstein distances for $p>1$),
but it is slightly more complicated to state. See \cite{Vil03,Vil07} for a very nice treatment of many questions related to Wasserstein distances.
\end{remark}

\begin{remark}
It may appear surprising at first sight that an explicit bound on the rate of convergence of integrals of a sequence $\{\mu_n\}$
against Lipschitz continuous functions should yield convergence of the same integrals against any continuous function. 
However,  recall that Prohorov's theorem tells us that any converging sequence of probability measures is essentially concentrated
on compact sets. Since on a compact set, any continuous function can be approximated uniformly by Lipschitz continuous functions,
this should make it much more plausible that \eref{e:dual} does indeed define the topology of weak convergence.
\end{remark}

\begin{remark}
Although $\cP(\CX)$ is complete under $d$, the space of signed measures with finite mass is not complete, if we endow it with the norm
defined in \eref{e:dual}. 
To see this, take for example $\CX = [0,1]$ and let $d$ be the usual distance function. Then, the sequence
\begin{equ}
f_n = \sum_{k=1}^n k^2 \bigl(\delta_{3^{-k}} - \delta_{2\cdot 3^{-k}}\bigr)\;,
\end{equ}
is Cauchy, but it obviously does not converge to a measure with finite mass.
\end{remark}

One special case of the Monge--Kantorovich--Rubinstein duality is of particular interest.
Setting $d_\TV$ to be the trivial
distance function which is equal to $2$ for all pairs $(x,y)$ with $x \neq y$, we see that the
$1$-Wasserstein lift of $d_\TV$ to $\cP(\CX)$ as in \eref{e:defW1} is nothing but the total variation distance as 
characterised in \eref{e:TVCoupl}.

On the other hand, the set of $d_\TV$-Lipschitz continuous functions $\phi$ with
best Lipschitz constant $1$ is, up to translations by constants, equal to the set of bounded functions with $\sup_x |\phi(x)| \le 1$, so that the
dual representation of the $1$-Wasserstein lift of $d_\TV$ as in \eref{e:dual} is nothing but the original definition of the total variation distance
given in \eref{e:defTV}.

\begin{exercise}
Convince yourself that the identity \eref{e:TVCouplV} is also a special case of the Monge--Kantorovich--Rubinstein duality. What is the
corresponding distance function? \index{Monge--Kantorovich--Rubinstein duality}
\end{exercise}

\begin{exercise}
Show that every probability measure $\mu$ on $\CX$ can be approximated by a finite convex combination of Dirac measures
in the topology of weak convergence. \textbf{Hint:} Use a slight modification of the 
construction \eqref{e:defFmu}. 
\end{exercise}

\section{Gaussian Measure Theory}
\label{sec:Gaussian}

This section is devoted to the study of Gaussian measures on general Banach spaces, but we start
by recalling the definition of a Gaussian measure on the reals:

\begin{definition}
A probability measure $\mu$ on $\R$ is sait to be \textit{Gaussian} if there exist 
$\sigma \ge 0$ and $m \in \R$ such that, for every $\ell \in \R$,
\begin{equ}[e:GaussFourier]
\int e^{i\ell x}\mu(dx) = \exp\Big(- \f\sigma2 \ell^2 + i\ell m\Big)\;.
\end{equ}
It is said to be centred if furthermore $m=0$.
\end{definition}

\begin{remark}
The advantage of this definition is that it covers the case $\sigma = 0$, where $\mu$ is a Dirac
measure. It also avoids having factors of $\sqrt{2\pi}$ appearing in the definition. 
When $\sigma > 0$, one can invert the Fourier transform to recover the usual expression, identifying
$\sigma$ as the variance of $\mu$ and $m$ as its mean.
\end{remark}

Throughout this section and throughout most of the remainder of these notes, we will denote by $\CB$
an arbitrary separable Banach space. Recall that a space is separable if it contains
a countable dense subset, see for example the monograph \cite{Yos95}. This separability assumption 
turns out to be crucial for measures on $\CB$ to behave
in a non-pathological way. It  can be circumvented
by trickery in most natural situations where non-separable spaces arise, but we choose not to complicate our lives by considering
overly general cases in these notes. Another convention that will be used throughout these notes is that all of
the measures that we consider are Borel measures, meaning that we consider every open set to be measurable.

One additional assumption that would appear to be natural in the context of Gaussian measure theory is that $\CB$ be
reflexive (that is $\CB^{**} = \CB$). This is for example because the mean of a measure $\mu$ appears at 
first sight to be an element of $\CB^{**}$ rather
than of $\CB$, since the natural\footnote{Without further assumption, we do not know a priori whether $x\mapsto \|x\|$ is integrable, so that the more natural definition $m = \int_\CB x\,\mu(dx)$ is prohibited.} way of defining the mean $m$ of $\mu$ is to set 
$m(\ell) = \int_\CB \ell(x)\,\mu(dx)$ for
any $\ell \in \CB^*$. This turns out not to be a problem, since the mean of a Gaussian measure on a \textit{separable} Banach space $\CB$
 is always an element of $\CB$ itself, see the monograph \cite{Bog98Gauss}. 
 However  this result is not straightforward to prove, so we will take here the more pragmatic approach that whenever we consider
 Gaussian measures with non-zero mean, we simply take the mean $m \in \CB$ as given.

\begin{example}
Before we proceed, let us just mention a few examples of Banach spaces. The spaces 
$L^p(\CM,\nu)$ (with $(\CM,\nu)$ any countably generated measure space like for example any Polish space 
equipped with a
Radon measure $\nu$)
for $p \in (1,\infty)$ are both reflexive and separable. However, reflexivity fails in general for $L^1$ 
spaces 
and both properties fail to hold in general for $L^\infty$ spaces \cite{Yos95}.\footnote{This is actually a very subtle question which depends on the full axiom of choice. 
If one replaces it by the slightly weaker axiom of dependent choice, which is sufficient to develop
all of ``concrete'' mathematics, then the question of whether $L^1 = (L^\infty)^*$ is independent of the axioms.} The space of bounded 
continuous functions on a compact space is separable, but not reflexive. The space of bounded continuous functions from $\R^n$ to $\R$
is neither separable nor reflexive, but the space of  continuous functions from $\R^n$ to $\R$ vanishing at infinity is separable.
(The last two statements are still true if we replace $\R^n$ by any locally compact complete separable metric space.)
Hilbert spaces are obviously reflexive since $\CH^* \simeq \CH$ for every Hilbert space $\CH$ by the Riesz representation theorem
\cite{Yos95}. There exist non-separable Hilbert spaces,
but they have rather pathological properties and do not appear very often in practice. 
One such example is the space of all functions $f \colon \R \to \R$ that vanish at all but countably many values and
such that $\sum_{t \in \R} |f(t)|^2 < \infty$.
\end{example}

We start with the definition of a Gaussian measure on a Banach space. Since there is no equivalent to Lebesgue measure
in infinite dimensions (one could never expect it to be $\sigma$-additive), we cannot
define it by prescribing the form of its density. However, it turns out that Gaussian measures on $\R^n$ can be characterised by
prescribing that the projections of the measure onto any one-dimensional subspace of $\R^n$ are all Gaussian.
This is a property that can readily be generalised to infinite-dimensional spaces:

\begin{definition}
A \index{Gaussian measure}\textit{Gaussian probability measure} $\mu$ on a Banach space $\CB$ is a Borel measure such that
$\ell^*\mu$ is a real Gaussian probability measure on $\R$ for every linear functional $\ell \colon \CB \to \R$.
We call it \textit{centred} if $\ell^*\mu$ is centred for every $\ell$.
\end{definition}

\begin{remark}
We used here the notation $f^*\mu$ for the push-forward of a measure $\mu$ under a map $f$. This is defined by
the relation $\bigl(f^*\mu\bigr)(A) = \mu \bigl(f^{-1}(A)\bigr)$.
\end{remark}

\begin{remark}\label{rem:Rn}
We could also have defined Gaussian measures by imposing that $T^*\mu$ is Gaussian for every bounded linear
map $T\colon \CB \to \R^n$ and every $n$. These two definitions are equivalent because probability measures on $\R^n$ are
characterised by their Fourier transforms and these are determined by one-dimensional marginals, see Proposition~\ref{prop:FT} below.
\end{remark}

\begin{exercise}\label{ex:traceClass}
Let $\{\xi_n\}$ be a sequence of i.i.d.\ $\CN(0,1)$ random variables and let $\{a_n\}$ be a sequence of real numbers.
Show that the law of $(a_0\xi_0, a_1\xi_1,\ldots)$ determines a Gaussian measure on $\ell^2$ if and only if $\sum_{n \ge 0} a_n^2 < \infty$.
\end{exercise}

One first question that one may ask is whether this is indeed a reasonable definition. After all, it only makes a 
statement about the one-dimensional
projections of the measure $\mu$, which itself lives on a huge infinite-dimensional space. 
However, this turns out to be reasonable since, provided
that $\CB$ is separable, the one-dimensional projections of any probability measure carry sufficient
information to characterise it. This statement can be formalised as follows:

\begin{proposition}\label{prop:determinefinitedim}
Let $\CB$ be a separable Banach space and let $\mu$ and $\nu$ be two probability Borel measures on $\CB$. If 
$\ell^*\mu = \ell^*\nu$ for every $\ell \in \CB^*$, then $\mu = \nu$.
\end{proposition}

\begin{proof}
Denote by $\Cyl(\CB)$ the algebra of cylindrical sets on $\CB$, that is $A \in \Cyl(\CB)$ if and only
if there exists $n > 0$, a continuous linear map $T \colon \CB \to \R^n$, and a Borel set $\tilde A \subset \R^n$
such that $A = T^{-1} \tilde A$. It follows from the fact that measures on $\R^n$ are determined by their one-dimensional
projections that $\mu(A) = \nu(A)$ for every $A \in \Cyl(\CB)$
and therefore, by a basic uniqueness result in measure theory (see Lemma~II.4.6 in \cite{RogWil1} or Theorem~1.5.6 in \cite{BogMT} for example),
for every $A$ in the $\sigma$-algebra $\CE(\CB)$ generated by $\Cyl(\CB)$. It thus remains to show that 
$\CE(\CB)$ coincides with the Borel $\sigma$-algebra of $\CB$. Actually, since every cylindrical set is a Borel set, it suffices
to show that all open (and therefore all Borel) sets are contained in $\CE(\CB)$. 

Since $\CB$ is separable, every open set $U$ can be written as a countable union of closed balls. (Fix any dense countable subset
$\{x_n\}$ of $\CB$ and check that one has for example $U = \bigcup_{x_n \in U} \bar B(x_n, r_n)$, where $r_n = {1\over 2}\sup \{r>0\,:\, \bar B(x_n, r) \subset U\}$ and $\bar B(x,r)$ denotes the closed ball of radius $r$ centred at $x$.)
Since $\CE(\CB)$ is invariant under translations and dilations, it remains to check that $\bar B(0,1) \in \CE(\CB)$. 
Let $\{x_n\}$ be a countable dense subset of $\{x\in \CB \,:\, \|x\| = 1\}$ and let $\ell_n$ be any sequence in $\CB^*$ such that
$\|\ell_n\| = 1$ and $\ell_n(x_n) = 1$ (such elements exist by the Hahn--Banach extension theorem \cite{Yos95}).
Define now $K = \bigcap_{n \ge 0} \{x \in \CB \,:\, |\ell_n(x)| \le 1\}$. It is clear that $K \in \CE(\CB)$, so that the proof is complete if we can show that $K = \bar B(0,1)$. 

Since obviously $\bar B(0,1) \subset K$, it suffices to show that the reverse inclusion holds.
Let  $y \in \CB$ with $\|y\| > 1$ be arbitrary and set $\hat y = y / \|y\|$.
By the density of the $x_n$'s, there exists a subsequence $x_{k_n}$ such that $\|x_{k_n} - \hat y\| \le {1\over n}$, say, so that
$\ell_{k_n}(\hat y) \ge 1-{1\over n}$. By linearity, this implies that $\ell_{k_n}(y) \ge \|y\| \bigl(1-{1\over n}\bigr)$, so that there exists
a sufficiently large $n$ so that $\ell_{k_n}(y) > 1$. This shows that $y \not \in K$ and we conclude that $K \subset \bar B(0,1)$ as required. 
\end{proof}

From now on, we will mostly consider centred Gaussian measures, since one can always reduce oneself to the general case
by a simple translation.
Given a centred Gaussian measure $\mu$, we define a map $C_\mu \colon \CB^* \times \CB^* \to \R$
by
\begin{equ}[e:defCmu]
C_\mu(\ell, \ell') = \int_\CB \ell(x) \ell'(x)\,\mu(dx)\;.
\end{equ}
\begin{remark}
In the case $\CB = \R^n$, this is just the covariance matrix, provided that we perform the usual identification of $\R^n$ with its dual.
\end{remark}

\begin{remark}\label{rem:Chat}
One can identify in a canonical way $C_\mu$ with an operator $\hat C_\mu \colon \CB^* \to \CB^{**}$ via the identity
$\hat C_\mu(\ell)(\ell') = C_\mu(\ell,\ell')$.
\end{remark}

\index{covariance operator}
The map $C_\mu$ will be called the \textit{Covariance operator} of $\mu$. It follows immediately from the definitions that
the operator $C_\mu$ is bilinear and positive definite, although there might in general exist some $\ell$ such that $C_\mu(\ell, \ell)=0$. 
Furthermore, it immediately follows from \eqref{e:GaussFourier} that a measure $\mu$ is Gaussian
if and only if its Fourier transform $\hat\mu$ is given by\index{Fourier transform}
\begin{equ}[e:FT]
\hat \mu(\ell) \eqdef \int_\CB e^{i\ell(x)}\, \mu(dx) = \exp\bigl(-\hf C_\mu(\ell, \ell)\bigr)\;,
\end{equ}
where $\ell \in \CB^*$. 
It will be convenient to know that a version of Bochner's theorem still holds in this case, namely
that probability measures on $\CB$ are determined by their Fourier transforms:

\begin{proposition}\label{prop:FT}
Let $\mu$ and $\nu$ be any two probability measures on a separable Banach space $\CB$. If $\hat \mu(\ell) = \hat \nu(\ell)$ for
every $\ell \in \CB^*$, then $\mu = \nu$.
\end{proposition}
\begin{proof}
In the particular case $\CB = \R$, if $\phi$ is a smooth function with compact support, it follows from Fubini's theorem
and the invertibility of the Fourier transform that one has the identity
\begin{equ}
\int_{\R} \phi(x)\,\mu(dx) = {1\over 2\pi} \int_{\R} \int_{\R} \hat \phi(k) e^{-ikx}\,dk \,\mu(dx) = {1\over 2\pi}\int_{\R} \hat \phi(k)\,\hat \mu(-k)\,dk \;,
\end{equ}
so that, since bounded continuous functions can be approximated by smooth functions, 
$\mu$ is indeed determined by $\hat \mu$. The general case then follows immediately from Proposition~\ref{prop:determinefinitedim}.
\end{proof}

As a simple consequence, we have the following trivial but useful property:

\begin{proposition}\label{prop:rotation}
Let $\mu$ be a Gaussian measure on $\CB$ and, for every $\phi \in \R$, define the ``rotation'' $R_\phi\colon \CB^2 \to \CB^2$
by
\begin{equ}
R_\phi(x,y) = \bigl(x \sin \phi + y \cos \phi, x \cos \phi - y \sin \phi\bigr)\;.
\end{equ}
Then, one has $R_\phi^* \bigl(\mu \otimes \mu\bigr) = \mu \otimes \mu$.
\end{proposition}

\begin{proof}
Since we just showed in Proposition~\ref{prop:FT} that a measure is characterised by its Fourier transform, 
it suffices to check that $\widehat {\mu \otimes \mu} \circ R_\phi =
\widehat {\mu \otimes \mu}$, which is an easy exercise. 
\end{proof}

\subsection{A priori bounds on Gaussian measures}
\label{sec:boundsGauss}

We are going to show now that the operator $C_\mu$ has to be bounded, as a 
straightforward consequence of the fact that $x \mapsto \|x\|^2$ is integrable for any Gaussian measure. 
Actually, we are going to show much more, 
namely that there always exists a constant $\alpha > 0$ such that $\exp(\alpha \|x\|^2)$ is integrable! In other words, the norm of
any Banach-space valued Gaussian random variable has Gaussian tails, just like in the finite-dimensional case. 
Amazingly, this result uses the Gaussianity of the measure only indirectly through the rotation
invariance shown in Proposition~\ref{prop:rotation}, and even this property is only used for rotations by the angle $\phi=\pi/4$.
This is the
content of the following fundamental result \cite{Fernique} in the theory of Gaussian measures:

\index{Fernique's theorem}
\begin{theorem}[Fernique, 1970]\label{theo:Fernique}
Let $\mu$ be any probability measure on a separable Banach space $\CB$ such that the conclusion of Proposition~\ref{prop:rotation} 
holds for $\phi = \pi/4$. Then, there exists $\alpha > 0$ such that $\int_\CB \exp(\alpha \|x\|^2)\,\mu(dx) < \infty$.
\end{theorem}
\begin{proof}
Note first that, from Proposition~\ref{prop:rotation}, one has for any two positive numbers $t$ and $\tau$ the
bound
\begin{equs}
\mu(\|x\| \le \tau)\, \mu(\|x\| > t) &= \int_{\|x\| \le \tau}\int_{\|y\| > t} \mu(dx)\,\mu(dy) = \int_{\|{x - y \over \sqrt 2}\| \le \tau}\int_{\|{x + y \over \sqrt 2}\| > t} \mu(dx)\,\mu(dy) \\
&\le \int_{\|x\| > {t-\tau \over \sqrt 2}}\int_{\|y\| > {t-\tau \over \sqrt 2}}\mu(dx)\,\mu(dy) = \mu\Bigl(\|x\| > \textstyle{t-\tau \over \sqrt 2}\Bigr)^2\;. \label{e:FernBound}
\end{equs}
In order to go from the first to the second line, we have used the fact that the triangle inequality implies
\begin{equ}
\min\{\|x\|, \|y\|\} \ge \hf \bigl(\|x+y\| - \|x-y\|\bigr)\;,
\end{equ}
so that $\|x+y\| > \sqrt 2t$ and $\|x-y\| \le \sqrt 2\tau$ do indeed imply that both $\|x\|$ and $\|y\|$ are greater than ${t-\tau \over \sqrt 2}$. Since $\|x\|$ is $\mu$-almost surely finite, there exists some $\tau > 0$ such that $\mu(\|x\| \le \tau) \ge {3\over 4}$.
Set now $t_0 = \tau$ and define $t_n$ for $n > 0$ recursively by the relation $t_n = {t_{n+1} - \tau \over \sqrt 2}$.
It follows from \eref{e:FernBound} that
\begin{equ}
\mu(\|x\| > t_{n+1}) \le  \mu\Bigl(\|x\| > \textstyle{t_{n+1}-\tau \over \sqrt 2}\Bigr)^2 / \mu(\|x\| \le \tau)  \le  {4\over 3}\mu(\|x\| > t_{n})^2\;.
\end{equ}
Setting $y_n = {4\over 3}\mu(\|x\| > t_{n+1})$, this yields the recursion $y_{n+1} \le y_n^2$ with $y_0 \le 1/3$.
Applying this inequality  repeatedly, we obtain
\begin{equ}
\mu(\|x\| > t_{n}) = {3\over 4} y_n \le {3\over 4} y_0^{2^n} \le {1\over 4}3^{-1-2^n} \le 3^{-2^n}\;.
\end{equ}
On the other hand, one can check explicitly that $t_n = {{\sqrt 2}^{n+1} - 1 \over \sqrt 2 - 1} \tau \le 2^{n/2}\cdot (2+\sqrt 2)\tau$, 
so that in particular $t_{n+1} \le 2^{n/2} \cdot 5\tau$. This shows that one has the bound 
\begin{equ}
\mu(\|x\| > t_{n}) \le 3^{- {t_{n+1}^2 \over 25 \tau^2} }\;,
\end{equ}
implying that there exists a \textit{universal} constant $\alpha>0$ such that the bound
$\mu(\|x\| > t) \le \exp(-2\alpha t^2/\tau^2)$ holds for every $t \ge \tau$. Integrating by parts, we finally obtain
\begin{equs}
\int_{\CB} \exp\Bigl({\alpha \|x\|^2 \over \tau^2}\Bigr)\,\mu(dx) &\le e^{\alpha} + {2\alpha\over \tau^2} \int_\tau^\infty t e^{\alpha {t^2\over \tau^2}} \mu(\|x\| > t)\, dt \\
& \le e^{\alpha} + 2\alpha \int_1^\infty t e^{-\alpha t^2}\, dt < \infty\;, \label{e:universal}
\end{equs}
which is the desired result.
\end{proof}

As an immediate corollary of Fernique's theorem, we have
\begin{corollary}
There exists a constant $\|C_\mu\| < \infty$ such that $C_\mu(\ell, \ell') \le \|C_\mu\| \|\ell\|\|\ell'\|$ for any 
$\ell, \ell' \in \CB^*$. Furthermore, the operator $\hat C_\mu$ defined in Remark~\ref{rem:Chat} is a 
continuous operator from $\CB^*$ to $\CB$.
\end{corollary}
\begin{proof}
The boundedness of $C_\mu$ implies that $\hat C_\mu$ is continuous from $\CB^*$ to $\CB^{**}$. However, $\CB^{**}$
might be strictly larger than $\CB$ in  general. The fact that the range of $\hat C_\mu$ actually belongs to $\CB$ follows from the fact that one has the identity
\begin{equ}[e:defCmuhat]
\hat C_\mu \ell = \int_{\CB} x\, \ell(x)\,\mu(dx)\;.
\end{equ}
Here, the right-hand side is well-defined as a Bochner integral \cite{Bochner,Hildebrandt} because $\CB$ is 
assumed to be separable
and we know from Fernique's theorem that $\|x\|^2$ is integrable with respect to $\mu$.
\end{proof}

\begin{remark}
In Theorem~\ref{theo:Fernique}, one can actually take for $\alpha$ any value smaller than ${1/(2\|C_\mu\|)}$. 
Furthermore, this value happens to be sharp, see \cite[Thm~4.1]{Ledoux}.
\end{remark}

Another consequence of the proof of Fernique's theorem is an even stronger result, namely all moments
(including exponential moments!)
of the norm of a Banach-space valued Gaussian random variable can be estimated in a universal way in terms
of its first moment. More precisely, we have 

\begin{proposition}\label{prop:universalFernique}
There exist universal constants $\alpha, K > 0$ with the following properties.
Let $\mu$ be a centred Gaussian measure on a separable Banach space $\CB$ and let $f \colon \R_+ \to \R_+$ be any measurable
function such that $f(x) \le C_{\!f} \exp(\alpha x^2)$ for every $x \ge 0$. Define furthermore the first moment of $\mu$ by
$M = \int_\CB \|x\|\,\mu(dx)$. Then, one has the bound $\int_\CB f(\|x\|/M)\,\mu(dx) \le K C_{\! f}$.

In particular, the higher moments of $\mu$ are bounded by $\int_\CB \|x\|^{2n}\,\mu(dx) \le n! K \alpha^{-n} M^{2n}$.
\end{proposition}

\begin{proof}
It suffices to note that the bound \eref{e:universal} is independent of $\tau$ and that by Chebychev's inequality, one
can choose for example $\tau = 4M$. The last claim then follows from the fact that $e^{\alpha x^2} \ge {\alpha^{n} x^{2n}\over n!}$.
\end{proof}


Actually, the covariance operator 
$C_\mu$ is more than just bounded. If $\CB$ happens to be a (separable) Hilbert space, one has indeed
the following result, which allows us to characterise in a very precise way the set of all centred Gaussian measures on
a Hilbert space:
\begin{proposition}\label{prop:trace}
If $\CB = \CH$ is a Hilbert space, then the operator $\hat C_\mu \colon \CH \to \CH$ defined as 
before by the identity $\scal{\hat C_\mu h, k} = C_\mu(h,k)$
is trace class and one has the identity
\begin{equ}[e:expnorm]
\int_\CH \|x\|^2\,\mu(dx) = \tr \hat C_\mu\;.
\end{equ}
(Here, we used Riesz's representation theorem to identify $\CH$ with its dual.)

Conversely, for every positive symmetric trace class  operator $K\colon \CH \to \CH$, there exists a Gaussian measure $\mu$ on $\CH$ such
that $\hat C_\mu = K$.
\end{proposition}
\begin{proof}
Fix an arbitrary orthonormal basis $\{e_n\}$ of $\CH$. We know from Theorem~\ref{theo:Fernique}
that the second moment of the norm is finite: $\int_\CH \|h\|^2\,\mu(dh) < \infty$. On the other hand, one has
\begin{equ}
\int_\CH \|h\|^2\,\mu(dh) = \sum_{n=1}^\infty \int_\CH \scal{h, e_n}^2\,\mu(dh) = \sum_{n=1}^\infty \scal{e_n, \hat C_\mu e_n} = \tr \hat C_\mu\;,
\end{equ}
which is precisely \eref{e:expnorm}. To pull the sum out of the integral in the first equality, we used Lebesgue's 
dominated convergence theorem.

In order to prove the converse statement, since $K$ is compact, 
we can find an orthonormal basis $\{e_n\}$ of $\CH$
such that $K e_n = \lambda_n e_n$. Since $K$ is positive and symmetric, one further has 
$\lambda_n \ge 0$, while the trace class property reads $\sum_n \lambda_n < \infty$. Let now $\{\xi_n\}$ be
a collection of i.i.d.\ $\CN(0,1)$ Gaussian random variables (such a family exists by \index{Kolmogorov!extension theorem}
Kolmogorov's extension theorem but can also be realised explicitly over the probability space $[0,1]$). Then,
since $\sum_n \lambda_n \E \xi_n^2 = \tr K < \infty$, the series $\sum_n \sqrt \lambda_n \xi_n e_n$ converges in mean square,
so that it has a subsequence converging almost surely 
in $\CH$. One can check as in Exercise~\ref{ex:traceClass} that the law of the limiting random variable is Gaussian and has the requested covariance.
\end{proof}

No such precise characterisation of the covariance operators of Gaussian measures exists in 
general Banach spaces. One can however
show that $\hat C_\mu$ is at least a little bit better than bounded, namely that it is always a compact operator. We leave this statement
as an exercise for the interested reader, since we will not make any use of it in these notes:

\begin{exercise}
Show that in the case of a Gaussian measure $\mu$ on a general separable Banach space $\CB$, the covariance 
operator $\hat C_\mu \colon \CB^* \to \CB$ is compact in the sense that it maps the unit ball on $\CB^*$ into a compact subset
of $\CB$. \textbf{Hint:} Combine Fernique's theorem with Lebesgue's dominated convergence theorem
and the fact that, since $\CB$ is separable, the unit ball in $\CB^*$ is weak-* compact.
\end{exercise}

In many situations, it is furthermore helpful to find out whether a given covariance structure can be realised as
a Gaussian measure on some space of H\"older continuous functions. This can be achieved through the following version
of Kolmogorov's continuity criterion, which can be found for example in \cite[p.~26]{RevYor}:\index{Kolmogorov!continuity criterion}

\begin{theorem}[Kolmogorov]\label{theo:Kolmogorov}
For $d > 0$, let $C\colon [0,1]^d \times [0,1]^d \to \R$ be a symmetric function such that, for every finite collection $\{x_i\}_{i=1}^m$
of points in $[0,1]^d$, the matrix $C_{ij} = C(x_i, x_j)$ is positive semi-definite. If furthermore there exists $\alpha > 0$ and a constant
$K > 0$ such that $C(x,x) + C(y,y) - 2C(x,y)  \le K|x-y|^{2\alpha}$ for any two points $x,y \in [0,1]^d$ then
there exists a unique centred Gaussian measure $\mu$ on $\CC([0,1]^d, \R)$ 
such that
\begin{equ}[e:defmuKol]
\int_{\CC([0,1]^d, \R)} f(x)f(y)\,\mu(df) = C(x,y)\;,
\end{equ}
for any two points $x,y \in [0,1]^d$. Furthermore,  for every $\beta < \alpha$, one has $\mu\bigl(\CC^\beta([0,1]^d, \R)\bigr) = 1$, where
$\CC^\beta([0,1]^d, \R)$ is the space of $\beta$-H\"older continuous functions.
\end{theorem}

\begin{proof}
Set $\CB = \CC([0,1]^d, \R)$ and $\CB^*$ its dual, which consists of the set of Borel measures with finite total variation \cite[p.~119]{Yos95}.
Since convex combinations of Dirac measures are dense (in the topology of \index{weak convergence}{weak convergence}) in the set of probability measures,
it follows that the set of linear combinations of point evaluations is weakly dense in $\CB^*$. Therefore, the claim follows if we
are able to construct a measure $\mu$ on $\CB$ such that \eref{e:defmuKol} holds and such that, if $f$ is distributed
according to $\mu$, then for any finite collection of points $\{x_i\} \subset [0,1]^d$, the joint law of the $f(x_i)$ is Gaussian.

By Kolmogorov's extension theorem, we can construct a measure $\mu_0$ on $\X = \R^{[0,1]^d}$ endowed with the product $\sigma$-algebra
such that the laws of all finite-dimensional marginals are Gaussian and satisfy \eref{e:defmuKol}.\footnote{If the uncountable product appearing here makes you queasy, you can restrict yourself to the countable set of points
with dyadic coordinates, which is all our construction will rely on.} We denote by $X$ an $\CX$-valued
random variable with law $\mu_0$. At this stage, one would think that the proof is complete if we can show that $X$ almost surely
has finite $\beta$-H\"older norm. The problem with this statement is that the $\beta$-H\"older norm is not a measurable function on $\CX$!
The reason for this is that it requires point evaluations of $X$ at uncountably many locations, while functions that are measurable
with respect to the product $\sigma$-algebra on  $\CX$ are allowed to depend on at most countably many function evaluations.

This problem can be circumvented very elegantly in the following way.
Denote by $\CD \subset [0,1]^d$ the subset of dyadic numbers (actually any countable dense subset would do for now, but the dyadic numbers 
will be convenient later on) and define the event $\Omega_\beta$ by 
\begin{equ}
\Omega_\beta = \Bigl\{X\,:\, \hat X(x) \eqdef \lim_{y \to x \atop y \in \CD} X(y) \,\text{exists for every $x \in [0,1]^d$ and $\hat X$ belongs to $\CC^\beta([0,1]^d, \R)$}\Bigr\}\;.
\end{equ}
Since the event $\Omega_\beta$ can be constructed from evaluating $X$ at only countably many points, it is a measurable set.
For the same reason, the map $\iota\colon \CX \to \CC^\beta([0,1]^d, \R)$ given by
\begin{equ}
\iota(X) = \left\{\begin{array}{cl} \hat X & \text{if $X \in \Omega_\beta$,} \\ 0 & \text{otherwise} \end{array}\right.
\end{equ}
is measurable with respect to the product $\sigma$-algebra on $\CX$ (and the Borel $\sigma$-algebra on $\CC^\beta$), 
so that the claim follows if we can show that $\mu_0(\Omega_\beta) = 1$
for every $\beta < \alpha$. (Take $\mu = \iota^*\mu_0$.)
Denoting the $\beta$-H\"older norm of $X$ restricted to the dyadic numbers
by $M_\beta(X) = \sup_{x \neq y \,:\, x,y \in \CD} \{|X(x) - X(y)| / |x-y|^\beta\}$, we see that
$\Omega_\beta$ can alternatively be characterised as $\Omega_\beta = \{X\,:\, M_\beta(X) < \infty\}$, so that the claim
follows if we can show for example that $\E M_\beta(X) <\infty$.

Denote by $\CD_m \subset \CD$ the set of those numbers whose coordinates
are integer multiples of $2^{-m}$ and denote by $\Delta_m$ the set of pairs $x,y \in \CD_m$ such that $|x-y| = 2^{-m}$.
In particular, note that $\Delta_m$ contains at most $2^{(m+2)d}$ such pairs. 

We are now going to make use of our simplifying assumption that we are dealing with Gaussian random variables, so that $p$th
moments can be bounded in terms of second moments.
More precisely, for every $p \ge 1$ there exists a constant $C_p$ such that if $X$ is a Gaussian random variable, then  
one has the bound $\E |X|^p \le C_p \bigl(\E |X|^2\bigr)^{p/2}$.
Setting $K_m(X) = \sup_{x,y \in \Delta_m} |X(x) - X(y)|$ and fixing some arbitrary $\beta' \in (\beta, \alpha)$, 
we see that for $p \ge 1$ large enough, there exists a constant $K_p$ such that
\begin{equs}
\E K_m^p(X) &\le \sum_{x,y \in \Delta_m} \E |X(x) - X(y)|^p \le C_p \sum_{x,y \in \Delta_m} (\E |X(x) - X(y)|^2)^{p/2} \\
& = C_p \sum_{x,y \in \Delta_m} \bigl(C(x,x) + C(y,y) - 2C(x,y)\bigr)^{p/2} \le \hat C_p 2^{(m+2)d - \alpha m p} \\
&\le \hat C_p 2^{-\beta' m p}\;,
\end{equs}
for some constants $\hat C_p$.
(In order to obtain the last inequality, we had to assume that $p \ge {d \over \alpha - \beta'}{m+2 \over m}$ which can always be 
achieved by some value of $p$ independent of $m$
since we assumed that $\beta' < \alpha$.) Using Jensen's inequality, this shows that there exists a constant $K$ such that
the bound
\begin{equ}[e:boundK]
\E K_m(X) \le K 2^{-\beta' m}
\end{equ}
holds uniformly in $m$.
Fix now any two points $x,y \in \CD$ with $x \neq y$ and denote by $m_0$ the 
largest $m$ such that $|x-y| < 2^{-m}$. One can then find sequences $\{x_n\}_{n \ge m_0}$ and $\{y_n\}_{n \ge m_0}$
with the following properties:
\begin{enumerate}
\item One has $\lim_{n \to \infty} x_n = x$ and $\lim_{n \to \infty} y_n = y$.
\item Either $x_{m_0} = y_{m_0}$ or both points belong to the vertices of the same ``dyadic hypercube'' in $\CD_{m_0}$, so that they
 can be connected by at most $d$ ``bonds'' in $\Delta_{m_0}$.
\item For every $n \ge m_0$, $x_n$ and $x_{n+1}$ belong to the vertices of the same ``dyadic hypercube'' in $\CD_{n+1}$, so that they
 can be connected by at most $d$ ``bonds'' in $\Delta_{n+1}$ and similarly for
$(y_n, y_{n+1})$.
\end{enumerate}

\begin{wrapfigure}{r}{3.2cm}
\begin{tikzpicture}
\draw[step=1cm,gray,very thin] (0,0) grid (3,3); 

\fill[blue, opacity=0.2] (0,0) rectangle (2.4,2.2); 
\fill [red] (2.4,2.2) circle (2pt);

\node at (2.4,2.2) [anchor=south west] {$x$};
\node at (2,2) [anchor=north east] {$x_n$};
\foreach \x in {0,1,2,3} {
	\foreach \y in {0,1,2,3} {
		\fill [black] (\x,\y) circle (1.5pt);
		}
	}
\fill [blue] (2,2) circle (2pt);
\end{tikzpicture} \vspace{-1.8cm}
\end{wrapfigure}
One way of constructing this sequence is to order elements in $\CD_m$ by lexicographic order and to choose
$x_n = \max\{\bar x \in \CD_n \,:\, \bar x_j \le x_j \;\forall j\}$, as illustrated in the picture to the right. 
This shows that one has the bound
\begin{equs}
|X(x) - X(y)| &\le |X(x_{m_0}) - Y(x_{m_0})| + \sum_{n = m_0}^\infty |X(x_{n+1})-X(x_{n})|\\
&\qquad + \sum_{n = m_0}^\infty |X(y_{n+1})-X(y_{n})| \\
&\le dK_{m_0}(X) + 2d\sum_{n=m_0}^\infty K_{n+1}(X) \le 2d\sum_{n=m_0}^\infty K_{n}(X)\;.
\end{equs}
Since $m_0$ was chosen in such a way that $|x-y| \ge 2^{-m_0 -1}$, one has the bound
\begin{equs}
M_\beta(X) \le 2d \sup_{m\ge 0} 2^{\beta (m+1)} \sum_{n=m}^\infty K_{n}(X)\le 2^{\beta + 1}d \sum_{n=0}^\infty 2^{\beta n} K_{n}(X)\;.
\end{equs}
It follows from this and from the bound \eref{e:boundK} that 
\begin{equs}
\E |M_\beta(X)| \le 2^{\beta + 1}d \sum_{n=0}^\infty 2^{\beta n} \E K_{n}(X)
\le 2^{\beta+1} dK  \sum_{n=0}^\infty 2^{(\beta - \beta')n} < \infty\;,
\end{equs}
since $\beta'$ was chosen strictly larger than $\beta$.
\end{proof}

Combining Kolmogorov's continuity criterion with Fernique's theorem, we note that we can apply it not only to real-valued
processes, but to any Gaussian Banach-space valued process:

\begin{proposition}\label{prop:KolmogorovGauss}
Let $\CB$ be a separable Banach space and 
let $\{X(x)\}_{x \in [0,1]^d}$ be a collection of $\CB$-valued Gaussian random variables such that
\begin{equ}
\E \|X(x) - X(y)\| \le C |x-y|^\alpha\;,
\end{equ}
for some $C>0$ and some $\alpha \in (0,1]$. Then, there exists a unique Gaussian measure $\mu$ on $\CC([0,1]^d,\CB)$
such that, if $Y$ is a random variable with law $\mu$, then $Y(x)$ is equal in law to $X(x)$ for every $x$. Furthermore,
$\mu\bigl(\CC^\beta([0,1]^d,\CB)\bigr) = 1$ for every $\beta < \alpha$.
\end{proposition}

\begin{proof}
The proof is identical to that of Theorem~\ref{theo:Kolmogorov}, noting that the bound
$\E \|X(x)-X(y)\|^p \le C_p |x-y|^{\alpha p}$ follows from the assumption and Proposition~\ref{prop:universalFernique}.
\end{proof}

\begin{remark}
The space $\CC^\beta([0,1]^d, \R)$ is not separable. However, the space $\CC^\beta_0([0,1]^d, \R)$ of H\"older continuous
functions that furthermore satisfy  $\lim_{y \to x} {|f(x) - f(y)| \over |x-y|^\beta} = 0$ uniformly in $x$ is separable
(polynomials with rational coefficients are dense in it). This is in complete analogy with the fact that the space of bounded
measurable functions is not separable, while the space of continuous functions is. 

It is furthermore possible to check that $\CC^{\beta'} \subset \CC_0^\beta$ for every $\beta' > \beta$,
so that Exercise~\ref{ex:Kolmo} below shows that $\mu$ can actually be realised as a Gaussian measure
on $\CC^\beta_0([0,1]^d, \R)$.
\end{remark}

\begin{exercise}
Given a \textit{compact} metric space $(\CX,d)$, define its $\eps$-entropy by
\begin{equ}
\CH_\eps(\CX) = \log \inf\Big\{N\,:\, \exists x_1,\ldots, x_N \in \CX \;\text{with}\; \bigcup B(x_n,\eps) = \CX\Big\}\;.
\end{equ}
Show that the statement of Kolmogorov's theorem still holds if we replace $[0,1]^d$ by $\CX$, provided 
that $\limsup_{\eps \to 0} \CH_\eps(\CX) / |{\log \eps}| < \infty$.

Using a slightly different union bound to obtain \eqref{e:boundK}, show that if we wish to obtain 
Kolmogorov's theorem for some \textit{fixed} $\beta < \alpha$, then it is sufficient that
$\limsup_{\eps \to 0} \CH_\eps(\CX) \eps^{2(\alpha-\beta)} < \infty$, which also allows for infinite-dimensional
domains $\CX$. 
\end{exercise}


\begin{exercise}
Using Proposition~\ref{prop:trace}, show that, setting $D = [0,1]^d$, if 
$\CH$ is a separable Hilbert space and $C\colon D \times D \to \CL(\CH,\CH)$ is such that $C(x,y)$ is positive definite,
symmetric, and trace class for any two $x,y \in D$, then Kolmogorov's continuity theorem still holds if its
condition is replaced by $\tr C(x,x) + \tr C(y,y) - 2\tr C(x,y) \le K |x-y|^\alpha$. More precisely, one can construct a measure
$\mu$ on the space $\CC^\beta([0,1]^d, \CH)$ with jointly Gaussian marginals such that
\begin{equ}
\int_{\CC^\beta([0,1]^d, \CH)} \scal{h,f(x)}\scal{f(y),k}\,\mu(df) = \scal{h,C(x,y)k}\;,
\end{equ}
for any $x,y \in D$ and $h, k \in \CH$.
\end{exercise}

A very useful consequence of Kolmogorov's continuity criterion is the following result:

\begin{corollary}\label{cor:Kol}
Let $\{\eta_k\}_{k \ge 0}$ be countably many i.i.d.\
standard Gaussian random variables (real or complex).
Moreover let $\{f_k\}_{k \ge 0}\subset \mathrm{Lip}(G,\C)$
where the domain $G\subset \R^d$  is sufficiently regular for Kolomgorov's continuity theorem to hold.
Suppose there is some $\delta\in(0,2)$ such that
\begin{equ}[e:critKol]
S_1^2=\sum_{k\in I} \|f_k\|^2_{L^\infty}<\infty
\quad\mathrm{and}\quad
S_2^2=\sum_{k\in I} \|f_k\|^{2-\delta}_{L^\infty}\mathrm{Lip}(f_k)^\delta<\infty\;,
\end{equ}
and define $f = \sum_{k\in I}  \eta_k f_k$.
Then $f$ is almost surely bounded and H\"older continuous for every H\"older exponent smaller than $\delta / 2$.
\end{corollary}
\begin{proof} From the assumptions we immediately derive that
$f(x)$ and $f(x)-f(y)$ are a centred Gaussian for any $x,y\in G$. Moreover, the
corresponding series converge absolutely. Using that the $\eta_k$ are i.i.d., we
obtain
\begin{equs}
\E |f(x)-f(y)|^2
&= \sum_{k\in I}|f_k(x)-f_k(y)|^2
\le \sum_{k\in I}\min\{2\|f_k\|_{L^\infty}^2, \mathrm{Lip}(f_k)^2|x-y|^2\}\\
&\le 2 \sum_{k\in I} \|f_k\|_{L^\infty}^{2-\delta}
 \mathrm{Lip}(f_k)^\delta  |x-y|^\delta 
 = 2S_2^2  |x-y|^\delta\;,
\end{equs}
where we used that $\min\{a,bx^2\} \le a^{1-\delta/2} b^{\delta/2} |x|^\delta$
for any $a,b\ge0$. The claim now follows from Kolmogorov's continuity theorem.
\end{proof}

\begin{remark}
One should really think of the $f_k$'s in Corollary~\ref{cor:Kol} as being an orthonormal basis
of the Cameron--Martin space of some Gaussian measure. (See Section~\ref{sec:CM} below for the
definition of the Cameron--Martin space associate to a Gaussian measure.) 
The criterion \eref{e:critKol} then provides
an effective way of deciding whether the measure in question can be realised on a space of H\"older
continuous functions. 
\end{remark}

\subsection{The Cameron--Martin space}
\label{sec:CM}

Given a Gaussian measure $\mu$ on a separable Banach space $\CB$, it is possible to associate to
it in a canonical way a Hilbert space $\CH_\mu \subset \CB$, called the \index{Cameron--Martin!space}Cameron--Martin space of $\mu$.
The main importance of the Cameron--Martin space is that it characterises precisely those directions in $\CB$
in which translations leave the measure $\mu$ ``quasi-invariant'' in the sense that the translated measure
has the same null sets as the original measure. In general, the space $\CH_\mu$ will turn out to be strictly
smaller than $\CB$. Actually, this is always the case as soon as $\dim \CH_\mu = \infty$ and, even worse, we will see that
in this case one necessarily has $\mu(\CH_\mu) = 0$! Contrast this to the case
of finite-dimensional Lebesgue measure which is invariant under translations in any direction! This is a striking
illustration of the fact that measures in infinite-dimensional spaces have a strong tendency of being mutually singular.

The definition of the Cameron--Martin space is the following, where we postpone to Remark~\ref{rem:unique} and
Proposition~\ref{prop:CMbig} the verification that $\|h\|_\mu$ is well-defined and that $\|h\|_\mu > 0$ for $h \neq 0$:

\begin{definition}
The Cameron--Martin space $\CH_\mu$ of $\mu$ 
is the completion of the linear subspace $\HH_\mu \subset \CB$ defined by
\begin{equ}
\HH_\mu = \{ h \in \CB \,:\, \exists\, h^* \in \CB^*\;\text{with}\; C_\mu(h^*, \ell) = \ell(h)\;\forall \ell \in \CB^*\}\;,
\end{equ}
under the \index{Cameron--Martin!norm}norm $\|h\|_\mu^2 = \scal{h,h}_\mu = C_\mu(h^*, h^*)$.
It is a Hilbert space when endowed with the scalar product $\scal{h,k}_\mu = C_\mu(h^*, k^*)$.
\end{definition}

\begin{exercise}
Show that the space $\HH_\mu$ is nothing but the range of the operator 
$\hat C_\mu$ defined in Remark~\ref{rem:Chat} and that $\hat C_\mu h^* = h$.
\end{exercise}

\begin{remark}\label{rem:unique}
Even though the map $h \mapsto h^*$ may not be one to one, the norm $\|h\|_\mu$ is well-defined.
To see this, assume that for a given $h \in \HH_\mu$, there are two corresponding elements $h_1^*$ and $h_2^*$ in $\CB^*$.
Then, defining $k = h_1^* + h_2^*$, one has
\begin{equ}
C_\mu(h_1^*, h_1^*) - C_\mu(h_2^*, h_2^*) = C_\mu(h_1^*, k) - C_\mu(h_2^*, k) = k(h) - k(h) = 0\;,
\end{equ}
showing that $\|h\|_\mu$ does indeed not depend on the choice of $h^*$.
\end{remark}

\begin{exercise}\label{ex:CMWiener}
The \index{Wiener!measure}Wiener measure $\mu$ is defined on $\CB = \CC([0,1],\R)$ as the centred Gaussian measure with covariance
operator given by $C_\mu(\delta_s,\delta_t) = s \wedge t$.
Show that the Cameron--Martin space for the Wiener measure on $\CB = \CC([0,1],\R)$
is given by the space $H^{1,2}_0([0,1])$ of all absolutely continuous functions $h$ such that 
$h(0) = 0$ and $\int_0^1 \dot h^2(t)\,dt < \infty$.
\end{exercise}

\begin{exercise}
Show that in the case $\CB= \R^n$, the Cameron--Martin space is given by the range of the covariance matrix.
Write an expression for $\|h\|_\mu$ in this case.
\end{exercise}

\begin{exercise}
Show that the Cameron--Martin space of a Gaussian measure determines it. More precisely, if $\mu$ and $\nu$ are two Gaussian
measures on $\CB$ such that $\CH_\mu = \CH_\nu$ and such that $\|h\|_\mu = \|h\|_\nu$ for every $h \in \CH_\mu$,
then they are identical. 

For this reason, a Gaussian measure on $\CB$ is sometimes given by specifying the Hilbert space structure $(\CH_\mu, \|\cdot\|_\mu)$.
Such a specification is then usually called an \textit{abstract Wiener space}.\index{Wiener!space (abstract)}
\end{exercise}

Let us discuss a few properties of the Cameron--Martin space. First of all, we show that it is a
subspace of $\CB$ despite the completion procedure and that all non-zero elements of $\CH_\mu$ have strictly
positive norm:

\begin{proposition}\label{prop:CMbig}
One has $\CH_\mu \subset \CB$. Furthermore, one has the bound
\begin{equ}[e:embedding]
\scal{h,h}_\mu \ge \|C_\mu\|^{-1} \|h\|^2\;,
\end{equ}
where the norms on the right hand side are understood to be taken in $\CB$. 
\end{proposition}

\begin{proof}
One has the chain of inequalities
\begin{equ}
\|h\|^2 = \sup_{\ell \in \CB^* \setminus \{0\}} {\ell(h)^2 \over \|\ell\|^2} = \sup_{\ell \in \CB^* \setminus \{0\}} {C_\mu(h^*, \ell)^2 \over \|\ell\|^2}
\le \sup_{\ell \in \CB^* \setminus \{0\}} {C_\mu(h^*, h^*)\, C_\mu(\ell, \ell) \over \|\ell\|^2}
\le \|C_\mu\| \scal{h,h}_\mu\;,
\end{equ}
which yields the bound on the norms. It follows immediately that the canonical inclusion 
$\iota\colon \HH_\mu \hookrightarrow \CB$ extends to a bounded linear map $\iota \colon \CH_\mu \to \CB$.
To show that $\CH_\mu$ is a subset of $\CB$ (or rather that it can be interpreted as such),
it then remains to show that $\iota$ is injective on $\CH_\mu$. 

Assume that it is not, so that there exists a sequence
$f_n \in \HH_\mu$ with $\|f_n\| \to 0$ but $f_n \to f \neq 0$ in $\CH_\mu$. 
For any $h \in \HH_\mu$, we then have
\begin{equ}
|\scal{h,f_n}_\mu| = |h^*(f_n)| \le \|h^*\|\, \|f_n\| \to 0\;,
\end{equ}
where we used the fact that $h^* \in \CB^*$ (and in particular its norm $\|h^*\|$ is finite)
by the definition of $\HH_\mu$. It follows that one necessarily has 
$\scal{h,f}_\mu = 0$ for every $h \in \HH_\mu$ and, since
$\HH_\mu$ is dense in $\CH_\mu$, it follows that $f = 0$.
\end{proof}

\begin{remark}
The reader familiar with functional analysis will surely have remarked the analogy
between the proof of injectivity of $\iota$ and the proof of the fact that symmetric
operators on a Hilbert space are always closable. In general, given a linear subspace 
$\CE \subset \CB$ and a norm $|\cdot|$ on $\CE$ with $|h| \ge \|h\|$
it may be the case that the completion of $\CE$ under $|\cdot|$ cannot be interpreted naturally
as a subspace of $\CB$. Think for example of the case $\CB = L^2([-1,1])$, $\CE
= \CC([-1,1])$, and $|h| = \|h\| + |h(0)|$. In this case, the completion of $\CE$
is canonically isomorphic to $\CB \oplus \R$ but has no natural embedding into $\CB$.
\end{remark}

A simple example showing that the correspondence $h \mapsto h^*$ in the definition of $\HH_\mu$ is not necessarily
unique is the case $\mu = \delta_0$, so that $C_\mu = 0$. If one chooses $h = 0$, then any $h^* \in \CB$ has
the required property that $C_\mu(h^*, \ell) = \ell(h)$, so that this is an extreme case of non-uniqueness. 
However, if we view $\CB^*$ as a subset of $L^2(\CB,\mu)$ (by identifying
linear functionals that agree $\mu$-almost surely), then the correspondence $h \mapsto h^*$ is always an isomorphism. 
One has indeed
$\int_\CB h^*(x)^2\,\mu(dx) = C_\mu(h^*,h^*) = \|h\|_\mu^2$. In particular, if $h_1^*$ and $h_2^*$ are two distinct elements of
$\CB^*$ associated to the same element $h \in \CB$, then $h_1^* - h_2^*$ is associated to the element $0$ and therefore
$\int_\CB \bigl(h_1^*-h_2^*\bigr)^2(x)\,\mu(dx) = 0$, showing that $h_1^* = h_2^*$ as elements of $L^2(\CB,\mu)$. 
We have:
\begin{proposition}\label{prop:defRmu}
There is a canonical isomorphism $\iota\colon h \mapsto h^*$ between $\CH_\mu$ and the closure $\CR_\mu$ of $\CB^*$ in $L^2(\CB,\mu)$. 
In particular, $\CH_\mu$ is separable.
\end{proposition}

\begin{proof}
We have already shown that $\iota \colon \CH_\mu \to L^2(\CB,\mu)$ is an isomorphism onto its image, so it remains to show that
all of $\CB^*$ belongs to the image of $\iota$. For $h \in \CB^*$, define $h_* \in \CB$ as in \eref{e:defCmuhat} by
\begin{equ}
h_* = \int_\CB x\, h(x)\,\mu(dx)\;.
\end{equ}
This integral converges since $\|x\|^2$ is integrable by Fernique's theorem.
Since one has the identity $\ell(h_*) = C_\mu(\ell,h)$, it follows that $h_* \in \HH_\mu$ and
$h = \iota(h_*)$, as required to conclude the proof.

The separability of $\CH_\mu$ then follows immediately from the fact that $L^2(\CB,\mu)$ is separable whenever $\CB$
is separable, since its Borel $\sigma$-algebra is countably generated.
\end{proof}

\begin{remark}
The space $\CR_\mu$ is called the \index{reproducing kernel}\textit{reproducing kernel Hilbert space} for $\mu$ (or just \textit{reproducing kernel} for short). However, since it is isomorphic
to the Cameron--Martin space in a natural way, there is considerable confusion between the two in the literature.
We retain in these notes the terminology from \cite{Bog98Gauss}, but we urge the reader to
keep in mind that there are authors who use a slightly different terminology.
\end{remark}

\begin{remark}
In general, there do exist Gaussian measures with non-separable Cameron--Martin space, but they are measures on more general
vector spaces. One example would be the measure on $\R^\R$ (yes, the space of \textit{all} functions 
from $\R$ to $\R$ endowed with the product $\sigma$-algebra) 
given by the uncountable product of one-dimensional
Gaussian measures. The Cameron--Martin space for this somewhat pathological measure is given by those functions $f$ that are non-zero
on at most countably points and such that $\sum_{t \in \R} |f(t)|^2 < \infty$. This is a prime example of a non-separable Hilbert space.
\end{remark}

\begin{exercise}\label{ex:GaussK}
Let $\mu$ be a Gaussian measure on a Hilbert space $\CH$ with covariance $K$ and consider the spectral
decomposition of $K$: $K e_n = \lambda_n e_n$ with $\sum_{n \ge 1} \lambda_n < \infty$ and $\{e_n\}$ an 
orthonormal basis of eigenvectors. Such a decomposition exists since we already know that $K$ must be
 trace class from Proposition~\ref{prop:trace}.

Assume now that $\lambda_n > 0$ for
every $n$. Show that $\HH_\mu$
is given by the range of $K$ and that the correspondence
$h \mapsto h^*$ is given by $h^* = K^{-1} h$. Show furthermore that the Cameron--Martin
space $\CH_\mu$ consists of those elements $h$ of $\CH$ such that $\sum_{n \ge 1} \lambda_n^{-1} \scal{h, e_n}^2 < \infty$
and that $\scal{h,k}_\mu = \scal{K^{-1/2} h, K^{-1/2} k}$.
\end{exercise}

\begin{exercise}\label{ex:charnorm}
Show that one has the alternative characterisation
\begin{equ}[e:CMnorm]
\|h\|_\mu = \sup \{\ell(h)\,:\, C_\mu(\ell,\ell) \le 1\}\;,
\end{equ}
and $\CH_\mu = \{h \in \CB\,:\, \|h\|_\mu < \infty\}$.
\textbf{Hint:} For the second statement, show first that if $\|h\|_\mu < \infty$, then 
$h$ determines a bounded linear functional on $\CR_\mu$. Then use Riesz's representation theorem,
the isometry of Proposition~\ref{prop:defRmu}, and the fact that $\CB^*$ is dense in $\CR_\mu$.
\end{exercise}

Since elements in $\CR_\mu$ are built from the space of all bounded linear functionals
on $\CB$, it should come as little surprise that its elements are ``almost'' linear functionals on $\CB$ in the following sense:

\begin{proposition}\label{prop:linearfunctional}
For every $\ell \in \CR_\mu$ there exists a measurable linear subspace $V_\ell$ of $\CB$ such that $\mu(V_\ell) = 1$ and a linear map
$\hat \ell\colon V_\ell \to \R$ such that $\ell = \hat \ell$ $\mu$-almost surely.
\end{proposition}
\begin{proof}
Fix $\ell \in \CR_\mu$. By the definition of $\CR_\mu$ and Borel-Cantelli, we can find a sequence $\ell_n \in \CB^*$ such that 
$\lim_{n\to \infty} \ell_n(x) = \ell(x)$ for $\mu$-almost every $x \in\CB$. (Take for example $\ell_n$ such that $\|\ell_n - \ell\|_\mu^2 \le n^{-4}$.)
 It then suffices to define 
\begin{equ}
V_\ell = \Bigl\{x\,:\, \lim_{n \to \infty} \ell_n(x) \;\hbox{exists}\Bigr\}\;,
\end{equ}
and to set $\hat \ell(x) = \lim_{n \to \infty} \ell_n(x)$ on $V_\ell$.
\end{proof}

Another very useful fact about the reproducing kernel space is given by:

\begin{proposition}\label{prop:GaussRK}
The law of any element $h^* = \iota(h) \in \CR_\mu$ is a centred Gaussian with variance $\|h\|_\mu^2$. Furthermore,
any two elements $h^*, k^*$ have covariance $\scal{h,k}_\mu$.
\end{proposition}
\begin{proof}
We already know from the definition of a Gaussian measure 
that the law of any element of $\CB^*$ is a centred Gaussian. Let now $h^*$ be any element of $\CR_\mu$ and 
let $h_n$ be a sequence in $\CR_\mu \cap \CB^*$ such that $h_n \to h^*$ in $\CR_\mu$.
We can furthermore choose this approximating sequence such that $\|h_n\|_{\CR_\mu} = \|h^*\|_{\CR_\mu} = \|h\|_{\mu}$,
so that the law of each of the $h_n$ is equal to $\CN(0,\|h\|_\mu^2)$.

Since $L^2$-convergence implies convergence in law, we conclude that the law of $h^*$ is also given by $\CN(0,\|h\|_\mu^2)$.
The statement about the covariance then follows by polarisation\index{polarisation}, since
\begin{equ}
\E h^* k^* = {1\over 2} \bigl(\E (h^* + k^*)^2 - \E (h^*)^2 - \E (k^*)^2\bigr) = {1\over 2} \bigl(\|h+k\|_\mu^2 - \|h\|_\mu^2 - \|k\|_\mu^2\bigr) = \scal{h,k}_\mu\;,
\end{equ}
by the previous statement.
\end{proof}

\begin{remark}
Actually, the converse of Proposition~\ref{prop:linearfunctional} is also true: if $\ell\colon \CB \to \R$ is measurable and linear on
a measurable linear subspace $V$ of full measure, then $\ell$ belongs to 
$\CR_\mu$. This is not an obvious statement. It can be viewed 
for example as a consequence of the highly non-trivial
fact that every Borel measurable linear map between two sufficiently ``nice'' topological vector spaces (separable Banach will do) is continuous, 
see for example \cite{SchwaMeas,KatsMeas}. (The point here is that the map must be linear on the whole space
and not just on some ``large'' subspace as is usually the case with unbounded operators.)
This implies by Proposition~\ref{prop:carH} that $\ell$ is a measurable linear extension of some bounded
linear functional on $\CH_\mu$. Since such
extensions are unique (up to null sets) by Theorem~\ref{theo:uniqueext} below, the claim follows from 
Proposition~\ref{prop:defRmu}. 
\end{remark}

\begin{exercise}\label{ex:Kolmo}
Show that if $\tilde \CB \subset \CB$ is a continuously embedded Banach space with $\mu(\tilde \CB) = 1$, then 
the embedding $\CB^* \hookrightarrow \CR_\mu$ extends to an embedding $\tilde \CB^* \hookrightarrow \CR_\mu$. Deduce
from this that the restriction of $\mu$ to $\tilde \CB$ is again a Gaussian measure. In particular, Kolmogorov's continuity
criterion yields a Gaussian measure on $\CC_0^\beta([0,1]^d, \R)$.
\end{exercise}

The properties of the reproducing kernel space of a Gaussian measure allow us to give another illustration of the fact that
measures on infinite-dimensional spaces behave in a rather different way from measures on $\R^n$:

\begin{proposition}\label{prop:singDilat}
Let $\mu$ be a centred Gaussian measure on a separable Banach space $\CB$ such that $\dim \CH_\mu = \infty$. Denote by
$D_c$ the dilatation by a real number $c$ on $\CB$, that is $D_c(x) = cx$. Then, $\mu$  and $D_c^* \mu$ are 
mutually singular for every $c \neq \pm 1$. 
\end{proposition}
\begin{proof}
Since the reproducing
Kernel space $\CR_\mu$ is a separable Hilbert space, we can find an orthonormal basis $\{e_n\}_{n\ge 0}$.
Consider the sequence of random variables $X_N(x) = {1\over N}\sum_{n=1}^N |e_n(x)|^2$ over $\CB$.
If $\CB$ is equipped with the measure $\mu$ then, since the $e_n$ are independent under $\mu$, we can apply
 the law of large numbers and deduce that 
\begin{equ}[e:conv1]
\lim_{N \to \infty} X_N(x) = 1\;,
\end{equ}
for $\mu$-almost every $x$. On the other hand, it follows from the linearity of the $e_n$ that when we equip $\CB$ with the measure
$D_c^*\mu$, the $e_n$ are still independent, but have variance $c^2$, so that
\begin{equ}
\lim_{N \to \infty} X_N(x) = c^2\;,
\end{equ}
for $D_c^*\mu$-almost every $x$. This shows that if $c \neq \pm 1$,  the set on which the convergence \eref{e:conv1}
takes place must be of $D_c^*\mu$-measure $0$, which implies that $\mu$ and $D_c^*\mu$ are mutually singular.
\end{proof}

As already mentioned earlier, the importance of the Cameron--Martin space is that it represents precisely those directions in which
one can translate the measure $\mu$ without changing its null sets:

\index{Cameron--Martin!theorem}
\begin{theorem}[Cameron--Martin]\label{theo:CM}
For $h \in \CB$, define the map $T_h \colon \CB \to \CB$ by $T_h(x) = x+h$. Then, the measure $T_h^*\mu$ is absolutely continuous
with respect to $\mu$ if and only if $h \in \CH_\mu$.
\end{theorem}

\begin{proof}
Fix $h \in \CH_\mu$ and let $h^* \in L^2(\CB,\mu)$ be the corresponding element of the reproducing kernel. 
Since the law of $h^*$ is Gaussian by Proposition~\ref{prop:GaussRK}, the map $x \mapsto \exp(h^*(x))$ is integrable.
Since furthermore the variance of $h^*$ is given by $\|h\|_\mu^2$, the function
\begin{equ}[e:CMformula]
\CD_h(x) = \exp \bigl(h^*(x) - \hf  \|h\|_\mu^2\bigr)
\end{equ}
is strictly positive, belongs to $L^1(\CB,\mu)$, and integrates to $1$. It is therefore the Radon-Nikodym derivative of a measure
$\mu_h$ that is absolutely continuous with respect to $\mu$. To check that one has indeed $\mu_h = T_h^*\mu$, it suffices
to show that their Fourier transforms coincide. Assuming that $h^* \in \CB^*$, one has
\begin{equs}
\hat \mu_h(\ell) &= \int_\CB \exp \bigl(i\ell(x) + h^*(x) - \hf  \|h\|_\mu^2\bigr)\,\mu(dx) 
 = \exp \bigl(\hf C_\mu(i\ell + h^*, i\ell + h^*) - \hf  \|h\|_\mu^2\bigr) \\
& = \exp \bigl(- \hf C_\mu(\ell, \ell) - i C_\mu(\ell, h^*)\bigr) = \exp \bigl(- \hf C_\mu(\ell, \ell) + i \ell(h)\bigr)\;.
\end{equs}
Using Proposition~\ref{prop:GaussRK} for the joint law of $\ell$ and $h^*$, it is an easy exercise to check that this 
equality still holds for arbitrary $h \in \CH_\mu$.

On the other hand, we have
\begin{equs}
\widehat {T_h^*\mu}(\ell) &=  \int_\CB \exp \bigl(i\ell(x)\bigr)\, T_h^*\mu(dx) =  \int_\CB \exp \bigl(i\ell(x + h)\bigr)\,\mu(dx) = e^{i\ell(h)} \int_\CB \exp \bigl(i\ell(x)\bigr)\,\mu(dx) \\
& =  \exp \bigl(- \hf C_\mu(\ell, \ell) + i \ell(h)\bigr)\;,
\end{equs}
showing that $\mu_h = T_h^*\mu$.

To show the converse, note first that one can check by an explicit calculation that
 $\|\CN(0,1) - \CN(h,1)\|_\TV \ge 2 - 2\exp(-{h^2 \over 8})$. Fix now some arbitrary $n > 0$.
 If $h \not \in \CH_\mu$ then, by Exercise~\ref{ex:charnorm}, there exists $\ell \in \CB^*$ with $C_\mu(\ell, \ell) = 1$ such that $\ell(h) \ge n$.
Since the image $\ell^*\mu$ of $\mu$ under $\ell$ is $\CN(0,1)$ and the image of $T_h^*\mu$ under $\ell$ is $\CN(-\ell(h),1)$,
this shows that
\begin{equ}
\|\mu - T_h^*\mu\|_\TV \ge \|\ell^*\mu - \ell^*T_h^*\mu\|_\TV = \|\CN(0,1) - \CN(-\ell(h),1)\|_\TV  \ge 2 - 2\exp(-{n^2 \over 8})\;.
\end{equ} 
Since this is true for every $n$, we conclude that $\|\mu - T_h^*\mu\|_\TV = 2$, thus showing that they are mutually singular.
\end{proof}

As a consequence, we have the following characterisation of the Cameron--Martin space

\begin{proposition}\label{prop:carH}
The space $\CH_\mu \subset \CB$ is the intersection of all (measurable) linear subspaces  of full measure. However, if $\CH_\mu$
is infinite-dimensional, then one has $\mu(\CH_\mu) = 0$.
\end{proposition}

\begin{proof}
Take an arbitrary linear subspace $V \subset \CB$ of full measure and take an arbitrary $h \in \CH_\mu$. It follows from
Theorem~\ref{theo:CM} that the affine space $V-h$ also has full measure. Since $(V-h) \cap V = \emptyset$ unless $h\in V$,
one must have $h \in V$, so that $\CH_\mu \subset \bigcap\{V\subset \CB\,:\, \mu(V) = 1\}$. 

Conversely, take an arbitrary $x \not \in \CH_\mu$ and let us construct a linear space $V \subset \CB$ of full measure,
but not containing $x$. 
Since $x \not \in \CH_\mu$, one has $\|x\|_\mu = \infty$ with $\|\cdot\|_\mu$ extended to $\CB$ as in
\eref{e:CMnorm}. Therefore, we can find a sequence $\ell_n \in \CB^*$ such that $C_\mu(\ell_n, \ell_n) \le 1$ and
$\ell_n(x) \ge n$. Defining the norm $|y|^2 = \sum_n n^{-2} \bigl(\ell_n(y)\bigr)^2$, we see that
\begin{equ}
\int_\CB |y|^2\,\mu(dy) =  \sum_{n=1}^\infty {1\over n^2} \int_\CB \bigl(\ell_n(y)\bigr)^2\,\mu(dy) \le {\pi^2 \over 6}\;,
\end{equ}
so that the linear space $V = \{y\,:\, |y| < \infty\}$ has full measure. However, $|x| = \infty$ by construction, so that $x \not \in V$.

To show that $\mu(\CH_\mu) = 0$ if $\dim \CH_\mu = \infty$, consider an orthonormal sequence $e_n \in \CR_\mu$ so that
the random variables $\{e_n(x)\}$ are i.i.d.\ $\CN(0,1)$ distributed. By Exercise~\ref{ex:charnorm} and the second Borel-Cantelli lemma, it follows that
$\|x\|_\mu \ge \sup_{n} |e_n(x)| = \infty$ for $\mu$-almost every $x$, so that 
the claim follows.
\end{proof}

\begin{exercise}
Recall that the (topological) support $\supp \mu$ of a Borel measure on a complete separable metric space 
consists of those points $x$ such that 
$\mu(U) > 0$ for every neighbourhood $U$ of $x$.
Show that, if $\mu$ is a Gaussian measure, then its support is the closure $\bar \CH_\mu$ of $\CH_\mu$ in $\CB$.
\end{exercise}

\subsection{Images of Gaussian measures}

It follows immediately from the definition of a Gaussian measure and the
expression for its Fourier transform that if $\mu$ is a Gaussian measure on some Banach space $\CB$
and $A \colon \CB \to \hat\CB$ is a bounded linear map for $\hat\CB$ some other Banach space,
then $\nu = A^*\mu$ is a Gaussian measure on $\hat\CB$ with covariance
\begin{equ}
C_\nu(\ell, \ell') = C_\mu(A^* \ell, A^*\ell')\;,
\end{equ}
where $A^* \colon \hat\CB^* \to \CB^*$ is the adjoint\index{adjoint!operator} to $A$, that is the operator such that $ \bigl(A^*\ell\bigr)(x) = \ell(Ax)$ for
every $x \in \CB$ and every $\ell \in \hat\CB^*$.

Recall now that $\CH_\mu$ is the intersection over all linear subspaces of $\CB$ that have full measure under $\mu$.
This suggests that in order to determine the image of $\mu$ under a linear map, it is sufficient to know how that map acts
on elements of $\CH_\mu$. This intuition is made precise by the following theorem:

\begin{theorem}\label{theo:imageGaussian}
Let $\mu$ be a centred Gaussian probability measure on a separable Banach space $\CB$. Let furthermore $\CH$ be
a separable Hilbert space and let $A \colon \CH_\mu \to \CH$ be a Hilbert-Schmidt operator. (That is
$AA^*\colon \CH \to \CH$ is trace class.) Then, there exists a measurable map $\hat A \colon \CB \to \CH$ such that
$\nu = \hat A^*\mu$ is Gaussian with covariance $C_\nu(h,k) = \scal{A^*h, A^*k}_\mu$. 
Furthermore, there exists a measurable linear subspace $V \subset \CB$ of full $\mu$-measure such that $\hat A$ restricted to $V$ is linear
and $\hat A$ restricted to $\CH_\mu \subset V$ agrees with $A$. 
\end{theorem}

\begin{proof}
Let $\{e_n\}_{n \ge 1}$ be an orthonormal basis for $\CH_\mu$ and denote by $e_n^*$ the corresponding elements in
$\CR_\mu \subset L^2(\CB, \mu)$ and define $S_N(x) = \sum_{n = 0}^N e_n^*(x) Ae_n$. 
Recall from Proposition~\ref{prop:linearfunctional} that we can find subspaces
$V_{e_n}$ of full measure such that $e_n^*$ is linear on $V_{e_n}$. Define now a linear subspace $V \subset \CB$ by
\begin{equ}
V = \Bigl\{x \in \bigcap_{n \ge 0} V_{e_n}\,:\, \text{the sequence $\{S_N(x)\}$ converges in $\CH$}\Bigr\}\;,
\end{equ}
(the fact that $V$ is linear follows from the linearity of each of the $e_n^*$) and set
\begin{equ}
\hat A(x) = \left\{\begin{array}{cl}  \lim_{N\to\infty} S_N(x) & \text{for $x\in V$,} \\ 0 & \text{otherwise.} \end{array}\right.
\end{equ}
Since the random variables $\{e_n^*\}$ are i.i.d.\ $\CN(0,1)$-distributed under $\mu$, the sequence $\{S_N\}$ forms 
an $\CH$-valued martingale and one has
\begin{equ}
\sup_N \E_\mu \|S_N(x)\|^2 = \sum_{n = 0}^\infty \|A e_n\|^2 \le \tr A^* A < \infty\;,
\end{equ}
where the last inequality is a consequence of $A$ being Hilbert-Schmidt. It follows that $\mu(V) = 1$ by Doob's martingale convergence theorem.

To see that $\nu = \hat A^*\mu$ has the stated property,
fix an arbitrary $h\in \CH$ and note that the series $\sum_{n \ge 1} e_n^*\scal{Ae_n, h}$ converges in $\CR_\mu$ 
to an element with covariance $\|A^* h\|^2$. The statement then follows from  Proposition~\ref{prop:GaussRK} and
the fact that $C_\nu(h,h)$ determines $C_\nu$ by polarisation\index{polarisation}. To check that $\nu$ is Gaussian, we can compute its
Fourier transform in a similar way.
\end{proof}

\begin{remark}
In fact, we will show in Theorem~\ref{theo:uniqueext} below that the extension
$\hat A$ in Theorem~\ref{theo:imageGaussian} is unique up to sets of $\mu$-measure $0$.
\end{remark}

The proof of Theorem~\ref{theo:imageGaussian} can easily be extended to the case where the image space is a Banach space
rather than a Hilbert space. However, in this case we cannot give a straightforward characterisation of those maps $A$ that are
`admissible', since we have no good complete characterisation of covariance operators for Gaussian measures on Banach spaces.
However, we can take the pragmatic approach and simply assume that the new covariance determines a Gaussian measure on the
target Banach space. With this approach, we can formulate the following version for Banach spaces:

\begin{proposition}\label{prop:imageGaussian}
Let $\CB_1$ and $\CB_2$ be two separable Banach spaces and let $\mu$ be a centred Gaussian probability measure on $\CB_1$.
Let $A \colon \CH_\mu \to \CB_2$ be a bounded linear operator such that there exists a centred
Gaussian measure $\nu$ on $\CB_2$ with covariance $C_\nu(h,k) = \scal{A^*h, A^*k}_\mu$.
Then, there exists a measurable map $\hat A \colon \CB_1 \to \CB_2$ such that
$\nu = \hat A^*\mu$ and such that  there exists a measurable linear subspace 
$V \subset \CB$ of full $\mu$-measure such that $\hat A$ restricted to $V$ is linear
and $\hat A$ restricted to $\CH_\mu \subset V$ agrees with $A$. 
\end{proposition}

\begin{proof}
As a first step, we construct a Hilbert space $\CH_2$ such that $\CB_2 \subset \CH_2$ as a Borel subset. Denote by
$\CH_\nu \subset \CB_2$ the Cameron--Martin space of $\nu$ and let $\{e_n\} \subset \CH_\nu$ be an orthonormal basis of elements
such that $e_n^* \in \CB_2^*$ for every $n$. (Such an orthonormal basis can always be found by using the Grahm-Schmidt procedure.)
We then define a norm on $\CB_2$ by
\begin{equ}
\|x\|_2^2 = \sum_{n \ge 1} {e_n^*(x)^2 \over n^2 \|e_n^*\|^2}\;,
\end{equ}
where $\|e_n^*\|$ is the norm of $e_n^*$ in $\CB_2^*$. It is immediate that $\|x\|_2 <\infty$ for every $x \in \CB_2$, so that this turns $\CB_2$ into
a pre-Hilbert space.
We finally define $\CH_2$ as the completion of $\CB_2$ under $\|\cdot\|_2$.

Denote by $\nu'$ the image of the measure $\nu$ under the inclusion map $\iota\colon \CB_2 \hookrightarrow \CH_2$.
It follows that the map $A' = \iota \circ A$ satisfies the assumptions of Theorem~\ref{theo:imageGaussian},
so that there exists a map $\hat A \colon \CB_1 \to \CH_2$ which is linear on a subset of full $\mu$-measure and such that $\hat A^* \mu = \nu'$.
On the other hand, we know by construction that $\nu'(\CB_2) = 1$, so that the set $\{x\,:\, \hat A x \in \CB_2\}$ is of full measure. Modifying
$\hat A$ outside of this set by for example setting it to $0$ and using Exercise~\ref{ex:Kolmo} then yields the required statement.
\end{proof}

\begin{exercise}
In the context of Proposition~\ref{prop:imageGaussian}, show that the 
Cameron--Martin space $\CH_\nu$ for $\nu$ is given by $\CH_\nu = \range \CH_\mu$ 
with $\|h\|_\nu = \inf\{\|u\|_\mu\,:\, Au = h\}$.
\end{exercise}

\subsubsection{Uniqueness of measurable extensions and the isoperimetric inequality}
\label{sec:uniqueext}

This section is devoted to a proof of the converse of Theorem~\ref{theo:imageGaussian} and Proposition~\ref{prop:imageGaussian},
namely

\begin{theorem}\label{theo:uniqueext}
Let $\mu$ be a Gaussian measure on a separable Banach space $\CB_1$ with Cameron--Martin space $\CH_\mu$ and let
$A\colon \CH_\mu \to \CB_2$ be a linear map
satisfying the assumptions of Proposition~\ref{prop:imageGaussian}. Then the linear measurable extension $\hat A$ of $A$ 
is unique, up to sets of  measure $0$.
\end{theorem}

\begin{remark}
As a consequence of this result, the precise Banach spaces $\CB_1$ and $\CB_2$ are completely irrelevant when one considers the
image of a Gaussian measure under a linear transformation. The only thing that matters is the Cameron--Martin space for the starting measure
and the way in which the linear transformation acts on this space. This fact will be used repeatedly in the sequel.
\end{remark}

This is probably one of the most remarkable results in Gaussian measure theory. At first sight, it appears completely
counterintuitive: the Cameron--Martin space $\CH_\mu$ has measure $0$, so how can the specification of a 
measurable map on a set of measure $0$ be sufficient to determine it on a set of measure $1$? Part of the answer 
lies of course in the requirement that the extension $\hat A$ should be linear on a set of full measure. However,
even this requirement would not be sufficient by itself to determine $\hat A$ since the Hahn--Banach theorem 
provides a huge number of different extension of $A$ that do not coincide anywhere except on $\CH_\mu$. 
The missing ingredient that solves this mystery is the requirement that $\hat A$ be
not just any linear map, but a \textit{measurable} linear map. This additional constraint rules out all of the 
non-constructive extensions of $A$ provided by the Hahn--Banach theorem and leaves only one (constructive)
extension of $A$.

The main ingredient in the proof of Theorem~\ref{theo:uniqueext} is the Borell--Sudakov--Cirel'son inequality \index{Borell--Sudakov--Cirel'son inequality|see{isoperimetric inequality}}
\cite{SudTsi,Borell}, 
a general form of isoperimetric inequality\index{isoperimetric inequality} for Gaussian measures which is very interesting and useful in its own right.
In order to state this result, we first introduce the notation $B_\eps$ for the $\CH_\mu$-ball of radius $\eps$ centred at the origin.
We also denote by $A+B$ the sum of two sets defined by
\begin{equ}
A+B = \{x+y\,:\, x \in A\;,\quad  y \in B\}\;,
\end{equ}
and we denote by $\Phi$ the cumulative distribution function of the normal Gaussian: $\Phi(t) = {1\over \sqrt{2\pi}} \int_{-\infty}^t e^{-s^2/2}\,ds$.
With these notations at hand, we have the following:

\begin{theorem}[Borell--Sudakov--Cirel'son]\label{theo:Borell}
Let $\mu$ be a Gaussian measure on a separable Banach space $\CB$ with Cameron--Martin space $\CH_\mu$ and let $A\subset \CB$
be a measurable subset with measure $\mu(A) = \Phi(\alpha)$ for some $\alpha \in \R$. Then, for every $\eps > 0$,
one has the bound $\mu(A + B_\eps) \ge \Phi(\alpha + \eps)$.
\end{theorem}

\begin{remark}
Theorem~\ref{theo:Borell} is remarkable since it implies that even though $\CH_\mu$ itself has measure $0$, whenever $A$ is a set of positive 
measure, no matter how small, the set $A + \CH_\mu$ has full measure!
\end{remark}

\begin{remark}
The bound given in Theorem~\ref{theo:Borell} is sharp whenever $A$ is a half space, in the sense that 
$A = \{x \in \CB \,:\, \ell(x) \ge c\}$ for some $\ell \in \CR_\mu$ and $c \in \R$. In the case where $\eps$ is small, $(A + B_\eps) \setminus A$ is a fattened boundary for the set $A$, so that
$\mu(A + B_\eps) - \mu(A)$ can be interpreted as a kind of ``perimeter'' for $A$.
The statement can then be interpreted as stating that in the context of Gaussian measures,
half-spaces are the sets of given perimeter that have the largest measure. This 
justifies the statement that Theorem~\ref{theo:Borell} is an isoperimetric inequality.
\end{remark}

We are not going to give a proof of Theorem~\ref{theo:Borell} in these notes because this would lead us too far astray
from our main object of study. The interested reader may want to look into the monograph \cite{LedTal} for a more exhaustive treatment
of probability theory in Banach spaces in general and isoperimetric inequalities in particular.
Let us nevertheless remark shortly on how the argument of the proof goes, as it can be found in the original papers \cite{SudTsi,Borell}. 
In a nutshell, it is a consequence of the two following remarks:
\begin{claim}
\item Let $\nu_M$ be the uniform measure on a sphere of radius $\sqrt M$ in $\R^M$ and let $\Pi_{M,n}$ be the orthogonal
projection from $\R^M$ to $\R^n$. Then, the sequence of measures $\Pi_{M,n}\nu_M$ converges as $M \to \infty$ to the standard 
Gaussian measure on $\R^n$. This remark is originally due to Poincar\'e.
\item A claim similar similar to that of Theorem~\ref{theo:Borell} holds for the uniform measure on the sphere, in the sense that 
the volume of a fattened set $A + B_\eps$ on the sphere is bounded from below by the volume of a fattened ``cap'' of volume identical to that of $A$. 
Originally, this fact
was discovered by L\'evy, and it was then later generalised by Schmidt, see  \cite{Schmidt48} or the review
article \cite{Gardner}.
\end{claim}
These two facts can then be combined in order to show that half-spaces are optimal for finite-dimensional Gaussian measures. Finally,
a clever approximation argument is used in order to generalise this statement to infinite-dimensional measures as well.

An immediate corollary is given by the following type of zero-one law\index{zero-one law} for Gaussian measures:

\begin{corollary}
Let $V \subset \CB$ be a measurable linear subspace. Then, one has either $\mu(V) = 0$ or $\mu(V) = 1$.
\end{corollary}

\begin{proof}
Let us first consider the case where $\CH_\mu \not \subset V$. In this case, just as in the proof of Proposition~\ref{prop:carH},
we conclude that $\mu(V) = 0$, for otherwise we could construct an uncountable collection of disjoint sets with positive measure.

If $\CH_\mu \subset V$, then we have $V + B_\eps = V$ for every $\eps > 0$, so that if $\mu(V) > 0$, one must have $\mu(V) = 1$ by Theorem~\ref{theo:Borell}.
\end{proof}

We have now all the necessary ingredients in place to be able to give a proof of Theorem~\ref{theo:uniqueext}:

\begin{proof}[Proof of Theorem~\ref{theo:uniqueext}]
Assume by contradiction that there exist two measurable extensions $\hat A_1$ and $\hat A_2$ of $A$. In other words, 
we have $\hat A_i x = Ax$ for $x \in \CH_\mu$ and there exist measurable subspaces $V_i$ with $\mu(V_i) = 1$ such that 
the restriction of $\hat A_i$ to $V_i$ is linear. Denote $V = V_1 \cap V_2$ and $\Delta = \hat A_2 - \hat A_1$, so that 
$\Delta$ is linear on $V$ and $\Delta \restr \CH_\mu = 0$.

Let $\ell \in \CB_2^*$ be arbitrary and consider the events $V_\ell^c = \{x\,:\, \ell(\Delta x) \le c\}$. By the linearity of $\Delta$, each of these events is
invariant under translations in $\CH_\mu$, so that by Theorem~\ref{theo:Borell} we have $\mu(V_\ell^c) \in \{0,1\}$ for every choice of 
$\ell$ and $c$. Furthermore, for fixed $\ell$, the map $c \mapsto \mu(V_\ell^c)$ is increasing and it follows from the
$\sigma$-additivity of $\mu$ that we have $\lim_{c \to -\infty} \mu(V_\ell^c) = 0$ and $\lim_{c \to \infty} \mu(V_\ell^c) = 1$.
Therefore, there exists a unique $c_\ell \in \R$ such that $\mu(V_\ell^c)$ jumps from $0$ to $1$ at $c = c_\ell$. In particular, this
implies that $\ell(\Delta x) = c_\ell$ $\mu$-almost surely. However, the measure $\mu$ is invariant under the map 
$x \mapsto -x$, so that we must have $c_\ell = -c_\ell$, implying that  $c_\ell = 0$. Since this is true for every $\ell \in \CB_2^*$,
we conclude from Proposition~\ref{prop:determinefinitedim} that the law of $\Delta x$ is given by the Dirac measure at $0$, so that $\Delta x = 0$ $\mu$-almost surely, which is precisely what we wanted.
\end{proof}

\begin{exercise}
Let $(b_{ij})_{i,j \in \N}$ be such that $\sum_{i,j} b_{ij}^2 < \infty$ and let $(\xi_i)_{i \in \N}$ be a sequence of i.i.d.\ $\CN(0,1)$ random variables. Show first that the
sum $\sum_{i,j} b_{ij} (\xi_i\xi_j - \delta_{i,j})$ converges in $L^2$ to some
random variable $F$. 
Use Theorem~\ref{theo:Borell} to show that $F$ has exponential tails, i.e.\ there 
exists a constant $c$ such that $\P(|F| > K) \le c^{-1} e^{-cK}$.
\textbf{Hint:} Writing $B \colon \ell^2 \to \ell^2$ for the map such that 
$Be_i = b_{ij}e_j$ (the $e_i$'s are the canonical basis vectors), one can take
the event $A$ to be the realisations such that $|F(\xi)| + \|B\xi\| + \|B^*\xi\| \le C$
for some sufficiently large $C$.
\end{exercise}

\begin{exercise}
Show that any unitary operator on $\CH_\mu$ uniquely extends to a measure-preserving
transformation of $(\CB,\mu)$. In particular, show that if $W$ is an $n$-dimensional
Wiener process (i.e.\ a sample from the Gaussian measure on $\CC([0,1],\R^n)$ with
covariance $C_{ij}(s,t) = \delta_{i,j} (s\wedge t)$) and if $\CO \colon [0,1] \to O(n)$ is
a measurable map with values in the space of orthogonal matrices, then the process
$\hat W(t) = \int_0^t \CO(s)\,dW(s)$ is well-defined and is again a Wiener process.
\end{exercise}

In the next section, we will see how we can take advantage of this fact to construct a theory of stochastic integration with respect to
a ``cylindrical Wiener process'', which is the infinite-dimensional analogue of a standard $n$-dimensional Wiener process.

\subsection{Cylindrical Wiener processes and stochastic integration}
\label{sec:Ito}

Central to the theory of stochastic PDEs is the notion of a \index{Wiener!process (cylindrical)}\textit{cylindrical Wiener process}. 
Recall that in general a stochastic process $X$ taking values in a separable Banach space $\CB$ is nothing but a collection $\{X(t)\}$
of $\CB$-valued random variables indexed by time $t \in \R$ (or taking values in some subset of $\R$). 
A notable special case which will be of interest here is
the case where the probability space is taken to be for example $\Omega = \CC([0,T], \CB)$ (or some other space of $\CB$-valued continuous
functions) endowed with some Gaussian measure $\P$ and where the process $X$ is given by
\begin{equ}
X(t)(\omega) = \omega(t)\;,\qquad \omega \in \Omega\;.
\end{equ}
In this case, $X$ is called the \textit{canonical process}\index{canonical process} on $\Omega$.

Recall that the usual (one-dimensional) Wiener process is a real-valued centred Gaussian process $B(t)$ such that
$B(0) = 0$ and $\E |B(t) - B(s)|^2 = |t-s|$ for any pair of times $s,t$. From our point of view,
the Wiener process on any finite time interval $I$ can always be realised as the canonical process for the Gaussian measure on $\CC(I,\R)$ with
covariance function $C(s,t) = s \wedge t = \min\{s,t\}$. (Note that such a measure exists by
 Kolmogorov's continuity criterion.) 

Since the space $\CC(\R,\R)$ is not a Banach space and we have not extended our study of Gaussian measures to
Fr\'echet spaces, we refrain from defining a measure on it. However,  
one can define Wiener measure on a separable Banach space of the type
\begin{equ}
\CC_\rho(\R_+,\R) = \Bigl\{f\in \CC(\R_+,\R)\,:\, \lim_{t \to \infty} f(t)/\rho(t) \;\text{exists}\Bigr\}\;,\qquad
\|f\|_\rho = \sup_{t \in \R} {|f(t)|\over \rho(t)}\;,
\end{equ}
for a suitable weight function $\rho\colon \R \to [1,\infty)$. For example, we will see that $\rho(t) = 1+t^2$ is suitable, and we will therefore
define $\CC_W = \CC_\rho$ for this particular choice.

\begin{proposition}
There exists a Gaussian measure $\mu$ on $\CC_W$ with covariance function $C(s,t) = s\wedge t$.
\end{proposition}

\begin{proof}
We use the fact that $f \in \CC([0,\pi],\R)$ if and only if
the function $T(f)$ given by $T(f)(t) = (1+t^2) f(\arctan t)$ belongs to $\CC_W$. Our aim is then to construct a Gaussian measure
$\mu_0$ on $\CC([0,\pi],\R)$ which is such that $T^*\mu_0$ has the required covariance structure.

The covariance $C_0$ for $\mu_0$ is then given by
\begin{equ}
C_0(x,y) =  {\tan x \wedge \tan y \over (1+\tan^2x)(1+\tan^2 y)}\;.
\end{equ}
It is now a straightforward exercise to check that this covariance function does indeed satisfy the assumption
of Kolmogorov's continuity theorem.
\end{proof}

Let us now fix a (separable) Hilbert space $\CH$, as well as a larger Hilbert space $\CH'$ containing $\CH$
as a dense subset and such that 
the inclusion map $\iota \colon \CH \to \CH'$ is Hilbert-Schmidt. Given $\CH$, it is always possible to construct a space
$\CH'$ with this property: choose an orthonormal basis $\{e_n\}$ of $\CH$ and take $\CH'$ to be the closure of $\CH$
under the norm
\begin{equ}
\|x\|_{\CH'}^2 = \sum_{n=1}^\infty {1\over n^2} \scal{x, e_n}^2\;.
\end{equ}
One can check that the map $\iota \iota^*$ is then given by $\iota \iota^* e_n = {1\over n^2} e_n$, so that it is indeed 
trace class.

\begin{definition}
Let $\CH$ and $\CH'$ be as above.
We then call a \textit{cylindrical Wiener process on $\CH$}\index{Wiener!process (cylindrical)} any $\CH'$-valued Gaussian process $W$ such that
\begin{equ}[e:cylWien]
\E \scal{h,W(s)}_{\CH'} \scal{W(t),k}_{\CH'} = (s\wedge t) \scal{\iota^* h,\iota^* k}= (s\wedge t) \scal{\iota \iota^* h, k}_{\CH'}\;,
\end{equ}
for any two times $s$ and $t$ and any two elements $h,k \in \CH'$.
By Kolmogorov's continuity theorem, this can be realised as the canonical process for some Gaussian measure
on $\CC_W(\R, \CH')$.
\end{definition}

\begin{proposition}\label{prop:identity}
In the same setting as above, the Gaussian measure $\mu$ on $\CH'$ with covariance $\iota \iota^*$ 
has $\CH$ as its Cameron--Martin space. Furthermore, $\|h\|_\mu^2 = \|h\|^2$ for every $h \in \CH$.
\end{proposition}

\begin{proof}
It follows from the definition of $\HH_\mu$ that this is precisely the range of $\iota \iota^*$ and that
the map $h \mapsto h^*$ is given by $h^* = (\iota \iota^*)^{-1}h$. In particular, $\HH_\mu$ is contained in 
the range of $\iota$. Therefore, for any $h,k \in \HH_\mu$, there exist $\hat h, \hat g \in \CH$
such that $h = \iota \hat h$ and $k = \iota \hat k$. Using this, we have
\begin{equ}
\scal{h,k}_\mu = \scal{(\iota \iota^*) h^*,k^*}_{\CH'} = \scal{h, (\iota \iota^*)^{-1}k}_{\CH'} = \scal{\iota \hat h, (\iota \iota^*)^{-1} \iota \hat k}_{\CH'} 
= \scal{\hat h, \iota^* (\iota \iota^*)^{-1} \iota \hat k} = \scal{\hat h, \hat k}\;,
\end{equ}
from which the claim follows. 
\end{proof}

\begin{exercise}
As in Exercise~\ref{ex:CMWiener}, show that the Cameron--Martin space of 
a cylindrical Wiener process on $\CH$ equals $H_0^{1,2}([0,T],\CH)$.
\end{exercise}

In view of this exercise, we could also have defined the cylindrical Wiener process 
on $\CH$ as the canonical process
for any Gaussian measure with Cameron--Martin space $H_0^{1,2}([0,T],\CH)$.
The name ``cylindrical Wiener process on $\CH$'' may sound confusing at first, since it is actually \textit{not}
an $\CH$-valued process. Note however that if $h$ is an element
in $\CH$ that is in the range of $\iota^*$ (so that $\iota h$ belongs to the range of $\iota \iota^*$ and 
$\iota^* (\iota \iota^*)^{-1} \iota h = h$), then 
\begin{equ}
\scal{h,k} = \scal{\iota^* (\iota \iota^*)^{-1} \iota h, k} = \scal{(\iota \iota^*)^{-1} \iota h, \iota k}_{\CH'} \;.
\end{equ}
In particular, if we just \textit{pretend} for a moment that $W(t)$ belongs to $\CH$ for every $t$ (which is of course not true!), then we get
\begin{equs}
\E \scal{h,W(s)} \scal{W(t),k} &= \E \scal{(\iota \iota^*)^{-1} \iota h, \iota W(s)}_{\CH'} \scal{(\iota \iota^*)^{-1} \iota k, \iota W(t)}_{\CH'} \\
&= (s\wedge t) \scal{\iota \iota^* (\iota \iota^*)^{-1} \iota h, (\iota \iota^*)^{-1} \iota k}_{\CH'} \\
&= (s\wedge t) \scal{\iota h, (\iota \iota^*)^{-1} \iota k}_{\CH'} = (s\wedge t) \scal{h, \iota^*(\iota \iota^*)^{-1} \iota k}_{\CH'} \\
&= (s\wedge t) \scal{h,k}\;.
\end{equs}
Here we used \eref{e:cylWien} to go from the first to the second line. This shows that $W(t)$ should be thought
of as an $\CH$-valued random variable with covariance $t$ times the identity operator (which is of course not trace class
if $\CH$ is infinite-dimensional, so that such an object cannot exist if $\dim \CH = \infty$). Combining Proposition~\ref{prop:identity}
with Theorem~\ref{theo:imageGaussian}, we see however that if $\CK$ is some Hilbert space
and $A\colon \CH \to \CK$ is a Hilbert-Schmidt operator, then the $\CK$-valued random
variable $A W(t)$ is well-defined. (Here we made an abuse of notation and also used the symbol $A$
for the measurable extension of $A$ to $\CH'$.)
Furthermore, its law does not depend on the choice of the larger space $\CH'$.

\begin{example}[White noise]\index{white noise}
Recall that we informally defined ``white noise'' as a Gaussian process $\xi$ with covariance
$\E \xi(s)\xi(t) = \delta(t-s)$. In particular, if we denote by $\scal{\cdot,\cdot}$ the scalar product in $L^2(\R)$,
this suggests that 
\begin{equ}[e:covgh]
\E \scal{g,\xi}\scal{h,\xi} = \E \int \!\!\! \int g(s)h(t)\xi(s) \xi(t) \,ds\,dt = \int \!\!\! \int g(s)h(t) \delta(t-s) \,ds\,dt = \scal{g,h}\;.
\end{equ}
This calculation shows that white noise can be constructed as a Gaussian random variable on any Hilbert space $\CH$
of distributions containing $L^2(\R)$ and such the embedding $L^2(\R) \hookrightarrow \CH$ is Hilbert--Schmidt. Furthermore, by Theorem~\ref{theo:imageGaussian} (and in fact already by
definition of the reproducing kernel space),
integrals of the form $\int g(s)\xi(s)\,ds$ are well-defined random variables, provided that $g \in L^2(\R)$. 
\end{example}

\begin{exercise}
Taking for $A\colon L^2 \to H_0^{1,2}$ the integration map, 
namely $(A h)(t) = \int_0^t h(s)\,ds$, 
show that if $\mu$ is the white noise measure, then $\hat A^*\mu$ is Wiener measure,
thus justifying the statement that ``white noise is the derivative of Brownian motion''.
\end{exercise}

The interesting fact about this construction is that we can use it to define space-time white noise in exactly the same way, simply replacing
$L^2(\R)$ by $L^2(\R^2)$.

\begin{exercise}
Let $\T^d$ be the $d$-dimensional torus and let $\xi$ be white noise on $\R^d$. Show
that one can realise $\xi$ as a sample of a Gaussian measure on the negative Sobolev space
$H^{-s}$ if and only if $s > d/2$.
\end{exercise}

This will allow us to define a Hilbert space-valued stochastic integral against a cylindrical Wiener process in pretty much the
same way as what is usually done in finite dimensions. In the sequel, we fix a cylindrical Wiener process $W$ on
some Hilbert space $\CH \subset \CH'$, which we realise as the canonical coordinate process on $\Omega = \CC_W(\R_+, \CH')$
equipped with the measure constructed above.
We also denote by $\F_{\!s}$ the $\sigma$-field on $\Omega$ generated by $\{W_r\,:\, r \le s\}$.

Consider now a finite collection of \textit{disjoint} intervals $(s_n, t_n] \subset \R_+$ with $n=1,\ldots,N$ and a corresponding finite collection of
$\F_{\!s_n}$-measurable random variables $\Phi_n$ taking values in the space $\CL_2(\CH,\CK)$ of Hilbert-Schmidt
operators from $\CH$ into some other fixed Hilbert space $\CK$. Let furthermore $\Phi$ be the $L^2(\R_+ \times \Omega, \CL_2(\CH,\CK))$-valued
function defined by
\begin{equ}
\Phi(t,\omega) = \sum_{n=1}^N \Phi_n(\omega)\, \one_{(s_n, t_n]}(t) \;,
\end{equ}
where we denoted by $\one_{A}$ the indicator function of a set $A$. We call such a $\Phi$ an \textit{elementary process on $\CH$}.

\begin{definition}
Given an elementary process $\Phi$ and a cylindrical Wiener process $W$ on $\CH$, we define the 
$\CK$-valued \index{stochastic integral}stochastic integral
\begin{equ}
\int_0^\infty \Phi(t)\,dW(t) \eqdef \sum_{n=1}^N \Phi_n(W)\, \bigl(W(t_n) - W(s_n)\bigr)\;.
\end{equ}
Note that since $\Phi_n$  is $\F_{s_n}$-measurable, $\Phi_n(W)$ is independent of $W(t_n) - W(s_n)$, therefore each 
term on the right hand side can be interpreted in the sense of the construction of Theorems~\ref{theo:imageGaussian}
and \ref{theo:uniqueext}. 
\end{definition}

\begin{remark}
Thanks to Theorem~\ref{theo:uniqueext}, this construction is well-posed without requiring
 to specify the larger Hilbert space $\CH'$ on which $W$ can be realised as an $\CH'$-valued process.
 This justifies the terminology of $W$ being ``the cylindrical Wiener process on $\CH$'' without
 any mentioning of $\CH'$, since the value of stochastic integrals against $W$ is independent of the
 choice of $\CH'$. 
\end{remark}

It follows from Theorem~\ref{theo:imageGaussian} and \eref{e:expnorm} that one has the identity
\begin{equ}[e:Ito]
\E \Bigl\| \int_0^\infty \Phi(t)\,dW(t)\Bigr\|_{\CK}^2 = \sum_{n=1}^N \E \tr \bigl(\Phi_n(W)\Phi_n^*(W) \bigr)(t_n - s_n) 
= \E \int_0^\infty \tr \Phi(t) \Phi^*(t)\,dt\;,
\end{equ}
which is an extension of the usual It\^o isometry to the Hilbert space setting. It follows that the stochastic
integral is an isometry from the subset of elementary processes in $L^2(\R_+ \times \Omega, \CL_2(\CH,\CK))$
to $L^2(\Omega, \CK)$.

Let now $\Fp$ be the ``predictable'' $\sigma$-field, that is the $\sigma$-field over $\R_+ \times \Omega$ generated
by all subsets of the form $(s,t] \times A$ with $t > s$ and $A \in \F_{\!s}$. 
This is the smallest $\sigma$-algebra with respect to which all elementary processes are $\Fp$-measurable.
One furthermore has:

\begin{proposition}
The set of elementary processes is dense in the space $L^2_{\mathrm{pr}}(\R_+ \times \Omega, \CL_2(\CH,\CK))$ of
all predictable $\CL_2(\CH,\CK)$-valued processes.
\end{proposition}

\begin{proof}
We first show that the set of real-valued elementary processes is dense in 
$L^2_{\mathrm{pr}}(\R_+ \times \Omega, \R)$.
Denote by $\hat{\Fp}$ the set of all sets of the form $(s,t] \times A$ with $A \in \F_s$. Denote furthermore by $\hat L^2_{\mathrm{pr}}$
the closure of the set of elementary processes in $L^2$.
One can check that $\hat{\Fp}$ is closed under finite intersections (i.e.\ it is a $\pi$-system), so that $\one_G \in \hat L^2_{\mathrm{pr}}$ for every set $G$ in the algebra
generated by $\hat{\Fp}$. It follows from the monotone class theorem that $\one_G \in \hat L^2_{\mathrm{pr}}$ for every set 
$G \in \Fp$. The claim then follows from the definition of the Lebesgue integral, just as for the corresponding statement in $\R$.

Note now that $\CL_2(\CH,\CK)$ is a separable Hilbert space, fix some orthonormal basis for it, and 
denote by $\Pi_n$ the orthonormal projection onto the span $\CQ_n$ of the first $n$ such basis vectors.
Since $\CQ_n$ is finite-dimensional, we just showed that elementary processes are dense in  
$L^2_{\mathrm{pr}}(\R_+ \times \Omega, \CQ_n)$ for every $n$. Finally, $\bigcup_n L^2_{\mathrm{pr}}(\R_+ \times \Omega, \CQ_n)$ is dense in $L^2_{\mathrm{pr}}(\R_+ \times \Omega, \CL_2(\CH,\CK))$ by Lebesgue's
dominated convergence theorem since $\|\Pi_n x - x\| \to 0$ for every $x \in \CL_2(\CH,\CK)$.
\end{proof}

By using the It\^o isometry \eref{e:Ito} and the completeness of $L^2(\Omega, \CK)$, it follows that:

\begin{corollary}
The stochastic integral $\int_0^\infty \Phi(t)\,dW(t)$ can be uniquely 
defined for every process $\Phi \in L^2_{\mathrm{pr}}(\R_+ \times \Omega, \CL_2(\CH,\CK))$.
\end{corollary}

This concludes our presentation of the basic properties of Gaussian measures on infinite-dimensional
spaces. The next section deals with the other main ingredient to solving stochastic PDEs, which is the
behaviour of deterministic linear PDEs.

\section{A Primer on Semigroup Theory}

\index{semigroup!strongly continuous ($\CC_0$)}
This section is strongly based on Davies's excellent monograph \cite{Davies} for the first part on strongly continuous semigroups
and very loosely follows \cite{Yos95} and  \cite{Lunardi} for the second part on analytic semigroups. 
Another good reference on some of the material covered here is the monograph \cite{Pazy}.
Its aim is to give a rigorous meaning to solutions to linear equations of the type
\begin{equ}[e:linear]
\d_t x = L x\;,\quad x(0) = x_0 \in \CB\;,
\end{equ}
where $x$ takes values in some Banach space $\CB$ and $L$ is a possibly unbounded operator on $\CB$.
From a formal point of view, if such a solution exists, one expects the existence of a linear operator $S(t)$ that 
maps the initial condition $x_0$ onto the solution $x(t)$ of \eref{e:linear} at time $t$. If such a solution is unique, then
the family of operators $S(t)$ should satisfy $S(0) = 1$ and $S(t)\circ S(s) = S(t+s)$. 
This is called the \textit{semigroup} property.

Furthermore, such a family of solution operators
$S(t)$ should have some regularity as $t \to 0$ in order to give a meaning to the notion of an initial condition.
(The family given by $S(t) = 0$ for $t>0$ and $S(0) = 1$ does satisfy the semigroup property but clearly
doesn't define a family of solution operators to an equation of the type \eref{e:linear}.)

This motivates the following definition:

\begin{definition}
A semigroup $S(t)$ on a Banach space $\CB$ is a family of bounded linear operators $\{S(t)\}_{t \ge 0}$
with the properties that $S(t)\circ S(s) = S(t+s)$ for any $s,t \ge 0$ and that $S(0) = \mathrm{Id}$. A semigroup is furthermore called
\begin{claim}
\item \textit{strongly continuous} if the map $(x,t) \mapsto S(t)x$ is continuous from $\CB \times [0,\infty)$ into $\CB$.
\item \textit{analytic} if there exists $\theta > 0$ such that the operator-valued map 
$t \mapsto S(t)$ has an analytic extension to $\{\lambda \in \C \,:\, |{\arg \lambda}| < \theta\}$, satisfies the semigroup property there,
and is such that $t \mapsto S(e^{i\phi} t)$ is a strongly continuous semigroup for every angle $\phi$ with $|\phi| < \theta$.
\end{claim}
\end{definition}
A strongly continuous semigroup is also sometimes called a $\CC_0$-semigroup.

\begin{exercise}\label{ex:SC}
Show that, for a semigroup of bounded operators, being strongly continuous is equivalent to $t \mapsto S(t)x$ being continuous at $t=0$ for every $x\in\CB$ 
and the operator norm of $S(t)$ being bounded by $Me^{at}$ for some constants $M$ and $a$.
Show then that the first condition can be relaxed to $t \mapsto S(t)x$ being continuous for all $x$ in some dense subset
of $\CB$. Finally, use the Banach--Steinhaus theorem (also called uniform boundedness principle) to show that the second condition can be dropped. 
(However, both relaxing the first condition and dropping the second condition is not possible. 
See Exercise~\ref{ex:unboundedSg} on how to construct
a semigroup of bounded operators such that $t \mapsto S(t)x$ is continuous for a dense set of $x$'s, but
such that $\|S(t)\|$ is nevertheless unbounded as $t\to0$.)
\end{exercise}

\begin{remark}
Some authors, like \cite{Lunardi}, do not impose strong continuity in the definition of an analytic semigroup. This can result in 
additional technical complications due to the fact that the generator may then not have dense domain.
The approach followed here has the slight drawback that with our definitions the heat semigroup is not 
analytic on $L^\infty(\R)$. (It lacks strong continuity as can be seen by applying it to a step function.)
It is however analytic for example on $\CC_0(\R)$, the space of continuous functions vanishing at infinity.
\end{remark}

This section is going to assume some familiarity with functional analysis. All the necessary results can be
found for example in the classical monograph by Yosida \cite{Yos95}. Recall that an unbounded operator
$L$ on a Banach space $\CB$ consists of a linear subspace $\CD(L) \subset \CB$ called the \textit{domain} of $L$
and a linear map $L \colon \CD(L) \to \CB$. The \textit{graph} of an operator is the subset of $\CB \times \CB$ consisting of
all elements of the form $(x,Lx)$ with $x \in \CD(L)$.
An operator is \textit{closed}\index{closed!operator} if its graph is a closed subspace of $\CB\times \CB$. It is \textit{closable} if the closure of its
graph is again the graph of a linear operator and that operator is called the \textit{closure} of $L$.

The domain $\CD(L^*)$ of the \textit{adjoint}\index{adjoint!operator} $L^*$ of an unbounded operator $L\colon \CD(L) \to \CB$ is defined as the set of all elements
$\ell \in \CB^*$ such that there exists an element $L^*\ell \in \CB^*$ with the property that $ \bigl(L^*\ell\bigr)(x) = \ell(Lx)$ for every
$x \in \CD(L)$. It is clear that in order for the adjoint to be well-defined, we have to require that the domain of $L$ is dense in $\CB$.
Fortunately, this will be the case for all the operators that will be considered in these notes.

\begin{exercise}\label{ex:closed}
Show that $L$ being closed is equivalent to the fact that if $\{x_n\} \subset \CD(L)$ is Cauchy in $\CB$ and
$\{Lx_n\}$ is also Cauchy, then $x = \lim_{n \to \infty} x_n$ belongs to $\CD(L)$ and
$Lx = \lim_{n \to \infty} Lx_n$.
\end{exercise}

\begin{exercise}
Show that the adjoint of an operator with dense domain is always closed.
\end{exercise}

The \textit{resolvent set}\index{resolvent!set} $\rho(L)$ of an operator $L$ is defined by
\begin{equ}
\rho(L) = \{\lambda \in \C\,:\, \text{$(\lambda - L) \colon \CD(L) \to \CB$ is bijective and admits a continuous inverse}\}\;,
\end{equ}
and the \textit{resolvent}\index{resolvent} $R_\lambda$ is given for $\lambda \in \rho(L)$ by $R_\lambda = (\lambda - L)^{-1}$.
(Here and in the sequel we view $\CB$ as a complex Banach space. If an operator is defined on a real Banach space,
it can always be extended to its complexification in a canonical way and we will identify the two without further notice in the sequel.)
The spectrum $\sigma(L)$ of $L$ is the complement of the resolvent set.

The most important results regarding the resolvent of an operator that we are going to use are that any closed operator $L$ 
with non-empty resolvent set is defined in a unique way by its resolvent. Furthermore, the resolvent set is open and the resolvent 
is an analytic function from $\rho(L)$ to the space
$\CL(\CB)$ of bounded linear operators on $\CB$. Finally, the resolvent operators for different values of $\lambda$
all commute and satisfy the resolvent identity\index{resolvent!identity}
\begin{equ}
R_\lambda - R_\mu = (\mu - \lambda)R_\mu R_\lambda\;,
\end{equ}
for any two $\lambda, \mu \in \rho(L)$. 

The fact that the resolvent is operator-valued should not be 
a conceptual obstacle to the use of notions from complex analysis. Indeed, for $D \subset \C$ an open domain, a function $f\colon D \to \CB$  
where $\CB$ is any complex
Banach space (typically the complexification of a real Banach space which we identify with the original space without further ado)
is said to be analytic  in exactly the same way as usual by imposing that its Taylor series
at any point $a \in D$ converges to $f$ uniformly in $\CB$ on a neighbourhood of $a$. The same definition applies if $D \subset \R$ and
analytic continuation then works in exactly the same way as for complex-valued functions.
In particular, Cauchy's residue theorem, which is the main result from complex analysis that we are going to use later on,
works for Banach-space valued functions in exactly the same way as for complex-valued functions.

\subsection{Strongly continuous semigroups}

We start our investigation of semigroup theory with a discussion of the main results that can
be obtained for strongly continuous semigroups.
Given a $\CC_0$-semigroup, one can associate to it a ``generator'', which is essentially the derivative
of $S(t)$ at $t = 0$:

\begin{definition}
The \textit{generator} $L$ of a $\CC_0$-semigroup is given by
\begin{equ}[e:defGen]
Lx = \lim_{t \to 0} t^{-1} \bigl(S(t)x - x\bigr)\;,
\end{equ}
on the set $\CD(L)$ of all elements $x \in \CB$ such that this limit exists (in the sense of strong convergence in $\CB$).
\end{definition}

The following result shows that if $L$ is the generator of a $\CC_0$-semigroup $S(t)$, then 
$x(t) = S(t)x_0$ is indeed the solution to \eref{e:linear} in a weak sense. 
\begin{proposition}\label{prop:diffSt}
The domain $\CD(L)$ of $L$ is dense in $\CB$, invariant under $S$, and the identities
$\d_t S(t)x = LS(t)x = S(t)Lx$ hold for every $x \in \CD(L)$ and every $t\ge0$. 
Furthermore, for every $\ell \in \CD(L^*)$ and every $x \in \CB$, the map
$t \mapsto \scal{\ell,S(t)x}$ is differentiable and one has $\d_t \scal{\ell,S(t)x} = \scal{L^*\ell,S(t)x}$.
\end{proposition}
\begin{proof}
Fix some arbitrary $x \in \CB$ and set $x_t = \int_0^t S(s)x\, ds$. One then has
\begin{equs}
 \lim_{h \to 0} h^{-1} \bigl(S(h)x_t - x_t\bigr) &=  \lim_{h \to 0} h^{-1} \Bigl(\int_h^{t+h} S(s)x\, ds - \int_0^{t} S(s)x\, ds\Bigr) \\
 &=  \lim_{h \to 0} h^{-1} \Bigl(\int_t^{t+h} S(s)x\, ds - \int_0^{h} S(s)x\, ds\Bigr) = S(t)x - x\;,
\end{equs}
where the last equality follows from the strong continuity of $S$. This shows that $x_t \in \CD(L)$.
Since $t^{-1}x_t \to x$  as $t\to 0$ and since $x$ was arbitrary, it follows that $\CD(L)$ is dense in $\CB$. 
To show that it is invariant under $S$, note that for $x \in \CD(L)$ one has
\begin{equ}
 \lim_{h \to 0} h^{-1} \bigl(S(h)S(t)x- S(t)x\bigr)  = S(t) \lim_{h \to 0} h^{-1} \bigl(S(h)x- x\bigr)   = S(t) Lx\;,
\end{equ}
so that $S(t)x \in \CD(L)$ and $LS(t)x = S(t)Lx$. To show that it this is equal to $\d_t S(t)x$, it suffices to check that 
the left derivative of this expression exists and is equal to the right derivative. This is left as an exercise.

To show that the second claim holds, note that for $x \in \CD(L)$ it follows from the first claim that
\begin{equ}
\scal{\ell, S(t)x} = \int_0^t \scal{\ell, LS(s)x}\,ds = \int_0^t \scal{L^*\ell, S(s)x}\,ds\;.
\end{equ}
Since $\CD(L)$ is dense in $\CB$ and since both expressions are continuous on $\CB$ (as
functions of $x$ for $t$ fixed), this identity extends to all $x \in \CB$ and the claim follows from
the fundamental theorem of calculus.
\end{proof}

It follows as a corollary that no two semigroups can have the same generator (unless the semigroups coincide of course), which justifies the notation
$S(t) = e^{Lt}$ that we are occasionally going to use in the sequel.
\begin{corollary}\label{cor:LdetS}
If a function $x\colon [0,1] \to \CD(L)$ satisfies $\d_t x_t = L x_t$ for every $t \in [0,1]$, then $x_t = S(t) x_0$. In particular,
no two distinct $\CC_0$-semigroups can have the same generator. 
\end{corollary}
\begin{proof}
It follows from an argument almost identical to that given in the proof of Proposition~\ref{prop:diffSt} that the map $t \mapsto S(t) x_{T-t}$
is continuous on $[0,T]$ and differentiable on $(0,T)$. Computing its derivative, we obtain $\d_t S(t) x_{T-t} = LS(t) x_{T-t} - S(t) L x_{T-t} = 0$, so that
$x_{T} = S(T)x_{0}$.
\end{proof}

\begin{exercise}\label{ex:transl}
Show that the semigroup $S(t)$ on $L^2(\R)$ given by
\begin{equ}
\bigl(S(t) f\bigr)(\xi) = f(\xi+t)\;,
\end{equ}
is strongly continuous and that its generator is  given by $L = \d_\xi$ with $\CD(L) = H^1$. Similarly, show that the
heat semigroup on $L^2(\R)$ given by
\begin{equ}
\bigl(S(t) f\bigr)(\xi) = {1\over \sqrt{4\pi t}}\int \exp \Bigl(-{|\xi-\eta|^2\over 4t}\Bigr)f(\eta)\,d\eta\;,
\end{equ}
is strongly continuous and that its generator is  given by $L = \d_\xi^2$ with $\CD(L) = H^2$. 
\textbf{Hint:} Use Exercise~\ref{ex:SC} to show strong continuity. 
\end{exercise}

\begin{remark}
We did not make any assumption on the structure of the Banach space $\CB$. \textit{However}, it is a general
rule of thumb (although this is \textit{not} a theorem) that semigroups on non-separable Banach spaces tend not to be strongly
continuous. For example, neither the heat semigroup nor the translation semigroup from the previous exercise
are strongly continuous
on $L^\infty(\R)$ or even on $\CC_b(\R)$, the space of all bounded continuous functions. 
\end{remark}

Recall now that the resolvent set for an operator $L$ consists of those $\lambda \in \CC$ such that 
the operator $\lambda - L$ is one to one. For $\lambda$ in the resolvent set, we denote by $R_\lambda = (\lambda - L)^{-1}$
the resolvent of $L$. It turns out that the resolvent of the generator of a $\CC_0$-semigroup
can easily be computed:

\begin{proposition}\label{prop:resolvEqu}
Let $S(t)$ be a $\CC_0$-semigroup such that $\|S(t)\| \le Me^{a t}$ for some constants $M$ and $a$.
If $\Re \lambda > a$, then $\lambda$ belongs to the resolvent set of $L$ and one has the identity
$R_\lambda x = \int_0^\infty e^{-\lambda t} S(t)x\, dt$.
\end{proposition}

\begin{proof}
By the assumption on the bound on $S$, the expression $Z_\lambda = \int_0^\infty e^{-\lambda t} S(t)x\, dt$ is well-defined
for every $\lambda$ with $\Re \lambda > a$. In order to show that $Z_\lambda = R_\lambda$, we first show that
$Z_\lambda x \in \CD(L)$ for every $x \in \CB$ and that $(\lambda - L) Z_\lambda x = x$. We have
\begin{equs}
LZ_\lambda x &= \lim_{h \to 0} h^{-1} \bigl(S(h)Z_\lambda x- Z_\lambda x\bigr) =  \lim_{h \to 0} h^{-1} \int_0^\infty e^{-\lambda t} \bigl(S(t+h) x- S(t) x\bigr)\, dt \\
 &= \lim_{h \to 0} \Bigl({e^{\lambda h} - 1 \over h} \int_0^\infty e^{-\lambda t} S(t) x\, dt - {e^{\lambda h}\over h}\int_0^h e^{-\lambda t}S(t)x\, dt\Bigr) \\
 & = \lambda Z_\lambda x - x\;,
\end{equs}
which is the required identity. To conclude, it remains to show that $\lambda - L$ is an injection on $\CD(L)$ (otherwise one might imagine
a situation where $Z_\lambda$ maps $\CB$ into a strict subspace of $\CD(L)$). If it was not,
we could find $x \in \CD(L) \setminus \{0\}$ such that $Lx = \lambda x$. Setting $x_t = e^{\lambda t} x$ and applying Corollary~\ref{cor:LdetS},
this yields $S(t)x = e^{\lambda t} x$, thus contradicting the bound $\|S(t)\| \le Me^{a t}$ if $\Re \lambda > a$.
\end{proof}

We can deduce from this that:
\begin{proposition}\label{prop:closed}
The generator $L$ of a $\CC_0$-semigroup is a closed operator.
\end{proposition}
\begin{proof}
We are going to use the characterisation of closed operators given in Exercise~\ref{ex:closed}. Shifting $L$
by a constant if necessary (which does not affect it being closed or not), we can assume that $a=0$.
Take now a sequence $x_n \in \CD(L)$ such that $\{x_n\}$ and $\{Lx_n\}$ are both Cauchy in $\CB$
and set $x = \lim_{n\to\infty}x_n$ and $y = \lim_{n\to\infty} Lx_n$.
Setting $z_n = (1 - L)x_n$, we have $\lim_{n \to \infty} z_n = x-y$.

On the other hand, we know that $1$ belongs to the resolvent set, so that 
\begin{equ}
x = \lim_{n\to\infty}x_n = \lim_{n\to\infty} R_1 z_n = R_1(x-y)\;.
\end{equ}
By the definition of the resolvent, this implies that $x\in \CD(L)$ and that $x - Lx = x-y$, so that $Lx=y$ as required.
\end{proof}

We are now ready to give a full characterisation of the generators of $\CC_0$-semigroups. This is the content
of the following theorem:\index{Hille--Yosida theorem}

\begin{theorem}[Hille-Yosida]
A closed densely defined operator $L$ on the Banach space $\CB$ is the generator of a $\CC_0$-semigroup
$S(t)$ with $\|S(t)\| \le Me^{at}$ if and only if all $\lambda$ with $\Re \lambda > a$ lie in its resolvent set and
the bound $\|R_\lambda^n\| \le M \,\bigl(\Re\lambda - a\bigr)^{-n}$ holds there for every $n \ge 1$.
\end{theorem}


\begin{proof}
The generator $L$ of a $\CC_0$-semigroup is closed by Proposition~\ref{prop:closed}. The fact that its resolvent satisfies the stated bound
follows immediately from the fact that 
\begin{equ}
R_\lambda^n x = \int_0^\infty \cdots  \int_0^\infty  e^{-\lambda(t_1+\ldots +t_n)}S(t_1+\ldots+t_n)x\,dt_1\cdots dt_n
\end{equ}
by Proposition~\ref{prop:resolvEqu}.

To show that the converse also holds, we are going to construct the semigroup $S(t)$ by using the so-called
`Yosida approximations' $L_\lambda = \lambda L R_\lambda$ for $L$. Note first that $\lim_{\lambda \to \infty} L R_\lambda x = 0$
for every $x \in \CB$: it obviously holds for $x \in \CD(L)$ since then $\|L R_\lambda x\| = \|R_\lambda L x\| \le \|R_\lambda\| \|Lx\| \le M \,\bigl(\Re\lambda - a\bigr)^{-1} \|Lx\|$. Furthermore, $\|L R_\lambda x\| = \|\lambda R_\lambda x - x\| \le \bigl(M \lambda (\lambda-a)^{-1} + 1\bigr)\|x\| \le (M+2)\|x\|$
for $\lambda$ large enough, so that $\lim_{\lambda \to \infty} L R_\lambda x = 0$  for every $x$ by a standard density argument.

Using this fact, we can show that the Yosida approximation of $L$ does indeed approximate $L$ in the sense
that $\lim_{\lambda \to \infty} L_\lambda x = Lx$ for every $x \in \CD(L)$. Fixing an arbitrary $x\in \CD(L)$, we have
\begin{equ}[e:convL]
\lim_{\lambda \to \infty} \bigl\|L_\lambda x - Lx\bigr\| = \lim_{\lambda \to \infty} \bigl\|\bigl(\lambda R_\lambda - 1\bigr) Lx\bigr\| = 
\lim_{\lambda \to \infty} \bigl\|LR_\lambda Lx\bigr\| = 0\;.
\end{equ}
Define now a family of bounded operators $S_\lambda(t)$ by $S_\lambda(t) = e^{L_\lambda t} =  \sum_{n \ge 0} {t^n L_\lambda^n\over n!}$.
This series converges in the operator norm since $L_\lambda$ is bounded and one can easily check that $S_\lambda$ is indeed
a $\CC_0$-semigroup (actually a group) with generator $L_\lambda$. Since $L_\lambda = -\lambda + \lambda^2 R_\lambda$,
one has for $\lambda > a$ the bound
\begin{equ}[e:boundSlambda]
\|S_\lambda(t)\| = e^{-\lambda t} \sum_{n \ge 0} {t^n \lambda^{2n} \|R_\lambda^n\|\over n!} = M\exp\Bigl(- \lambda t + {\lambda^2\over \lambda - a} t \Bigr) = M\exp\Bigl({\lambda a t\over \lambda - a} \Bigr)\;,
\end{equ}
so that $\limsup_{\lambda \to \infty} \|S_\lambda(t)\| \le M e^{at}$. Let us show next that the limit $\lim_{\lambda \to \infty} S_\lambda(t)x$
exists for every $t \ge 0$ and every $x\in \CB$. Fixing $\lambda$ and $\mu$ large enough so that 
$\max\{\|S_\lambda(t)\|, \|S_\mu(t)\|\} \le Me^{2at}$, and fixing some arbitrary $t > 0$, we have for $s \in [0,t]$
\begin{equs}
\|\d_s S_\lambda({t-s})S_\mu(s)x\| &= \|S_\lambda({t-s}) \bigl(L_\mu - L_\lambda\bigr)S_\mu(s)x\| 
= \|S_\lambda({t-s}) S_\mu(s) \bigl(L_\mu - L_\lambda\bigr)x\| \\
&\le M^2 e^{2at} \|\bigl(L_\mu - L_\lambda\bigr)x\| \;.
\end{equs}
Integrating this bound between $0$ and $t$, we obtain
\begin{equ}[e:CauchySlambda]
\|S_\lambda(t)x - S_\mu(t)x\| \le M^2 t e^{2at} \|L_\mu x - L_\lambda x\|\;,
\end{equ}
which converges to $0$ for every $x \in \CD(L)$ as $\lambda, \mu \to \infty$ since one then has $L_\lambda x \to Lx$.
We can therefore \textit{define} a family of linear operators $S(t)$ by $S(t)x = \lim_{\lambda \to \infty} S_\lambda(t) x$.

It is clear from \eref{e:boundSlambda} that $\|S(t)\| \le M e^{a t}$ and it follows from the semigroup property of $S_\lambda$ that
$S(s) S(t) = S(s+t)$.  Furthermore, it follows from \eref{e:CauchySlambda} and \eref{e:convL} that for every fixed $x\in\CD(L)$,
 the convergence $S_\lambda(t)x \to S(t)x$ is uniform
in bounded intervals of $t$, so that the map $t \mapsto S(t)x$ is continuous. Combining this with our \textit{a priori} bounds
on the operator norm of $S(t)$, it follows from Exercise~\ref{ex:SC}  that $S$ is indeed a $\CC_0$-semigroup. 
It remains to show that the generator $\hat L$ of $S$ coincides with $L$. Taking first the limit $\lambda \to \infty$ and then the limit $t \to 0$ 
in the identity
\begin{equ}
t^{-1} \bigl(S_\lambda(t) x - x\bigr) = t^{-1}\int_0^t S_\lambda(s)L_\lambda x\, ds\;,
\end{equ}
we see that $x \in \CD(L)$ implies $x \in \CD(\hat L)$ and $\hat L x = Lx$, so that $\hat L$ is an extension of $L$. However,
for $\lambda > a$, both $\lambda - L$ and $\lambda - \hat L$ are one-to-one between their domain and $\CB$, so that they must coincide.
\end{proof}

One might think that the resolvent bound in the Hille--Yosida theorem is a consequence of the fact that the spectrum of $L$ is
assumed to be contained in the half plane $\{\lambda\,:\, \Re \lambda \le a\}$. This however isn't the case, as can be seen by the following example:

\begin{example}\label{ex:counterSpec}
We take $\CB = \bigoplus_{n \ge 1}\C^2$ (equipped with the usual Euclidean norms) and we define $L = \bigoplus_{n \ge 1} L_n$, where
$L_n \colon \C^2 \to \C^2$ is given by the matrix
\begin{equ}
L_n = 
\begin{pmatrix}
in & n\\
0 & in
\end{pmatrix}\;.
\end{equ}
In particular, the resolvent $R^{(n)}_\lambda$ of $L_n$ is given by
\begin{equ}
R^{(n)}_\lambda = {1\over (\lambda - in)^2}
\begin{pmatrix}
\lambda - in & n\\
0 & \lambda - in
\end{pmatrix}\;,
\end{equ}
so that one has the upper and lower bounds
\begin{equ}
{n \over |\lambda - in|^2} \le \|R^{(n)}_\lambda\| \le {n \over |\lambda - in|^2} + {\sqrt 2 \over |\lambda - in|}\;.
\end{equ}
Note now that the resolvent $R_\lambda$ of $L$ satisfies $\|R_\lambda\| = \sup_{n \ge 1} \|R_\lambda^{(n)}\|$. On one hand, this shows that the
spectrum of $L$ is given by the set $\{in\,:\, n \ge 1\}$, so that it does indeed lie in a half plane. On the other hand, for every fixed value $a > 0$,
we have $\|R_{a + in}\| \ge {n\over a^2}$, so that the resolvent bound of the Hille--Yosida theorem is certainly not satisfied.

It is therefore not surprising that $L$ does not generate a $\CC_0$-semigroup on $\CB$. Even worse, trying to define $S(t) = \oplus_{n \ge 1} S_n(t)$
with $S_n(t) = e^{L_n t}$ results in $\|S_n(t)\| \ge nt$, so that $S(t)$ is an unbounded operator for every $t > 0$!
\end{example}

\subsubsection{Adjoint semigroups}
\index{adjoint!semigroup}

It will be very useful in the sequel to have a good understanding of the behaviour of the adjoints of strongly continuous semigroups.
The reason why this is not a completely trivial topic is that, in general, it is simply not true
that the adjoint semigroup $S^*(t)\colon \CB^*\to \CB^*$ of a strongly continuous semigroup is again strongly continuous. This is probably best 
illustrated by an example.

Take $\CB = \CC([0,1],\R)$
and let $S(t)$ be the heat semigroup (with Neumann boundary conditions, say). Then $S^*(t)$ acts on finite signed measures
by convolving them with the heat kernel. While it is true that $S^*(t)\mu \to \mu$ weakly as $t\to 0$, it is not true in general
that this convergence is strong. For example, $S^*(t)\delta_x$ does \textit{not} converge to $\delta_x$ in the total variation norm
(which is the dual to the supremum norm on $\CC([0,1],\R)$). However, this difficulty can always be overcome by restricting $S^*(t)$ to a slightly
smaller space than $\CB^*$. This is the content of the following result:

\begin{proposition}\label{prop:SCDual}
If $S(t)$ is a $\CC_0$-semigroup on $\CB$, then $S^*(t)$ is a $\CC_0$-semigroup on the closure $\CB^\dagger$ of
$\CD(L^*)$ in $\CB^*$ and its generator $L^\dagger$ is given by the restriction of $L^*$ to the set $\CD(L^\dagger) = \{x \in \CD(L^*)\,:\, L^* x \in \CB^\dagger\}$.
\end{proposition}

\begin{proof}
We first show that $S^*(t)$ is strongly continuous on $\CB^\dagger$ and we will then identify its generator. Note first that 
it follows from Proposition~\ref{prop:diffSt} that $S^*(t)$ maps $\CD(L^*)$ into itself, so that it does indeed define a family of bounded operators
on $\CB^\dagger$.
Since the norm of $S^*(t)$ is $\CO(1)$ as $t \to 0$ and since $\CD(L^*)$ is dense in $\CB^\dagger$ by definition, it is sufficient to show that
$\lim_{t \to 0} S^*(t)x = x$ for every $x \in \CD(L^*)$. It follows immediately from Proposition~\ref{prop:diffSt} that for $x \in \CD(L^*)$
one has the identity
\begin{equ}
S^*(t)x - x = \int_0^t S^*(s) L^*x\,ds\;,
\end{equ}
from which we conclude that $S^*(t)x \to x$. 

It follows from Proposition~\ref{prop:resolvEqu} that the resolvent $R_\lambda^\dagger$ for $S^*(t)$ on $\CB^\dagger$ is nothing
but the restriction of $R_\lambda^*$ to $\CB^\dagger$. This immediately implies that $\CD(L^\dagger)$ is given by the 
stated expression.
\end{proof}

\begin{remark}
As we saw in the example of the heat semigroup, $\CB^\dagger$ is in general strictly smaller than $\CB^*$.
This fact was first pointed out by Phillips in \cite{Phillips}. In our example, $\CB^*$ consists of all finite signed Borel
measures on $[0,1]$, whereas $\CB^\dagger$ only consists of those measures that have a density with respect to
Lebesgue  measure.
\end{remark}

Even though $\CB^\dagger$ is in general a proper closed subspace of $\CB^*$, it is large enough to be dense in $\CB^*$, when
equipped with the (much weaker) weak-* topology. This the content of our last result in the theory of strongly continuous semigroups:

\begin{proposition}\label{prop:weakstardense}
For every $\ell \in \CB^*$ there exists a sequence $\ell_n \in \CB^\dagger$ such that $\ell_n(x) \to \ell(x)$ for every $x \in \CB$.
\end{proposition}

\begin{proof}
It suffices to choose $\ell_n = n R_n^*\ell$. Since we have $\ell_n \in \CD(L^*)$, it is clear that $\ell_n\in \CB^\dagger$. On the other hand,
we know from the proof of the Hille--Yosida theorem that $\lim_{n \to \infty} \|nR_n x - x\| = 0$ for every $x \in \CB$, from which the claim follows at once.
\end{proof}

\subsection{Semigroups with selfadjoint generators}

\index{semigroup!selfadjoint}
In this section, we consider the particular case of strongly continuous semigroups consisting of self-adjoint operators 
on a Hilbert space $\CH$. The reason why this is an interesting case is that it immediately implies very strong smoothing properties
of the operators $S(t)$ in the sense that for every $t>0$, they map $\CH$ into the domain of arbitrarily high powers of $L$.
Furthermore, it is very easy to obtain explicit bounds on the norm of $S(t)$ as an operator from $\CH$ into one of these domains.
We will then see later in Section~\ref{sec:analSG} on analytic semigroups that most of these properties still hold true for a much larger class
of semigroups.

Let $L$ be
a selfadjoint operator on $\CH$ which is bounded from above. Without loss of generality, we are going to assume that
it is actually negative definite, so that $\scal{x,Lx} \le 0$ for any $x \in \CH$. In this case, we can use functional calculus
(see for example \cite{ReedSimon}, in particular chapter VIII in volume I) to define selfadjoint operators $f(L)$ for any measurable map $f \colon \R \to \R$. This is because the spectral decomposition theorem can be formulated as:

\index{spectral decomposition theorem}
\begin{theorem}[Spectral decomposition]\label{theo:SD}
Let $L$ be a selfadjoint operator on a separable Hilbert space $\CH$. Then, there exists a measure space
$(\CM,\mu)$, an isomorphism $K \colon \CH \to L^2(\CM,\mu)$, and a function $f_L\colon \CM \to \R$ such that
via $K$, $L$ is equivalent to the multiplication operator by $f_L$ on $L^2(\CM,\mu)$. In other words, one has 
$L = K^{-1} f_L K$ and $K\CD(L) = \{g \,:\, f_L g \in L^2(\CM,\mu)\}$.
\end{theorem}

In particular, this allows one to define $f(L) = K^{-1} (f\circ f_L) K$, which has all the nice properties that one would expect
from functional calculus, like for example $(fg)(L) = f(L) g(L)$, $\|f(L)\| \le \sup_{\lambda\in \R} |f(\lambda)|$, etc.
Defining $S(t) = e^{L t}$, it is an exercise to check that $S$ is indeed a $\CC_0$-semigroup with generator $L$
(either use the Hille--Yosida  theorem and make sure that the semigroup constructed there coincides with $S$
 or check ``by hand'' that $S(t)$ is indeed $\CC_0$ with generator $L$).

The important property of semigroups generated by selfadjoint operators is that they do not only leave $\CD(L)$ invariant,
but they have a regularising effect in that they map $\CH$ into the domain of any arbitrarily high power of $L$. More
precisely, one has:

\begin{proposition}\label{prop:boundSelfAdj}
Let $L$ be self-adjoint and negative definite and let $S(t)$ be the semigroup on $\CH$ generated by $L$. Then, 
$S(t)$ maps $\CH$ into the domain of $(1-L)^\alpha$ for any $\alpha, t > 0$ and there exist constants $C_\alpha$
such that
$\|(1-L)^\alpha S(t)\| \le C_\alpha \bigl(1 + t^{-\alpha}\bigr)$.
\end{proposition}

\begin{proof}
By functional calculus, it suffices to show that $\sup_{\lambda \ge 0} (1+\lambda)^\alpha e^{-\lambda t} \le C_\alpha (1+t^{-\alpha})$.
One has
\begin{equ}
\sup_{\lambda \ge 0} \lambda^\alpha e^{-\lambda t} = t^{-\alpha} \sup_{\lambda \ge 0} \; (\lambda t)^\alpha e^{-\lambda t}
=  t^{-\alpha} \sup_{\lambda \ge 0} \lambda^\alpha e^{-\lambda} = \alpha^\alpha e^{-\alpha} t^{-\alpha}\;.
\end{equ}
The claim now follows from the fact that there exists a constant $C_\alpha'$ such that $(1-\lambda)^\alpha \le C_\alpha' \bigl(1 + (-\lambda)^\alpha\bigr)$ for every $\lambda \le 0$.
\end{proof}

\begin{exercise}
Show that if $L$ is self-adjoint and negative definite, then the semigroup $e^{Lt}$
defined using functional calculus (i.e.\ $e^{Lt} = K^{-1} e^{f_L} K$) is indeed a strongly
continuous semigroup with generator $L$. 
\end{exercise}

\begin{exercise}
Show that if the generator of a $\CC_0$ semigroup $S$ is self-adjoint, then $S(t)$ is analytic
on all of $\{\lambda \in \C\,:\, \Re \lambda > 0\}$. 
\end{exercise}

\subsection{Analytic semigroups}
\label{sec:analSG}

\index{semigroup!analytic}
Obviously, the conclusion of Proposition~\ref{prop:boundSelfAdj} does not hold for arbitrary
$\CC_0$-semigroups since the group of translations from Example~\ref{ex:transl} does not have any smoothing properties.
It does however hold for a very large class of semigroups, the so-called
analytic semigroups. The study of these semigroups is the object of the remainder of this section, 
and the equivalent of Proposition~\ref{prop:boundSelfAdj} is going to be one of our two main results.
The other result is a characterisation of generators for analytic semigroups that is analogous to the Hille--Yosida theorem
for $\CC_0$-semigroups. The difference will be that the role of the half-plane $\Re \lambda > a$ will be played by the complement
of a sector of the complex plane with an opening angle strictly smaller than $\pi$.

Recall that a semigroup $S$ on a Banach space $\CB$ is \textit{analytic} if there exists $\theta > 0$ such that the map 
$t \mapsto S(t)$ (taking values in $\CL(\CB)$) admits an analytic extension to the sector
$S_\theta = \{\lambda \in \C \,:\, |{\arg \lambda}| < \theta\}$, satisfies the semigroup property there,
and is such that $t \mapsto S_\phi(t) = S(e^{i\phi} t)$ is a strongly continuous semigroup for every $|\phi| < \theta$.
If $\theta$ is the supremum of all the angles such that the above property holds, we call $S$ \textit{analytic with angle $\theta$}.
The strong continuity of $t \mapsto S(e^{i\phi} t)$ implies that there exist constants $M(\phi)$ and $a(\phi)$ such that
\begin{equ}
\|S_\phi(t)\| \le M(\phi) e^{a(\phi) t}\;.
\end{equ}
Using the semigroup property, it is not difficult to show that $M$ and $a$ can be chosen bounded over compact intervals:

\begin{proposition}
Let $S$ be an analytic semigroup with angle $\theta$. Then, for every $\theta' < \theta$, there exist $M$ and $a$ such that
$\|S_\phi(t)\| \le Me^{a t}$ for every $t > 0$ and every $|\phi| \le \theta'$.
\end{proposition}

\begin{proof}
Fix $\theta' \in (0,\theta)$, so that in particular $\theta' < \pi/2$. Then there exists a constant $C$ such that, 
for every $t > 0$ and every $\phi$ with $|\phi| \le \theta'$, there exist numbers $t_+, t_- \in [0,Ct]$ such that
$t e^{i\phi} = t_+ e^{i\theta'} + t_- e^{-i\theta'}$. It follows that one has the bound
$\|S_\phi(t)\| \le M(\theta') M(-\theta') e^{a(\theta') Ct + a(-\theta') Ct}$, thus proving the claim.
\end{proof}

\begin{exercise}
Give an example of a semigroup $S$ that is strongly continuous, but not analytic.
\end{exercise}

We next compute the generators of the semigroups $S_\phi$ obtained by evaluating $S$ along a ``ray'' extending out of the origin
into the complex plane:

\begin{proposition}\label{prop:genSphi}
Let $S$ be an analytic semigroup with angle $\theta$. Then, for 
$|\phi| < \theta$, the generator $L_\phi$ of $S_\phi$ is given by $L_\phi = e^{i\phi} L$, where $L$ is the generator of $S$.
\end{proposition}

\begin{proof}
Recall Proposition~\ref{prop:resolvEqu} showing that for $\Re \lambda$ large enough the resolvent $R_\lambda$ for $L$
is given by
\begin{equ}
R_\lambda x = \int_0^\infty e^{-\lambda t} S(t)x\, dt\;.
\end{equ}
Since the map $t \mapsto e^{-\lambda t} S(t)$ is analytic in $S_\theta$ by assumption and since, provided again that 
$\Re \lambda$ is large enough, it decays exponentially to $0$ as $|t| \to \infty$, we can deform the contour of integration
to obtain
\begin{equ}
R_\lambda x = e^{i\phi} \int_0^\infty e^{-\lambda e^{i\phi} t} S(e^{i\phi} t)x\, dt\;.
\end{equ}
Denoting by $R_\lambda^\phi$ the resolvent for the generator $L_\phi$ of $S_\phi$, we thus have the identity
$R_\lambda = e^{i\phi} R_{\lambda e^{i\phi}}^\phi$, which is equivalent to $(\lambda - L)^{-1} = (\lambda - e^{-i\phi} L_\phi)^{-1}$,
thus showing that $L_\phi = e^{i\phi} L$ as stated.
\end{proof}

We now use this to show that if $S$ is an analytic semigroup, then the resolvent set of its generator $L$
not only contains the right half plane, but it contains
a larger sector of the complex plane. Furthermore, this characterises the generators of analytic semigroups,
providing a statement similar to the Hille--Yosida theorem:

\begin{theorem}\label{theo:charanal}
A closed densely defined operator $L$ on a Banach space $\CB$ is the generator of an analytic semigroup if and only
if there exists $\theta \in (0, {\pi \over 2})$ and $a \ge 0$ such that the spectrum of $L$ is contained in the sector
\begin{equ}
\CS_{\theta,a} = \{\lambda \in \C\,:\, \arg (a-\lambda) \in [- \textstyle{\pi \over 2}+\theta, \textstyle{\pi \over 2} - \theta]\}\;,
\end{equ}
and there exists $M>0$ such that the resolvent $R_\lambda$ satisfies the bound $\|R_\lambda\| \le M d(\lambda, \CS_{\theta,a})^{-1}$
for every $\lambda \not\in \CS_{\theta,a}$.
\end{theorem}

\begin{proof}
The fact that generators of analytic semigroups are of the prescribed form is a consequence of Proposition~\ref{prop:genSphi}
and the Hille--Yosida theorem. 

To show the converse statement, let $L$ be such an operator, let $\phi \in (0,\theta)$, let $b > a$, and let $\gamma_{\phi,b}$ be 
the curve in the complex plane obtained by going in a counterclockwise way around the boundary of $\CS_{\phi,b}$ (see the figure below).
For $t$ with $|{\arg t}| < \phi$, define $S(t)$ by
\begin{equs}
S(t) &= {1\over 2\pi i} \int_{\gamma_{\phi,b}} e^{t z} R_z\, dz \label{e:defSt} 
= {1\over 2\pi i} \int_{\gamma_{\phi,b}} e^{t z} (z - L)^{-1}\, dz\;.
\end{equs}
\begin{wrapfigure}{r}{6.5cm}
\begin{center}
\mhpastefig[3/2]{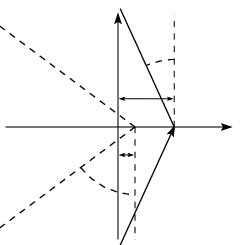}
\end{center}
\vspace{-1cm}
\end{wrapfigure}
It follows from the resolvent bound that $\|R_z\|$ is uniformly bounded for $z \in \gamma_{\phi,b}$. Furthermore, 
since $|{\arg t}| < \phi$, it follows that $e^{tz}$ decays exponentially as $|z| \to \infty$ along $\gamma_{\phi,b}$, so that this
expression is well-defined, does not depend on the choice of $b$ and $\phi$, and
(by choosing $\phi$ arbitrarily close to $\theta$) determines an analytic function $t \mapsto S(t)$ on the sector 
$\{t\,:\, |{\arg t}| < \theta\}$. As in the proof of the Hille--Yosida theorem, the function $(x,t) \mapsto S(t)x$ is jointly continuous because
the convergence of the integral defining $S$ is uniform over bounded subsets of $\{t\,:\, |{\arg t}| < \phi\}$ for any $|\phi| < \theta$.

It therefore remains to show that $S$ satisfies the semigroup property on the sector $\{t\,:\, |{\arg t}| < \theta\}$ and that its generator
is indeed given by $L$. Choosing $s$ and $t$ such that  $|{\arg s}|, |{\arg t}| < \theta$ and using the resolvent
identity $R_z - R_{z'} = (z'-z) R_z R_{z'}$, we have
\begin{equs}
S(s) S(t) &= -{1\over 4\pi^2} \int_{\gamma_{\phi,b'}} \int_{\gamma_{\phi,b}} e^{t z + s z'} R_z R_{z'}\, dz\, dz'
 = -{1\over 4\pi^2} \int_{\gamma_{\phi,b'}} \int_{\gamma_{\phi,b}} e^{t z + s z'} {R_z - R_{z'} \over z'-z}\, dz\, dz' \\
 & = -{1\over 4\pi^2} \int_{\gamma_{\phi,b}} e^{t z} R_z \int_{\gamma_{\phi,b'}} {e^{s z'} \over z'-z}\, dz'\, dz
-{1\over 4\pi^2} \int_{\gamma_{\phi,b'}} e^{s z} R_z \int_{\gamma_{\phi,b}} {e^{t z'} \over z'-z}\, dz'\, dz\;.
\end{equs}
Here, the choice of $b$ and $b'$ is arbitrary, as long as $b \neq b'$ so that the inner integrals are well-defined, say $b' > b$ for definiteness.
In this case, since the contour $\gamma_{\phi,b}$ can be ``closed up'' to the left but not to the right, the 
integral $\int_{\gamma_{\phi,b'}} {e^{sz'} \over z'-z}\, dz'$ is equal to $2i\pi e^{sz}$ for every $z \in \gamma_{\phi,b}$, whereas the integral with 
$b$ and $b'$ inverted vanishes, so that
\begin{equ}
S(s) S(t) = {1\over 2i\pi} \int_{\gamma_{\phi,b}} e^{(t+s) z} R_z  = S(s+t)\;,
\end{equ}
as required. The continuity of the map $t \mapsto S(t)x$ is a straightforward consequence of the resolvent bound, noting that
it arises as a uniform limit of continuous functions.
Therefore $S$ is a strongly continuous semigroup; let us call its generator $\hat L$ and $\hat R_\lambda$
the corresponding resolvent.

To show that $L= \hat L$, it suffices to show that $\hat R_\lambda = R_\lambda$, so 
we make use again of Proposition~\ref{prop:resolvEqu}. Choosing $\Re \lambda > b$ so that $\Re (z-\lambda) < 0$ for every
$z \in \gamma_{\phi,b}$, we have
\begin{equs}
\hat R_\lambda &= \int_0^\infty e^{-\lambda t} S(t)\, dt = {1\over 2\pi i} \int_0^\infty \int_{\gamma_{\phi,b}} e^{t (z-\lambda)} R_z \, dz\, dt \\
&= {1\over 2\pi i} \int_{\gamma_{\phi,b}} \int_0^\infty e^{t (z-\lambda)} \, dt \, R_z \, dz
= {1\over 2\pi i}  \int_{\gamma_{\phi,b}} {R_z \over z-\lambda} \, dz = R_\lambda\;.
\end{equs}
The last equality was obtained by using the fact that $\|R_z\|$ decays like $1/|z|$ for large enough $z$ with $|{\arg z}| \le {\pi \over 2} + \phi$,
so that the contour can be ``closed'' to enclose the pole at $z = \lambda$.
\end{proof}


As a consequence of this characterisation theorem, we can study perturbations of generators of analytic semigroups.
The idea is to give a constructive criterion which allows to make sure that an operator of the type $L = L_0 + B$ is the
generator of an analytic semigroup, provided that $L_0$ is such a generator and $B$ satisfies a type of ``relative total boundedness''
condition. The precise statement of this result is:
\begin{theorem}\label{theo:pertanalytic}
Let $L_0$ be the generator of an analytic semigroup and let $B\colon \CD(B) \to \CB$ be an operator such that
\begin{itemize}
\item The domain $\CD(B)$ contains $\CD(L_0)$.
\item For every $\eps>0$ there exists $C > 0$ such that $\|Bx\| \le \eps \|L_0 x\| + C\|x\|$ for every
$x \in \CD(L_0)$.
\end{itemize}
Then the operator $L = L_0 + B$ with domain $\CD(L) = \CD(L_0)$ is also the generator of an analytic semigroup.
\end{theorem}

\begin{proof}
In view of Theorem~\ref{theo:charanal} it suffices to show that there exists a sector $\CS_{\theta,a}$ containing the spectrum 
of $L$ and such that the resolvent bound $R_\lambda \le M d(\lambda, \CS_{\theta,a})^{-1}$ holds away from it.

Denote by $R_\lambda^0$ the resolvent for $L_0$ and consider the resolvent equation for $L$:
\begin{equ}
(\lambda - L_0 - B)x = y\;,\quad x\in \CD(L_0)\;.
\end{equ}
Since (at least for $\lambda$ outside of some sector) $x$ belongs to the range of $R_\lambda^0$, we can set
$x = R_\lambda^0 z$ so that this 
equation is equivalent to 
\begin{equ}
z - BR_\lambda^0 z = y\;.
\end{equ}
The claim therefore follows if we can show that there exists a sector $\CS_{\theta,a}$ and a constant $c < 1$ such that
$\|B R_\lambda^0\| \le c$ for $\lambda \not \in \CS_{\theta,a}$. This is because one then has the bound
\begin{equ}
\|R_\lambda y\|  = \|R_\lambda^0 z\| \le {\|R_\lambda^0\| \over 1-c} \|y\|\;.
\end{equ}

Using our assumption on $B$, we have the bound
\begin{equ}[e:boundB]
\|BR_\lambda^0 z\| \le \eps \|L_0 R_\lambda^0z\| + C\|R_\lambda^0 z\|\;.
\end{equ}
Furthermore, one has the identity $L_0 R_\lambda^0 = \lambda R_\lambda^0 - 1$ and, since $L_0$ is the generator
of an analytic semigroup by assumption, the resolvent bound
$\|R_\lambda^0\| \le M d(\lambda, \CS_{\alpha,b})^{-1}$ for some parameters $\alpha, b$. Inserting this into 
\eref{e:boundB}, we obtain the bound
\begin{equ}
\|BR_\lambda^0\| \le {(\eps |\lambda| + C) M \over d(\lambda, \CS_{\alpha,b})} + \eps \;.
\end{equ}
Note now that by choosing $\theta \in (0,\alpha)$, we can find some $\delta>0$ such that
$d(\lambda, \CS_{\alpha,b}) > \delta |\lambda|$ for all $\lambda \not \in \CS_{\theta,a}$ and all $a > 1 \vee (b+1)$.
We fix such a $\theta$ and we make $\eps$ sufficiently small such that one has both $\eps < 1/4$ and
$\eps \delta^{-1} < 1/4$.

We can then make $a$ large enough so that $d(\lambda, \CS_{\alpha,b}) \ge 4 CM$ for $\lambda \not \in \CS_{\theta,a}$,
so that $\|BR_\lambda^0\| \le 3/4$. for these values of $\lambda$, as requested.
\end{proof}

\begin{remark}
As one can see from the proof, one actually needs the bound $\|Bx\| \le \eps \|L_0 x\| + C\|x\|$ only for some
particular value of $\eps$ that depends on the properties of $L_0$.
\end{remark}

As a consequence, we have:
\begin{proposition}
Let $f \in L^\infty(\R)$. Then, the operator
\begin{equ}
\bigl(L g\bigr)(x) = g''(x) + f(x) g'(x)\;,
\end{equ}
on $L^2(\R)$ with domain $\CD(L) = H^2$ is the generator of an analytic semigroup.
\end{proposition}

\begin{proof}
It is well-known that the operator $\bigl(L_0 g\bigr)(x) = g''(x)$ with domain $\CD(L) = H^2$ is self-adjoint and
negative definite, so that it is the generator of an analytic semigroup with angle $\theta = \pi/2$.

Setting $Bg = fg'$, we have for $g \in H^2$ the bound
\begin{equ}
\|Bg\|^2 = \int_\R f^2(x) \bigl(g'(x)\bigr)^2\, dx \le \|f\|_{L^\infty}^2 \scal{g',g'}
= - \|f\|_{L^\infty}^2 \scal{g,g''} \le \|f\|_{L^\infty} \|g\| \|L_0g\|\;.
\end{equ}
It now suffices to use the fact that $2|xy| \le \eps x^2 + \eps^{-1} y^2$ to conclude that the assumptions of
Theorem~\ref{theo:pertanalytic} are satisfied.
\end{proof}

Similarly, one can show:
\begin{exercise}
Show that the generator of a uniformly elliptic diffusion with smooth coefficients 
on a compact Riemannian manifold $\CM$
generates an analytic semigroup on $L^2(\CM,\rho)$, where $\rho$ is the volume measure given by the 
Riemannian structure. 
\end{exercise}

\subsection{Interpolation spaces}

\index{interpolation space}
The remainder of this section will be devoted to the study of the domains of fractional powers of the generator $L$
of an analytic semigroup $S(t)$. For simplicity, we will assume \textit{throughout this section} that there exist $M>0$ and $w > 0$ such that
$\|S(t)\| \le Me^{-wt}$, thus making sure that the resolvent set of $L$ contains all the right half of the complex plane.
The general case can be recovered easily by ``shifting the generator to the left''. For $\alpha > 0$, we \textit{define}
negative fractional powers of $L$ by
\begin{equ}[e:fracpower]
(-L)^{-\alpha} \eqdef {1\over \Gamma(\alpha)} \int_0^\infty t^{\alpha - 1} S(t)\, dt\;,
\end{equ}
which is a bounded operator by the decay assumption on $\|S(t)\|$.
Since $\Gamma(1) = 1$, note that if $\alpha = 1$
one does indeed recover the resolvent of $L$ evaluated at $0$. Furthermore, it is straightforward to check that 
one has the identity $(-L)^{-\alpha}(-L)^{-\beta} = (-L)^{-\alpha-\beta}$, which together justify the definition \eref{e:fracpower}.

Note that it follows from this identity that $(-L)^{-\alpha}$ is injective for every $\alpha > 0$.
Indeed, given some $\alpha > 0$, one can find an integer $n>0$ such that $(-L)^{-n} = (-L)^{-n+\alpha}(-L)^{-\alpha}$. A failure for
$(-L)^{-\alpha}$ to be injective would therefore result in a failure for $(-L)^n$ and therefore $(-L)^{-1}$ to be injective.
This is ruled out by the fact that $0$ belongs to the resolvent set of $L$. We can therefore define
$(-L)^\alpha$ as the unbounded operator with domain $\CD((-L)^\alpha) = \range (-L)^{-\alpha}$
given by the inverse of $(-L)^{-\alpha}$. This definition is again consistent with the usual definition of $(-L)^\alpha$ for
integer values of $\alpha$. This allows us to set:

\begin{definition}\index{interpolation|space}
For $\alpha > 0$ and given an analytic semigroup $S$ on a Banach space $\CB$, 
we define the \textit{interpolation space} $\CB_\alpha$ as the domain of $(-L)^\alpha$ endowed with the norm $\|x\|_\alpha = \|(-L)^\alpha x\|$.
We similarly define $\CB_{-\alpha}$ as the completion of $\CB$ for the norm
$\|x\|_{-\alpha} = \|(-L)^{-\alpha} x\|$.
\end{definition}

\begin{remark}
If the norm of $S(t)$ grows instead of decaying with $t$, then we use $\lambda - L$ instead of $-L$ for some $\lambda$
sufficiently large. The choice of different values of $\lambda$ leads to equivalent norms on $\CB_\alpha$. 
\end{remark}

\begin{exercise}
Show that the inclusion $\CB_\alpha \subset \CB_\beta$ for $\alpha \ge \beta$ hold, whatever the signs of $\alpha$ and $\beta$.
\end{exercise}

\begin{exercise}\label{ex:exprLalpha}
Show that for $\alpha \in (0,1)$ and $x \in \CD(L)$, one has the identity
\begin{equ}[e:exprLa]
(-L)^\alpha x = {\sin \alpha \pi \over \pi} \int_0^\infty t^{\alpha - 1} (t - L)^{-1}(-L)x\, dt\;.
\end{equ}
\textbf{Hint:} Write the resolvent appearing in \eref{e:exprLa} in terms of the semigroup and apply the resulting expression to
$(-L)^{-\alpha}x$, as defined in \eref{e:fracpower}. The aim of the game is then to perform a smart change of variables.
\end{exercise}

\begin{exercise}\label{ex:interpBound}
Use \eref{e:exprLa} to show that, for every $\alpha \in (0,1)$, there exists a constant $C$
such that the bound $\|(-L)^\alpha x\| \le C\|Lx\|^\alpha \|x\|^{1-\alpha}$ holds for every $x \in \CD(L)$. \\ \textbf{Hint:} Split the integral 
as $\int_0^\infty = \int_0^K + \int_K^\infty$ and optimise over $K$. (The optimal value for $K$ will turn out to be proportional to $\|Lx\| / \|x\|$.) In the first integral, the identity $(t - L)^{-1}(-L) = 1- t(t-L)^{-1}$ 
might come in handy.
\end{exercise}

\begin{exercise}\label{ex:pertAnal}
Let $L$ be the generator of an analytic semigroup on $\CB$ and denote by $\CB_\alpha$ the corresponding interpolation spaces.
Let $B$ be a (possibly unbounded) operator on $\CB$. Using the results from the previous exercise, show 
that if there exists $\alpha \in [0,1)$ such that $\CB_\alpha \subset \CD(B)$
so that $B$ is a bounded operator from $\CB_\alpha$ to $\CB$, then one has the bound
\begin{equ}
\|Bx\| \le C \bigl(\eps  \|Lx\| + \eps^{-\alpha/(1-\alpha)} \|x\|\bigr)\;,
\end{equ}
for some constant $C>0$ and for all $\eps \le 1$. In particular, $L+B$ is also the generator of an analytic semigroup on $\CB$. \\
\textbf{Hint:} The assumption on $B$ implies that there exists a constant $C$ such that $\|Bx\| \le C \|x\|_\alpha$.
\end{exercise}

\begin{exercise}
Let $L$ and $B$ be as in Exercise~\ref{ex:pertAnal} and denote by $S_B$ the analytic semigroup with generator $L+B$. 
Use the relation $R_\lambda - R_\lambda^0 = R_\lambda^0 B R_\lambda$ to show that one has the identity
\begin{equ}
S_B(t)x = S(t)x + \int_0^t S(t-s) B S_B(s)x \, ds\;.
\end{equ}
\textbf{Hint:} Start from the right hand side of the equation and use an argument similar to that of the proof of Theorem~\ref{theo:charanal}.
\end{exercise}

\begin{exercise}
Show that $(-L)^\alpha$ commutes with $S(t)$ for every $t > 0$ and every $\alpha \in \R$. Deduce that $S(t)$ leaves $\CB_\alpha$ invariant
for every $\alpha > 0$.
\end{exercise}

\begin{exercise}\label{ex:interpAdjoint}
It follows from Theorem~\ref{theo:charanal} that the restriction $L^\dagger$ of the adjoint $L^*$ of the generator of an analytic semigroup on $\CB$
to the ``semigroup dual'' space $\CB^\dagger$ is again the generator of an analytic semigroup on $\CB^\dagger$. 
Denote by $\CB_\alpha^\dagger$ the corresponding
interpolation spaces. Show that one has $\CB_\alpha^\dagger = \CD\bigl((-L^\dagger)^\alpha\bigr) \subset \CD\bigl(\bigl((-L)^\alpha\bigr)^*\bigr) =  \bigl(\CB_{-\alpha}\bigr)^*$ for every $\alpha \ge 0$.
\end{exercise}

We now show that an analytic semigroup $S(t)$ always maps $\CB$ into $\CB_\alpha$ for $t > 0$, so that it has
a ``smoothing effect''. Furthermore, the norm in the domains of integer powers of $L$ can be bounded by:

\begin{proposition}\label{prop:mapInteger}
For every $t>0$ and every integer $k > 0$, $S(t)$ maps $\CB$ into $\CD(L^k)$ and there exists a constant $C_k$ such that
\begin{equ}
\|L^k S(t) x\| \le {C_k\over t^k}
\end{equ}
for every $t \in (0,1]$.
\end{proposition}

\begin{proof}
In order to show that $S$ maps $\CB$ into the domain of every power of $L$, we use \eref{e:defSt}, together with the 
identity $LR_\lambda = \lambda R_\lambda - 1$ which is an immediate consequence of the definition of the resolvent $R_\lambda$
of $L$. Since $ \int_{\gamma_{\phi,b}} e^{t z} dz = 0$ for every $t$ such that $|{\arg t}| < \phi$ and since the domain of $L^k$ is complete
under the graph norm, this shows that $S(t)x\in \CD(L^k)$ and
\begin{equ}
L^kS(t) = {1\over 2\pi i} \int_{\gamma_{\phi,b}} z^k e^{t z} R_z\, dz  \;.
\end{equ}
It follows that there exist positive constants $c_i$ such that
\begin{equ}
\|L^kS(t)\| \le {1\over 2\pi} \int_{\gamma_{\phi,b}} |z|^k |e^{t z}| \|R_z\|\, d|z|
\le c_1 \int_0^\infty (1+x)^k e^{-c_2t(x-c_3)} (1+x)^{-1}dx\;.
\end{equ}
Integrating by parts $k-1$ times, we obtain
\begin{equ}
\|L^kS(t)\| \le {c_4 \over t^{k-1}} \int_0^\infty e^{-c_2 t (x- c_4)}\, dx = {c_5 e^{c_6 t}\over t^k}\;,
\end{equ}
which implies the announced bound.
\end{proof}

It turns out that a similar bound also holds for interpolation spaces with non-integer indices:

\begin{proposition}
For every $t>0$ and every $\alpha > 0$, $S(t)$ maps $\CB$ into $\CB_\alpha$ and there exists a constant $C_\alpha$ such that
\begin{equ}[e:boundStinterp]
\|(-L)^\alpha S(t) x\| \le {C_\alpha \over t^\alpha}
\end{equ}
for every $t \in (0,1]$.
\end{proposition}
\begin{proof}
The fact that $S(t)$ maps $\CB$ into $\CB_\alpha$ follows from Proposition~\ref{prop:mapInteger} since there exists $n$ such
that $\CD(L^n) \subset \CB_\alpha$. 
We assume again that the norm of $S(t)$ decays exponentially for large $t$.
The claim for integer values of $\alpha$ is known to hold by Proposition~\ref{prop:mapInteger}, so
we fix some $\alpha > 0$ which is \textit{not} an integer.
Note first that $(-L)^\alpha = (-L)^{\alpha - [\alpha]-1} (-L)^{[\alpha]+1}$, were we
denote by $[\alpha]$ the integer part of $\alpha$. We thus obtain from \eref{e:fracpower} the identity
\begin{equ}
(-L)^\alpha S(t) = {(-1)^{[\alpha]+1}\over \Gamma([\alpha]-\alpha+1)} \int_0^\infty s^{[\alpha]-\alpha} L^{[\alpha]+1} S(t+s)\, ds\;.
\end{equ}
Using the previous bound for $k=[\alpha]$, we thus get for some $C>0$ the bound
\begin{equ}
\|(-L)^\alpha S(t)\| \le C \int_0^\infty s^{[\alpha]-\alpha} {e^{-w(t+s)} \over (t+s)^{[\alpha]+1}}\, ds \le C t^{-\alpha} \int_0^\infty  {s^{[\alpha]-\alpha} \over (1+s)^{[\alpha]+1}}\, ds
\;,
\end{equ}
where we used the substitution $s \mapsto ts$. Since the last function is integrable for every $\alpha > 0$, the claim follows at once.
\end{proof}

\begin{exercise}\label{ex:ASGab}
Using the fact that $S(t)$ commutes with any power of its generator,
show that $S(t)$ maps $\CB_\alpha$ into $\CB_\beta$ for every $\alpha, \beta \in \R$ and that, for $\beta>\alpha$, there
exists a constant $C_{\alpha, \beta}$ such that $\|S(t)x\|_{\CB_\beta} \le C_{\alpha, \beta} \|x\|_{\CB_\alpha} t^{\alpha-\beta}$ for all $t \in (0,1]$.
\end{exercise}

\begin{exercise}\label{ex:resolvBound}
Using the bound from the previous exercise and the definition of the resolvent, 
show that for every $\alpha \in \R$ and every $\beta \in [\alpha,\alpha+1)$ there exists a constant $C$ such that
the bound $\|(t-L)^{-1}x\|_{\CB_\beta} \le C (1+t)^{\beta-\alpha-1}\|x\|_{\CB_\alpha}$ holds for all $t \ge 0$.
\end{exercise}

\begin{exercise}
Consider an analytic semigroup $S(t)$ on $\CB$ and denote by $\CB_\alpha$ the corresponding interpolation spaces.
Fix some $\gamma \in \R$ and denote by $\hat S(t)$ the semigroup $S$ viewed as a semigroup on $\CB_\gamma$.
Denoting by $\hat \CB_\alpha$ the interpolation spaces corresponding to $\hat S(t)$, show that
one has the identity $\hat \CB_\alpha = \CB_{\gamma + \alpha}$ for every $\alpha \in \R$.
\end{exercise}

Another question that can be answered in a satisfactory way with the help of interpolation spaces is
 the speed of convergence of $S(t)x$ to $x$ as $t\to 0$. We know that
if $x \in \CD(L)$, then $t \mapsto S(t)x$ is differentiable at $t=0$, so that $\|S(t)x - x\| = t \|Lx\| + o(t)$. Furthermore, one can
in general find elements $x \in \CB$ so that the convergence $S(t)x \to x$ is arbitrarily slow. This suggests that
if $x \in \CD((-L)^\alpha)$ for $\alpha \in (0,1)$, one has $\|S(t)x - x\| = \CO(t^\alpha)$. This is indeed the case:

\begin{proposition}\label{prop:HolderCont}
Let $S$ be an analytic semigroup with generator $L$ on a Banach space $\CB$. Then, for every $\alpha \in (0,1)$,
there exists a constant $C_\alpha$, so that the bound
\begin{equ}[e:boundStx]
\|S(t)x - x\| \le C_\alpha t^\alpha \|x\|_{\CB_\alpha}
\end{equ}
holds for every $x \in \CB_\alpha$ and every $t\in (0,1]$.
\end{proposition}

\begin{proof}
By density, it is sufficient to show that \eref{e:boundStx} holds for every $x \in \CD(L)$. For such an $x$, one has
indeed the chain of inequalities
\begin{equs}
\|S(t)x - x\| &= \Bigl\| \int_0^t S(s)Lx\, dx\Bigr\|= \Bigl\| \int_0^t (-L)^{1-\alpha} S(s) (-L)^\alpha x\, dx\Bigr\| \\
&\le C\|x\|_{\CB_\alpha} \int_0^t \bigl\|(-L)^{1-\alpha} S(s)\bigr\|\, dx \le C \|x\|_{\CB_\alpha} \int_0^t s^{\alpha-1}\, ds
= C \|x\|_{\CB_\alpha} t^\alpha\;.
\end{equs}
Here, the constant $C$ depends only on $\alpha$ and changes from one expression to the next.
\end{proof}

We conclude this section with a discussion on the interpolation spaces arising from a perturbed analytic semigroup.
As a consequence of Exercises~\ref{ex:interpBound}, \ref{ex:pertAnal}, and \ref{ex:resolvBound}, we have the following result:

\begin{proposition}\label{prop:interpolationSpaces}
Let $L_0$ be the generator of an analytic semigroup on $\CB$ and denote by $\CB_\gamma^0$ the corresponding interpolation spaces.
Let $B$ be a bounded operator from $\CB_\alpha^0$ to $\CB$ for some $\alpha \in [0,1)$.
Let furthermore $\CB_\gamma$ be the interpolation spaces associated to $L = L_0 + B$. 
Then, one has $\CB_\gamma = \CB_\gamma^0$ for every $\gamma \in [0,1]$.
\end{proposition}

\begin{proof}
The statement is clear for $\gamma = 0$ and $\gamma = 1$. For intermediate values of $\gamma$, we will show that
there exists a constant $C$ such that $C^{-1} \|(-L_0)^\gamma x\| \le \|(-L)^\gamma x\| \le C \|(-L_0)^\gamma x\|$ for every $x \in \CD(L_0)$.

Since the domain of $L$ is equal to the domain of $L_0$, we know that the operator $BR_t$ is bounded for every $t>0$, where
$R_t$ is the resolvent of $L$.
Making use of the identity 
\begin{equ}[e:idenResolv]
R_t = R_t^0 + R_t^0 B R_t\;,
\end{equ}
(where we similarly denoted by $R_t^0$ the resolvent of $L_0$) it then follows from Exercise~\ref{ex:resolvBound} and 
the assumption on $B$ that one has for every $x \in \CB_\gamma^0$ the bound
\begin{equs}
\|B R_t x\| &\le \|B R_t^0 x\| +  \|B R_t^0 B R_t x\| \le C \bigl(\|R_t^0 x\|_{\CB_\alpha^0} +  \|R_t^0 B R_t x\|_{\CB_\alpha^0}\bigr) \\
&\le C(1+t)^{(0\vee(\alpha - \gamma)) -1}\|x\|_{\CB_\gamma^0} + C(1+t)^{\alpha -1}\|B R_t x\| \;.
\end{equs}
It follows that, for $t$ sufficiently large, one has the bound
\begin{equ}[e:boundBR]
\|B R_t x\| \le C (1+t)^{(0\vee (\alpha - \gamma)) -1}\|x\|_{\CB_\gamma^0}\;.
\end{equ}
(Note that this bound is also valid for $\gamma = 0$.)
Since one furthermore has the resolvent identity $R_s = R_t + (t-s)R_s R_t$, this bound can be extended to all $t>0$
by possibly changing the value of the constant $C$.

We now show that $\|(-L)^\gamma x\|$ can be bounded by $\|(-L_0)^\gamma x\|$.
We make use of Exercise~\ref{ex:exprLalpha} to get, for $x \in \CD(L_0)$, the bound
\begin{equs}
\|x\|_{\CB_\gamma} &= C \Bigl\|\int_0^\infty t^{\gamma -1} LR_t x\, dt \Bigr\| \\
&\le C  \Bigl\|\int_0^\infty t^{\gamma -1} L_0 R_t^0 x\, dt \Bigr\| +  C\int_0^\infty t^{\gamma -1} \|(L_0R_t^0+1)BR_t x\|\, dt \\
&\le \|x\|_{\CB_\gamma^0} + C\int_0^\infty t^{\gamma -1} \|BR_t x\|\, dt \\
&\le \|x\|_{\CB_\gamma^0} + C\int_0^\infty t^{\gamma -1}(1+t)^{(0\vee(\alpha - \gamma)) -1}\, dt \|x\|_{\CB_\gamma^0}\;.
\end{equs}
Here, we used again the identity \eref{e:idenResolv} to obtain the first inequality, \eqref{e:exprLa} 
for the second one, and the bound \eref{e:boundBR} in the last step.
Since this integral converges (it decays like $t^{(\alpha \vee \gamma)-2}$ for large $t$), we have obtained the required bound.

In order to obtain the converse bound, we have similarly to before
\begin{equ}
\|x\|_{\CB_\gamma^0} \le \|x\|_{\CB_\gamma} + C\int_0^\infty t^{\gamma -1} \|BR_t x\|\, dt \;.
\end{equ}
Making use of the resolvent identity, this yields for arbitrary $K>0$ the bound
\begin{equs}
\|x\|_{\CB_\gamma^0} &\le \|x\|_{\CB_\gamma} + C\int_0^\infty t^{\gamma -1} \|BR_{t+K} x\|\, dt + C K\int_0^\infty t^{\gamma -1} \|BR_{t+K} R_t x\|\, dt \\
&\le \|x\|_{\CB_\gamma} + C\int_0^\infty t^{\gamma -1}(t+K)^{(0\vee(\alpha - \gamma)) -1}\, dt \Big(\|x\|_{\CB_\gamma^0}+ K \sup_{t\ge 0}\|R_tx\|_{\CB_\gamma^0}\Bigr) \\
&\le  \|x\|_{\CB_\gamma} + C K^{(\alpha\vee\gamma)-1}\bigl( \|x\|_{\CB_\gamma^0}+ K \|x\|\bigr)\;,
\end{equs}
where we used the fact that $$\|R_tx\|_{\CB_\gamma^0}\lesssim \|R_tx\|_{\CB_1^0} \lesssim \|L_0 R_tx\| \lesssim \|L R_tx\| \lesssim \|x\|\;,$$
uniformly over $t\ge 0$.
By making $K$ sufficiently large, the prefactor of the second term can be made smaller than ${1\over 2}$, say, so that the required bound follows
by the usual trick of moving the term proportional to $\|x\|_{\CB_\gamma^0}$ to the left hand side of the inequality.
\end{proof}

\begin{exercise}
Assume that $\CB = \CH$ is a Hilbert space and that the antisymmetric part of $L$ is ``small'' in the sense
that $\CD(L^*) = \CD(L)$ and, for every $\eps>0$ there exists a constant $C$ such that
$\|(L-L^*)x\| \le \eps \|Lx\| + C\|x\|$ for every $x \in \CD(L)$. Show that in this case the space $\CH_{-\alpha}$ 
can be identified with the dual of $\CH_\alpha$ (under the pairing given by the scalar product of $\CH$) for $\alpha \in [0,1]$.
\end{exercise}

It is interesting to note that the range $[0,1]$ appearing in the statement of Proposition~\ref{prop:interpolationSpaces}
is not just a restriction of the technique of proof employed here. There are indeed examples of perturbations of generators
of analytic semigroups of the type considered here which induce changes in the corresponding interpolation spaces $\CB_\alpha$
for $\alpha \not \in [0,1]$. 

Consider for example the case $\CB = L^2([0,1])$ and $L_0 = \Delta$, the Laplacian
with periodic boundary conditions. Denote by $\CB_\alpha^0$ the corresponding interpolation spaces.
Let now $\delta\in (0,1)$ be some arbitrary index and let $g \in \CB$ be such that $g \not \in \CB_\delta^0$.
Such an element $g$ exists since $\Delta$ is an unbounded operator.
Define $B$ as the operator with domain $\CC^1([0,1]) \subset \CB$ given by
\begin{equ}[e:opB]
\bigl(Bf\bigr)(x) = f'(1/2)g(x)\;.
\end{equ} 
It turns out that $\CB_\alpha^0 \subset \CC^1([0,1])$ for $\alpha > 3/4$ (see for example Lemma~\ref{lem:Linfty} below), so that
the assumptions of Proposition~\ref{prop:interpolationSpaces} are indeed satisfied.
Consider now the interpolation spaces of index $1+\delta$. Since we know that $\CB_\delta = \CB_\delta^0$,
we have the characterisations
\begin{equs}
\CB_{1+\delta} &= \bigl\{f \in \CD(\Delta) \,:\, \Delta f + f'(1/2)g \in \CB_\delta^0\bigr\}\;, \\
\CB_{1+\delta}^0 &= \bigl\{f \in \CD(\Delta) \,:\, \Delta f \in \CB_\delta^0\bigr\}\;.
\end{equs}
Since on the other hand $g \not \in \CB_\delta^0$ by assumption, it follows that $\CB_{1+\delta}\cap \CB_{1+\delta}^0$ consists
precisely of those functions in $\CD(\Delta)$ that have a vanishing derivative at $1/2$. In particular, $\CB_{1+\delta}\neq \CB_{1+\delta}^0$.

One can also show that $\CB_{-1/4} \neq \CB_{-1/4}^0$ in the following way. Let $\{f_n\} \subset \CD(L)$ be an arbitrary sequence of elements
that form a Cauchy sequence in $\CB_{3/4}$. Since we have already shown that $\CB_{3/4} = \CB_{3/4}^0$, this implies that $\{f_n\}$ is Cauchy in
$\CB_{3/4}^0$ as well. It then follows from the definition of the interpolation
spaces that the sequence $\{\Delta f_n\}$ is Cauchy in $\CB_{-1/4}^0$ and that the sequence $\{(\Delta + B)f_n\}$ is Cauchy in $\CB_{-1/4}$. 
Assume now by contradiction that $\CB_{-1/4} = \CB_{-1/4}^0$. 

This would entail that both $\{\Delta f_n\}$ and $\{\Delta f_n + Bf_n\}$ are Cauchy in $\CB_{-1/4}$, so that
$\{f_n'(1/2)g\}$ is Cauchy in $\CB_{-1/4}$. This in turn immediately implies that the sequence $\{f_n'(1/2)\}$ must be Cauchy in $\R$. Define now $f_n$ by
\begin{equ}
f_n(x) = \sum_{k=1}^n {\sin(4\pi kx)\over k^{2} \log k}\;.
\end{equ}
It is then straightforward to check that, since $\sum_k (k \log^2 k)^{-1}$ converges, this sequence is Cauchy in $\CB_{3/4}^0$.
On the other hand, we have $f_n'(1/2) = \sum_{k=1}^n (k \log k)^{-1}$ which diverges, thus leading to the required contradiction.

\begin{exercise}
Show, again in the same setting as above, that if $g \in \CB_\delta^0$ for some $\delta > 0$, then 
one has $\CB_\alpha  = \CB_\alpha^0$ for every $\alpha \in [0,1+\delta)$.
\end{exercise}

\begin{remark}
The operator $B$ defined in \eref{e:opB} is not a closed operator on $\CB$. In fact, it is not even closable! 
This is however of no consequence for Proposition~\ref{prop:interpolationSpaces} since the operator $L = L_0 +B$
is closed and this is all that matters.
\end{remark}

\section{Linear SPDEs / Stochastic Convolutions}

We now apply the knowledge gathered in the previous sections to discuss the solution to linear stochastic PDEs.
Most of the material from this section can also be found in one way or the other in the monographs \cite{DPZ92,DPZ96} by Da Prato and Zabczyk. 
The aim of this
section is to define what we mean by the solution to a linear stochastic PDE of the form
\begin{equ}[e:linearSPDE]
dx = Lx\, dt + Q\, dW(t)\;,\quad x(0) = x_0\;,
\end{equ}
where we want $x$ to take values in a separable Banach space $\CB$, $L$ is the generator of
a $\CC_0$ semigroup on $\CB$, $W$ is a cylindrical Wiener process on some Hilbert space $\CK$, and
$Q\colon \CK \to \CB$ is a bounded linear operator.

We do not in general expect $x$ to take values in $\CD(L)$ and we do not even in general expect $QW(t)$ to
be a $\CB$-valued Wiener process, so that the usual way of defining solutions to \eref{e:linearSPDE}
by simply integrating both sides of the identity does not work. However, if we apply some $\ell \in \CD(L^*)$ to both sides
of \eref{e:linearSPDE}, then there is much more hope that the usual definition makes sense. This motivates the following definition:

\begin{definition}
A $\CB$-valued process $x(t)$ is said to be a \index{weak solution}\textit{weak solution} to \eref{e:linearSPDE} if, for every $t>0$,
 $\int_0^t \|x(s)\|\, ds < \infty$ almost surely and the identity
\begin{equ}[e:weakSol]
\scal{\ell, x(t)} = \scal{\ell, x_0} + \int_0^t \scal{L^*\ell, x(s)}\, ds + \int_0^t \scal{Q^*\ell, dW(s)}\;,
\end{equ}
holds almost surely for every $\ell \in \CD(L^*)$.
\end{definition}

\begin{remark}\textbf{(Very important!)}
The term ``weak'' refers to the PDE notion of a weak solution and \textit{not} to the probabilistic notion
of a weak solution to a stochastic differential equation. From a probabilistic point of view, we are always going
to be dealing with strong solutions in these notes, in the sense that \eref{e:linearSPDE} can be solved pathwise
for almost every realisation of the cylindrical Wiener process $W$.

Just as in the case of stochastic ordinary differential equations, there are examples of stochastic PDEs that are sufficiently
irregular so that they can only be solved in the probabilistic weak sense. We will however not consider any such example
in these notes.
\end{remark}

\begin{remark}
The stochastic integral in \eref{e:weakSol} can be interpreted in the sense of Section~\ref{sec:Ito} since the map $Q^*\ell \colon \CK \to \R$
is bounded (and therefore Hilbert--Schmidt since $\R$ is finite-dimensional) for every $\ell \in \CB^*$.
\end{remark}

\begin{remark}
Although separability of $\CB$ was not required in the previous section on semigroup theory, it is again needed
in this section, since many of the results from the section on Gaussian measure theory would not hold otherwise.
\end{remark}

On the other hand, suppose that $f \colon \R_+ \to \CD(L)$ is a continuous function and consider the 
function $x \colon \R_+ \to \CD(L)$ given by $x(t) = S(t)x_0 + \int_0^t S(t-s)f(s)\, ds$, where $S$ is the
$\CC_0$-semigroup generated by $L$. If $x_0 \in \CD(L)$
as well, then this function is differentiable and it is easy to check, using Proposition~\ref{prop:diffSt}, that
it satisfies the differential equation $\d_t x = Lx + f$. Formally replacing $f(s)\, ds$ by $Q\, dW(s)$, this suggests
the following alternative definition of a solution to \eref{e:linearSPDE}:

\begin{definition}
A $\CB$-valued process $x(t)$ is said to be a \index{mild solution}\textit{mild solution} to \eref{e:linearSPDE} if the identity
\begin{equ}[e:mildSol]
x(t) = S(t) x_0 + \int_0^t S(t-s)Q\,dW(s)\;,
\end{equ}
holds almost surely for every $t>0$. The right hand side of \eref{e:mildSol} is also sometimes called a
\index{stochastic convolution}\textit{stochastic convolution}.
\end{definition}

\begin{remark}
By the results from Section~\ref{sec:Ito}, the right hand side of \eref{e:mildSol} makes sense in any 
Hilbert space $\CH$ containing $\CB$ and such
that $\int_0^t \tr \iota S(t-s)QQ^*S(t-s)^*\iota^*\, ds < \infty$,  where $\iota \colon \CB \to \CH$ is the inclusion
map. The statement should then be interpreted as saying that the right hand side belongs to $\CB \subset \CH$
almost surely. In the case where $\CB$ is itself a Hilbert space, \eref{e:mildSol} makes sense if and only if 
$\int_0^t \tr S(t-s)QQ^*S(t-s)^*\, ds < \infty$.
\end{remark}

It turns out that these two notions of solutions are actually equivalent: 

\begin{proposition}\label{prop:mildweak}
If the \index{mild solution}mild solution is almost surely integrable, then it is 
also a weak solution. Conversely, every weak solution is a mild solution. 
\end{proposition}

\begin{proof}
Note first that, by considering the process $x(t) - S(t)x_0$ and using Proposition~\ref{prop:diffSt}, we can assume without loss of generality
that $x_0 = 0$. 

We now assume that the process  $x(t)$ defined by \eref{e:mildSol} takes values in $\CB$ almost surely and
we show that this implies that it satisfies \eref{e:weakSol}.
Fixing an arbitrary $\ell \in \CD(L^\dagger)$, applying $L^*\ell$ to both sides
of \eref{e:mildSol}, and integrating the result between $0$ and $t$, we obtain:
\begin{equs}
\int_0^t \scal{L^*\ell, x(s)}\,ds &= \int_0^t  \int_0^s \scal{L^*\ell,S(s-r)Q\,dW(r)}\,ds = 
\int_0^t  \scal[B]{\int_r^t S^*(s-r)L^*\ell\, ds,Q\,dW(r)}\;.
\end{equs}
Using Proposition~\ref{prop:diffSt} and the fact that, by Proposition~\ref{prop:SCDual},
$S^*$ is a strongly continuous semigroup on $\CB^\dagger$, the closure of $\CD(L^*)$ in $\CB^*$, we obtain
\begin{equs}
\int_0^t \scal{L^*\ell, x(s)}\,ds &= \int_0^t  \scal{S^*(t-r)\ell, Q\, dW(r)} -  \int_0^t  \scal{\ell,Q \,dW(r)}\\
&=  \scal[B]{\ell, \int_0^t S(t-r) Q\,dW(r)}-  \int_0^t  \scal{\ell,Q\,dW(r)} \\
&= \scal{\ell, x(t)} -  \int_0^t  \scal{\ell,Q\,dW(r)}\;,
\end{equs}
thus showing that \eref{e:weakSol} holds for every $\ell \in \CD(L^\dagger)$.
To show that $x$ is indeed a weak solution to \eref{e:linearSPDE}, we have to extend this to every $\ell \in \CD(L^*)$. This however
follows immediately from the fact that $\CB^\dagger$ is weak-* dense in $\CB^*$, which was the content of Proposition~\ref{prop:weakstardense}.

To show the converse, let now $x(t)$ be any weak solution to \eref{e:linearSPDE} (again with $x_0 = 0$). Fix an arbitrary $\ell \in \CD(L^\dagger)$, 
some final time $t > 0$,
and consider the function $f(s) = S^*(t-s)\ell$. Since $\ell \in \CD(L^\dagger)$, it follows from 
Proposition~\ref{prop:diffSt} that this function belongs to $\CE \eqdef \CC([0,t],\CD(L^\dagger)) \cap \CC^1([0,t], \CB^\dagger)$.
We are going to show that one has for such functions the almost sure identity
\begin{equ}[e:identityf]
\scal{f(t), x(t)} = \int_0^t \scal{\dot f(s) + L^*f(s), x(s)}\,ds + \int_0^t \scal{f(s), Q\,dW(s)}\;.
\end{equ}
Since in our case $\dot f(s) + L^*f(s) = 0$, this implies that the identity
\begin{equ}[e:almostmild]
\scal{\ell,x(t)} = \int_0^t \scal{\ell, S(t-s)Q\,dW(s)}\;,
\end{equ}
holds almost surely for all $\ell \in \CD(L^\dagger)$. By the closed graph theorem, $\CB^\dagger$ is large enough to separate 
points in $\CB$.\footnote{Assume that, for some $x, y\in\CB$, we have $\scal{\ell,x} = \scal{\ell,y}$ for every $\ell \in \CD(L^*)$.
We can also assume without loss of generality that the range of $L$ is $\CB$, so that $x = Lx'$ and $y = Ly'$, thus yielding
$\scal{L^*\ell,x'} = \scal{L^*\ell,y'}$. Since $L$ is injective and has dense domain, the 
closed range theorem states that the range of $L^*$ is all of $\CB^*$, so that $x'=y'$ and
thus also $x=y$. } Since $\CD(L^\dagger)$ is dense in $\CB^\dagger$ and since $\CB$ is separable, it follows that countably 
many elements of $\CD(L^\dagger)$ are already sufficient to separate points in $\CB$. 
This then immediately implies from \eref{e:almostmild} that $x$ is indeed a mild solution.

It remains to show that \eref{e:identityf} holds for all $f \in \CE$.
Since linear combinations of  functions of the type $\phi_\ell(s) = \ell \phi(s)$ for $\phi \in \CC^1([0,t],\R)$ and $\ell \in \CD(L^\dagger)$
are dense in $\CE$ (see Exercise~\ref{ex:dense} below) and since $x$ is almost surely integrable, 
it suffices to show that \eref{e:identityf} holds for $f = \phi_\ell$. Since $\scal{\ell, QW(s)}$ is a standard one-dimensional
Brownian motion, we can apply It\^o's formula to $\phi(s) \scal{\ell,x(s)}$, yielding
\begin{equ}
\phi(t) \scal{\ell, x(t)} = \int_0^t \phi(s)\scal{L^*\ell, x(s)} + \int_0^t \dot \phi(s)\scal{\ell, x(s)} + \int_0^t \phi(s)\scal{\ell, Q\,dW(s)}\;,
\end{equ}
which coincides with \eref{e:identityf} as required.
\end{proof}

\begin{remark}
It is actually possible to show that if the right hand side of \eref{e:mildSol} makes sense for some $t$, then it makes sense
for all $t$ and the resulting process belongs almost surely to $L^p([0,T],\CB)$ for every $p$. Therefore, the concepts
of mild and weak solution actually always coincide. This follows from the fact that the covariance of $x(t)$ increases with
$t$ (which is a concept that can easily be made sense of in Banach spaces as well as Hilbert spaces), see for example \cite{DiesJarTon}. 
\end{remark}

\begin{exercise}\label{ex:dense}
Consider the setting of the proof of Proposition~\ref{prop:mildweak}.
Let $f \in \CE = \CC([0,1],\CD(L^\dagger)) \cap \CC^1([0,1], \CB^\dagger)$ and, for $n > 0$, define $f_n$ on the interval
$s\in [k/n, (k+1)/n]$ by cubic spline interpolation:
\begin{equs}
f_n(s) &= f(k/n) (k+1-ns)^2(1+2ns-2k) + f((k+1)/n) (ns-k)^2(3-2ns+2k) \\
&\quad + (ns-k)(k+1-ns)^2 n \bigl(f((k+\textstyle{1\over 2})/n)-f((k-\textstyle{1\over 2})/n)\bigr) \\
&\quad + (ns-k)^2(ns-k-1) n \bigl(f((k+\textstyle{3\over 2})/n)-f((k+\textstyle{1\over 2})/n)\bigr) \;.
\end{equs}
Show that the map  $f_n$ is a finite linear combination of maps of the form $t \mapsto \ell  \phi(t)$ with $\phi \in \CC^1([0,1],\R)$
and $\ell \in \CD(L^\dagger)$, and that $f_n \to f$ in $\CC([0,1],\CD(L^\dagger)) \cap \CC^1([0,1], \CB^\dagger)$.
(\textbf{Note:} for the first and last interval, the above expression involves $f(x)$ for $x \not\in [0,1]$. We use the convention
that $f(x) = f(1) + 2n(x-1)(f(1) - f(1-1/2n))$ for $x \ge 1$, i.e.\ we extrapolate linearly from the values at $1$ and $1-1/2n$, 
and similarly for $x < 0$.) 
\end{exercise}

\subsection{Time and space regularity}
\label{sec:regularity}

In this subsection, we are going to study the space and time regularity of solutions to linear stochastic PDEs.
For example, we are going to see how one can easily derive the fact that the solutions to the stochastic heat equation are
`almost' ${1\over 4}$-H\"older continuous in time and ``almost'' ${1\over 2}$-H\"older continuous in space.
Since we are often going to use the Hilbert-Schmidt norm of a linear operator, we introduce the notation
\begin{equ}
\|A\|_\HS^2 = \tr AA^*\;.
\end{equ}
For most of this section, we are going to make use of the theory of analytic semigroups. However, we start with
a very weak regularity result for the solutions to stochastic PDEs whose linear operator $L$ generates an arbitrary 
$\CC_0$-semigroup:

\begin{theorem}\label{theo:regularitySC}
Let $\CH$ and $\CK$ be separable Hilbert spaces, let $L$ be the generator of a $\CC_0$-semigroup on $\CH$,
let $Q\colon \CK \to \CH$ be a bounded operator and let $W$ be a cylindrical Wiener process on $\CK$. 
Assume furthermore that $\|S(t)Q\|_\HS < \infty$ for every $t > 0$ and that 
there exists $\alpha \in (0,\hf)$ such that $\int_0^1 t^{-2\alpha} \|S(t)Q\|_\HS^2\, dt < \infty$. Then the solution $x$ to \eref{e:linearSPDE}
has almost surely continuous sample paths in $\CH$.
\end{theorem}

\begin{proof}
Note first that $\|S(t+s)Q\|_\HS \le \|S(s)\| \|S(t)Q\|_\HS$, so that  the assumptions of the theorem imply that
$\int_0^T t^{-2\alpha} \|S(t)Q\|_\HS^2\, dt < \infty$ for every $T>0$. 
Let us fix an arbitrary terminal time $T$ from now on.
Defining the process $y$ by
\begin{equ}
y(t) = \int_0^t (t-s)^{-\alpha} S(t-s)Q\, dW(s)\;,
\end{equ}
we obtain the existence of a constant $C$ such that
\begin{equ}
\E \|y(t)\|^2 = \int_0^t (t-s)^{-2\alpha} \|S(t-s)Q\|_\HS^2\, ds = \int_0^t s^{-2\alpha} \|S(s)Q\|_\HS^2\, ds \le C\;,
\end{equ}
uniformly for $t \in [0,T]$. It therefore follows from Fernique's theorem that for every $p > 0$ there exist a constant $C_p$
such that 
\begin{equ}[e:boundyp]
\E \int_0^T \|y(t)\|^p \, dt < C_p\;.
\end{equ}
Note now that there exists a constant $c_\alpha$ (actually $c_\alpha = (\sin 2\pi \alpha)/ \pi$) such that
the identity
\begin{equ}
c_\alpha \int_s^t (t-r)^{\alpha - 1} (r-s)^{-\alpha} \, dr = 1\;,
\end{equ}
holds for every $t>s$. It follows that one has the identity
\begin{equs}
x(t) &= S(t)x_0 + c_\alpha \int_0^t \int_s^t (t-r)^{\alpha - 1} (r-s)^{-\alpha}  S(t-s) \, dr \,Q\, dW(s) \\
&= S(t) x_0 + c_\alpha \int_0^t \int_0^r (t-r)^{\alpha - 1} (r-s)^{-\alpha}  S(t-s) Q\, dW(s)\, dr \\
&= S(t) x_0 + c_\alpha \int_0^t S(t-r) \int_0^r  (r-s)^{-\alpha}  S(r-s) Q\, dW(s)\,(t-r)^{\alpha - 1}\, dr \\
&= S(t) x_0 + c_\alpha \int_0^t S(t-r) y(r)\,(t-r)^{\alpha - 1}\, dr\;. \label{e:factorisation}
\end{equs}
The claim thus follows from \eref{e:boundyp} if we can show that for every $\alpha \in (0,{1\over 2})$ there exists $p>0$ such that 
the map
\begin{equ}
y \mapsto F_y \;,\quad F_y(t) = \int_0^t (t-r)^{\alpha - 1} S(t-r) y(r)\, dr 
\end{equ}
maps $L^p([0,T],\CH)$ into $\CC([0,T],\CH)$. Since the semigroup $t \mapsto S(t)$ is uniformly bounded (in the usual operator norm) 
on any bounded time interval and since
$t \mapsto (t-r)^{\alpha-1}$ belongs to $L^q$ for $q \in [1, 1/(1-\alpha))$, we deduce from H\"older's inequality that there exists a constant $C_T$ such that one does indeed have the bound
$\sup_{t\in[0,T]} \|F_y(t)\|^p \le C_T \int_0^T \|y(t)\|^p\,dt$, provided that $p > {1\over \alpha}$. Since continuous functions are dense in $L^p$, the proof is complete if we can 
show that $F_y$ is continuous for every continuous function $y$ with $y(0) = 0$.

Fixing such a $y$, we first show that $F_y$ is right-continuous and then that it is left continuous. Fixing $t>0$, we have for $h>0$ the bound
\begin{equs}
\|F_y(t+h) - F_y(t)\| &\le \int_0^t \bigl\|\bigl((t+h-r)^{\alpha-1}S(h) - (t-r)^{\alpha-1}\bigr)S(t-r)y(r)\bigr\|\, dr \\
&\quad + \int_t^{t+h} (t+h-r)^{\alpha-1} \|S(t+h-r) y(r)\|\, dr
\end{equs}
The second term is bounded by $\CO(h^\delta)$ for some $\delta>0$ by H\"older's inequality. It follows from
the strong continuity of $S$ that the integrand of the first term converges to $0$ pointwise as $h \to 0$. 
Since on the other hand the integrand is bounded by $C (t-r)^{\alpha-1} \|y(r)\|$ for some constant $C$, 
this term also converges to $0$ by the dominated convergence theorem. This shows that $F_y$ is right continuous.

To show that $F_y$ is also left continuous, we write
\begin{equs}
\|F_y(t) - F_y(t-h)\| &\le \int_0^{t-h} \bigl\|\bigl((t-r)^{\alpha-1}S(h) - (t-h-r)^{\alpha-1}\bigr)S(t-h-r)y(r)\bigr\|\, dr \\
&\quad + \int_{t-h}^{t} (t-r)^{\alpha-1} \|S(t-r) y(r)\|\, dr\;.
\end{equs}
We bound the second term by H\"older's inequality as before. The first term can be rewritten as
\begin{equ}
 \int_0^{t} \bigl\|\bigl((t+h-r)^{\alpha-1}S(h) - (t-r)^{\alpha-1}\bigr)S(t-r)y(r-h)\bigr\|\, dr\;,
\end{equ}
with the understanding that $y(r) = 0$ for $r < 0$. Since we assumed that $y$ is continuous, we can again 
use the dominated convergence theorem to show that this term tends to $0$ as $h \to 0$.
\end{proof}

\begin{remark}
The trick employed in \eref{e:factorisation} is sometimes called the ``factorisation method'' and was introduced in the
context of stochastic convolutions  by Da Prato, Kwapie{\'n}, and Zabczyk \cite{DKZ87:1,DaPZab92:143}.
\end{remark}

This theorem is quite sharp in the sense that, without any further assumption on $Q$ and $L$, it is not possible
in general to deduce that $t\mapsto x(t)$ has more regularity than just continuity, even if we start with a very regular
initial condition, say $x_0 = 0$. We illustrate this fact with the following exercise:

\begin{exercise}
Consider the case $\CH = L^2(\R)$, $\CK= \R$, $L = \d_x$ and $Q = g$ for some compactly supported function $g$ that is smooth outside 
the origin and such that
$g \ge 0$ and $g(x) = |x|^{-\beta}$ for some $\beta \in (0,\hf)$ and all $|x| < 1$. Show that this satisfies the conditions
of Theorem~\ref{theo:regularitySC} for any $\alpha < 1$.

Since $L$ generates the translation group, the solution to
\begin{equ}
du(x,t) = \d_x u(x,t)\, dt + g(x)\,dW(t)\;,\quad u(x,0) = 0\;,
\end{equ}
is given by
\begin{equ}
u(x,t) = \int_0^t g(x+t-s)\, dW(s)\;.
\end{equ}
Convince yourself that for fixed $t$, the map $x \mapsto u(x,t)$ is in general $\gamma$-H\"older continuous for $\gamma < {1\over 2}-\beta$,
but no better. Deduce from this that the map $t \mapsto u(\cdot,t)$ is in general also $\gamma$-H\"older continuous 
for $\gamma < {1\over 2}-\beta$ (if we consider it either as an $\CH$-valued map or as a $\CC_b(\R)$-valued map), but cannot be expected to
have more regularity than that. Since $\beta$ can be chosen arbitrarily close to $\hf$, it follows that the exponent $\alpha$ appearing
in Theorem~\ref{theo:regularitySC} is in general independent of the H\"older regularity of the solution.
\end{exercise}

One of the main insights of regularity theory for parabolic PDEs (both deterministic and stochastic) 
is that space regularity is intimately linked to time regularity
in several ways. Very often, the knowledge that a solution has a certain spatial regularity for fixed time implies that it also has
a certain temporal regularity at a given spatial location. 

From a slightly different point of view, if we consider time-regularity 
of the solution to a PDE viewed as an evolution equation in some infinite-dimensional space of functions, then the amount
of regularity that one obtains depends on  the functional space under consideration. As a general rule, the smaller the space
(and therefore the more spatial regularity it imposes) the lower the regularity of the solution, viewed as a function with values
in that space.

We start by giving a general result that tells us precisely in which interpolation space one can expect to find
the solution to a linear SPDE associated with an analytic semigroup. This provides us with the optimal spatial regularity
for a given SPDE:

\begin{theorem}\label{theo:spaceReg}
Consider \eref{e:linearSPDE} on a Hilbert space $\CH$, assume that $L$ generates an analytic semigroup, and denote by $\CH_\alpha$ the
corresponding interpolation spaces. If there exists $\alpha \ge 0$ such that $Q \colon \CK \to \CH_\alpha$ is bounded and $\beta \in (0,\hf +\alpha]$
such that $\|(-L)^{-\beta}\|_\HS < \infty$ then the solution $x$ takes values in $\CH_\gamma$ for every $\gamma < \gamma_0 = \hf + \alpha - \beta$. 
\end{theorem}

\begin{proof}
As usual, we can assume without loss of generality that $0$ belongs to the resolvent set of $L$.
It suffices to show that
\begin{equ}
I(T) \eqdef \int_0^T \|(-L)^\gamma S(t) Q\|_\HS^2\, dt < \infty\;,\qquad \forall T>0\;.
\end{equ}
Since $Q$ is assumed to be bounded from $\CK$ to $\CH_\alpha$, there exists a constant $C$ such that
\begin{equ}
I(T) \le C \int_0^T \|(-L)^\gamma S(t) (-L)^{-\alpha}\|_\HS^2\, dt  = C \int_0^T \|(-L)^{\gamma-\alpha} S(t)\|_\HS^2\, dt \;.
\end{equ}
Since $(-L)^{-\beta}$ is Hilbert-Schmidt, we have the bound
\begin{equ}
\|(-L)^{\gamma-\alpha} S(t)\|_\HS \le \|(-L)^{-\beta}\|_\HS \|(-L)^{\beta + \gamma-\alpha} S(t)\| \le C \bigl(1 \vee t^{\alpha -\gamma - \beta}\bigr)\;.
\end{equ}
For this expression to be square integrable near $t = 0$, we need $\alpha -\gamma - \beta > -\hf$, which is precisely the stated condition.
%
\end{proof}

\begin{exercise}
Show that if we are in the setting of Theorem~\ref{theo:spaceReg} and $L$ is selfadjoint, then the solutions to \eref{e:linearSPDE} 
actually belong to $\CH_\gamma$ for $\gamma = \gamma_0$.
\end{exercise}

\begin{exercise}\label{ex:SHE}
Show that the solution to the stochastic heat equation on $[0,1]$ with periodic boundary conditions has solutions
in the fractional \index{Sobolev!space}Sobolev space $H^s$ for every $s < 1/2$. Recall that $H^s$ is the Hilbert space with scalar product
$\scal{f,g}_s = \sum_{k} \hat f_k \hat g_k (1+k^2)^s$, where $\hat f_k$ denotes the $k$th Fourier coefficient of $f$. 
\end{exercise}

\begin{exercise}
Consider the following modified stochastic heat equation on $[0,1]^d$ with periodic boundary conditions:
\begin{equ}
dx = \Delta x\, dt + (1-\Delta)^{-\gamma}\, dW\;,
\end{equ}
where $W$ is a cylindrical Wiener process on $L^2([0,1]^d)$. For any given $s \ge 0$, how large does $\gamma$ need to be for
$x$ to take values in $H^s$? Use Corollary~\ref{cor:Kol} to show that in this case
it also takes values in $\CC^s$.
\end{exercise}

Using this knowledge about the spatial regularity of solutions, we can now turn to the time-regularity. We have:

\begin{theorem}\label{theo:timeReg}
Consider the same setting as in Theorem~\ref{theo:spaceReg} and fix $\gamma < \gamma_0$. Then, at all times $t>0$,
the process $x$ is almost surely $\delta$-H\"older continuous in 
$\CH_\gamma$ for every $\delta < \hf \wedge (\gamma_0 - \gamma)$.
\end{theorem}

\begin{proof}
It follows from Kolmogorov's continuity criteria, Proposition~\ref{prop:KolmogorovGauss}, that it suffices to check that the bound
\begin{equ}
\E \|x(t) - x(s)\|_\gamma^2 \le C |t-s|^{1\wedge 2(\tilde \gamma - \gamma)}
\end{equ}
holds uniformly in $s,t \in [t_0, T]$ for every $t_0, T>0$ and for every $\tilde \gamma < \gamma_0$. Here and below, $C$ is an unspecified constant that changes from expression to expression. Assume that $t>s$ from now on. It follows from the semigroup property and the independence of the increments of $W$ that the identity
\begin{equ}[e:Markov]
x(t) = S(t-s)x(s) + \int_s^t S(t-r)Q\,dW(r)\;,
\end{equ}
holds almost surely, where the two terms in the sum are independent. This property is also 
called the \textit{Markov property}\index{Markov!property}. Loosely speaking, it states that the future of $x$ depends on its present, but not on its past.
This transpires in \eref{e:Markov} through the fact that the right hand side depends on $x(s)$ and on the increments of $W$ between times
$s$ and $t$, but it does not depend on $x(r)$ for any $r < s$.

Furthermore, $x(s)$ is independent of the increments of $W$ over the interval $[s,t]$, so that Proposition~\ref{prop:HolderCont}
allows us to get the bound
\begin{equs}
\E \|x(t) - x(s)\|_\gamma^2 &= \E \|S(t-s)x(s) - x(s)\|_\gamma^2 + \int_0^{t-s} \|(-L)^\gamma S(r) Q\|_\HS^2\, dr \\
&\le C |t-s|^{2(\tilde \gamma-\gamma) \wedge 2} \E \|x(s)\|_{\tilde \gamma}^2 + C \int_0^{t-s} \bigl(1 \vee r^{\alpha -\gamma - \beta}\bigr)^2\, dr\;. 
\end{equs}
Here, we obtained the bound on the second term in exactly the same way as in the proof of Theorem~\ref{theo:spaceReg}.
The claim now follows from the fact that $\alpha -\gamma - \beta = (\gamma_0 - \gamma) - \hf$.
\end{proof}

\subsection{Long-time behaviour}
\label{sec:longtimelinear}

This section is devoted to the behaviour of the solutions to \eref{e:linearSPDE}  for large times. 
Let's again start with an example that illustrates some of the possible behaviours.

\begin{example}\label{ex:longtime}
Let $x \mapsto V(x)$ be some smooth ``potential''
and let $\CH = L^2(\R, \exp(-V(x))\,dx)$. Let $S$ denote the translation semigroup (to the right) on $\CH$ and denote its
generator by $-\d_x$. Let us first discuss which conditions on $V$ ensure that $S$ is a strongly continuous semigroup on $\CH$.
It is clear that it is a semigroup and that $S(t)u \to u$ for $u$ any smooth function with compact support. It therefore
only remains to show that $\|S(t)\|$ is uniformly bounded for $t \in [0,1]$ say. We have
\begin{equ}[e:normStu]
\|S(t)u\|^2 = \int u^2(x-t) e^{-V(x)}\,dx = \int u^2(x) e^{-V(x)} e^{V(x)-V(x+t)}\,dx\;.
\end{equ}
This shows that a necessary and sufficient condition for $S$ to be a strongly continuous semigroup on $\CH$ is that,
for every $t>0$, there exists $C_t$ such that $\sup_{x \in \R} (V(x) - V(x+t)) \le C_t$ and such that $C_t$ remains bounded
as $t \to 0$. Examples of potentials leading to a 
$\CC_0$-semigroup are $x$, $\sqrt{1+x^2}$, $\log(1+x^2)$, etc or any increasing function. Note however that the potential $V(x) = x^2$ does
\textit{not} lead to a strongly continuous semigroup. One different way of interpreting this is to consider the unitary transformation
$K\colon u \mapsto \exp(\hf V) u$ from the ``flat'' space $L^2$ into $\CH$. Under this transformation, the generator
$-\d_x$ is turned into
\begin{equ}
- \bigl(K^{-1} \d_x Ku\bigr)(x) = -\d_x u(x) - \hf V'(x) u(x)\;.
\end{equ}
Considering the characterisation of generators of $\CC_0$-semigroups given by the Hille--Yosida theorem, one would 
expect this to be the generator of a strongly continuous semigroup if $V'$ is bounded from below, which is indeed
a sufficient condition.

Let now $V$ be such that $S$ is a $\CC_0$-semigroup and consider the SPDE on $\CH$ given by
\begin{equ}[e:translweight]
du(x,t) = -\d_x u(x,t)\, dt + f(x)\,dW(t)\;,
\end{equ}
where $W$ is a one-dimensional Wiener process and $f$ is some function in $\CH$. 
The solution to \eref{e:translweight} with initial condition
$u_0 = 0$ is given as before by
\begin{equ}[e:soltransl]
u(x,t) = \int_0^t f(x+s-t)\, dW(s)\;.
\end{equ}
If we fix the time $t$, we can make the change of variable $s \mapsto t-s$, so that $u(x,t)$ is equal in
distribution to $\int_0^t f(x-s)\, dW(s)$. 

We see that if $f$ happens to be also square integrable (we will assume that this is the case in the sequel and
we will also assume that $f$ is not identically zero), then \eref{e:soltransl} has a limit in distribution as $t\to \infty$
given by
\begin{equ}[e:solstat]
\tilde u(x) = \int_0^\infty f(x-s)\, dW(s)\;.
\end{equ}
It is however not clear \textit{a priori} that $\tilde u$ does belong to $\CH$. On one hand, we have the bound
\begin{equ}
\E \int_\R \tilde u(x)^2 e^{-V(x)}\, dx = \int_\R \int_0^\infty f^2(x-t) \,dt\,e^{-V(x)}\, dx \le  \int_\R f^2(t) \,dt \int_\R \,e^{-V(x)}\, dx\;,
\end{equ}
thus showing that $\tilde u$ definitely belongs to $\CH$ if $e^{-V}$ has finite mass. On the other hand, there are examples
where $\tilde u \in \CH$ even though $e^{-V}$ has infinite mass. For example, if $f(x) = 0$ for $x \le 0$, then it is necessary and
sufficient to have $\int_0^\infty e^{-V(x)}\, dx < \infty$. Denote by $\nu$ the law of $\tilde u$ for further reference.

Furthermore, if $e^{-V}$ is integrable, there are many measures on $\CH$ that are invariant under the action of the semigroup
$S$. For example, given a function $h\in \CH$ which is periodic with period $\tau$ (that is $S(\tau)h = h$), we can check that
the push-forward of the Lebesgue measure on $[0,\tau]$ under the map $t \mapsto S(t)h$ is invariant under the action of $S$.
This is simply a consequence of the invariance of Lebesgue measure under the shift map.
Given any invariant probability measure $\mu_h$ of this type, let $v$ be an $\CH$-valued random variable with law $\mu_h$ that is
independent of $W$. We can then consider the solution to \eref{e:translweight} with initial condition $v$. Since the
law of $S(t)v$ is equal to the law of $v$ by construction, it follows that the law of the solution converges to
the distribution of the random variable $\tilde u + v$, with the understanding that $\tilde u$ and $v$ are independent.

This shows that in the case $\int e^{-V(x)}\, dx < \infty$, it is possible to construct solutions $u$ to \eref{e:translweight}
such that the law of $u(\cdot\,, t)$ converges to $\mu_h \star \nu$ for any periodic function $h$.
\end{example}

\begin{exercise}\label{ex:unboundedSg}
Construct an example of a potential $V$ such that the semigroup $S$ from the previous example 
is \textit{not} strongly continuous by choosing it such that
 $\lim_{t \to 0} \|S(t)\| = +\infty$, even though
each of the operators $S(t)$ for $t>0$ is bounded! \textbf{Hint:} Choose $V$ of the form $V(x) = x^3 - \sum_{n>0} n W\big(n({x-c_n})\big)$, where 
$W$ is an isolated ``spike'' and $c_n$ are suitably chosen constants. 
\end{exercise}

This example shows that in general, the long-time behaviour of solutions to \eref{e:linearSPDE} may depend on the choice
of initial condition. It also shows that depending on the behaviour of $\CH$, $L$ and $Q$, the law of the
solutions may or may not converge to a limiting distribution in the space in which solutions are considered.

In order to formalise the concept of ``long-time behaviour of solutions'' for \eref{e:linearSPDE}, it is convenient to 
introduce the \textit{Markov semigroup}\index{Markov!semigroup} associated to \eref{e:linearSPDE}. Given a linear SPDE with solutions in $\CB$,
we can define a family $\CP_t$ of bounded linear operators on $\cB_b(\CB)$, the space of Borel measurable
bounded functions from $\CB$ to $\R$ by
\begin{equ}[e:defPt]
\bigl(\CP_t \phi\bigr)(x) = \E \phi \Bigl(S(t)x + \int_0^t S(t-s)Q\, dW(s)\Bigr)\;.
\end{equ}
The operators $\CP_t$ are \textit{Markov operators}\index{Markov!operator} in the sense that the map $A \mapsto \CP_t \one_A(x)$
is a probability measure on $\CB$ for every fixed $x$. In particular, one has $\CP_t \one = \one$
and $\CP_t \phi \ge 0$ if $\phi \ge 0$, that is the operators $\CP_t$ preserve positivity. 
It follows furthermore from \eref{e:Markov} and the independence of the increments of $W$ over disjoint time intervals that $\CP_t$ satisfies the
semigroup property $\CP_{t+s} = \CP_t \circ \CP_s$ for any two times $s,t \ge 0$.

\begin{exercise}
Show that $\CP_t$ maps the space $\CC_b(\CB)$ of continuous bounded functions from $\CB$ to $\R$ into itself, i.e.\ it has the
Feller property.
\end{exercise}


If we denote by $\CP_t(x,\cdot\,)$ the law of $S(t)x + \int_0^t S(t-s)Q\, dW(s)$, then $\CP_t$ can alternatively be represented
as
\begin{equ}
\bigl(\CP_t \phi\bigr)(x) = \int_\CB \phi(y)\,\CP_t(x,dy)\;.
\end{equ}
It follows that its dual $\CP_t^*$ acts on measures with finite total variation by
\begin{equ}
\bigl(\CP_t^* \mu\bigr)(A) = \int_\CB \CP_t(x,A)\,\mu(dx)\;.
\end{equ}
Since it preserves the mass of positive measures, $\CP_t^*$ is a continuous
map from the space $\cP_1(\CB)$ of Borel probability measures on $\CB$ (endowed with the total variation topology) 
into itself. It follows from \eref{e:defPt} and the definition of the dual that $\CP_t\mu$ is nothing but the law at time $t$ of
the solution to \eref{e:linearSPDE} with its initial condition $u_0$ distributed according to $\mu$,
independently of the increments of $W$ over $[0,t]$. With these notations in place, we define:

\begin{definition}
A Borel probability measure $\mu$ on $\CB$ is an \textit{invariant measure}\index{invariant measure} for \eref{e:linearSPDE} if
$\CP_t^*\mu = \mu$ for every $t>0$, where $\CP_t$ is the Markov semigroup associated to solutions
of \eref{e:linearSPDE} via \eref{e:defPt}.
\end{definition}

In the case $\CB = \CH$ where we consider \eref{e:linearSPDE} on a Hilbert space $\CH$, the situations
in which such an invariant measure exists are characterised in the following theorem:

\begin{theorem}\label{theo:char}
Consider \eref{e:linearSPDE} with solutions in a Hilbert space $\CH$ and define the self-adjoint operator $Q_t \colon \CH \to \CH$ by
\begin{equ}
Q_t = \int_0^t S(s)QQ^*S^*(s)\, ds\;.
\end{equ}
Then there exists an invariant measure $\mu$ for \eref{e:linearSPDE} if and only if one of the following two
equivalent conditions are satisfied:
\begin{claim}
\item[1.] There exists a positive definite trace class operator $Q_\infty \colon \CH \to \CH$ such that the identity
$2\Re \scal{Q_\infty L^*x,x} + \|Q^*x\|^2 = 0$ holds for every $x \in \CD(L^*)$.
\item[2.] One has $\sup_{t > 0} \tr Q_t < \infty$.
\end{claim}
Furthermore, any invariant measure is of the form $\nu \star \mu_\infty$, where $\nu$ is a measure on $\CH$ that is invariant
under the action of the semigroup $S$ and $\mu_ \infty$ is the centred Gaussian measure with covariance $Q_\infty$.
\end{theorem}

\begin{proof}
The proof goes as follows. We first show that $\mu$ being invariant implies that 2.\ holds. Then we show that 2.\ implies 1., and we
conclude the first part by showing that 1.\ implies the existence of an invariant measure. 

Let us start by showing that if $\mu$ is an invariant measure for \eref{e:linearSPDE}, then 2.\ is satisfied.
By choosing $\phi(x) = e^{i\scal{h,x}}$ for arbitrary $h \in \CH$, it follows from \eref{e:defPt} that the Fourier transform
of $\CP_t^* \mu$ satisfies the equation
\begin{equ}[e:exprPthat]
\widehat{\CP_t^* \mu}(x) = \hat \mu(S^*(t)x) e^{-{1\over 2}\scal{x, Q_tx}}\;.
\end{equ}
Taking logarithms and using the fact that $|\hat \mu(x)| \le 1$ for  every $x \in\CH$ and every probability measure $\mu$,
 It follows that if $\mu$ is invariant, then 
\begin{equ}[e:boundQt]
\scal{x,Q_t x} \le - 2 \log |\hat \mu(x)|\;,\quad \forall x \in \CH\;,\quad \forall t>0\;.
\end{equ}
Choose now a sufficiently large value of $R>0$ so that $\mu(\|x\| > R) < 1/8$ (say) and define
a symmetric positive definite operator $A_R \colon \CH \to \CH$ by
\begin{equ}
\scal{h,A_R h} = \int_{\|x\| \le R} |\scal{x,h}|^2\, \mu(dx)\;.
\end{equ}
Since, for any orthonormal basis, one has $\|x\|^2 = \sum_n |\scal{x, e_n}|^2$, it follows that $A_R$ is trace class and that
$\tr A_R \le R^2$. Furthermore, one has the bound
\begin{equ}
|1-\hat \mu(h)| \le \int_\CH \bigl|1- e^{i\scal{h,x}}\bigr|\, \mu(dx) \le \sqrt{\scal{h,A_Rh}} + {1\over 4}\;.
\end{equ}
Combining this with \eref{e:boundQt}, it follows that $\scal{x,Q_t x}$ is bounded by $2\log 4$ for every $x\in \CH$
such that $\scal{x,A_R x} \le 1/4$ so that, by homogeneity,
\begin{equ}
\scal{x,Q_t x} \le (8\log 4) \scal{x,A_R x}\;.
\end{equ}
It follows that $\tr Q_t \le (8\log 4) R^2$, so that 2.\ is satisfied. To show that 2.\ implies 1., note that 
$\sup \tr Q_t < \infty$ implies that 
\begin{equ}
Q_\infty = \int_0^\infty S(t)QQ^*S^*(t)\, dt\;,
\end{equ}
is a well-defined positive definite trace class operator (since $t \mapsto Q_t^{1/2}$ forms a 
Cauchy sequence in the space of Hilbert-Schmidt operators). Furthermore, one has the identity
\begin{equ}
\scal{x,Q_\infty x} = \scal{S^*(t)x, Q_\infty S^*(t)x} + \int_0^t \|Q^*S^*(s) x\|^2\,ds\;.
\end{equ}
for $x \in \CD(L^*)$, both terms on the right hand side of this expression are differentiable. Taking the derivative
at $t=0$, we get
\begin{equ}
0 = 2 \Re \scal{Q_\infty L^*x,x} + \|Q^* x\|^2\;,
\end{equ}
which is precisely the identity  in 1.

Let now $Q_\infty$ be a given operator as in 1., we want to show that the centred Gaussian measure
 $\mu_\infty$ with covariance $Q_\infty$ is indeed invariant for $\CP_t$.
For $x \in \CD(L^*)$, it follows from Proposition~\ref{prop:diffSt} that
the map $F_x \colon t \mapsto \scal{Q_\infty S^*(t) x, S^*(t)x}$ is differentiable
with derivative given by $\d_t F_x(t) = 2\Re\scal{Q_\infty L^* S^*(t) x, S^*(t)x}$. It follows that
\begin{equs}
F_x(t) - F_x(0) = 2\int_0^t \Re\scal{Q_\infty L^* S^*(s) x, S^*(s)x}\,ds =  -\int_0^t \|Q^*S^*(s) x\|^2\,ds \;,
\end{equs} 
so that one has the identity
\begin{equ}
Q_\infty = S(t) Q_\infty S^*(t) + \int_0^t S(s) QQ^* S^*(s)\, ds = S(t)Q_\infty S^*(t) + Q_t\;.
\end{equ}
Inserting this into \eref{e:exprPthat}, the claim follows. Here, we used the fact that $\CD(L^*)$ is dense in
$\CH$, which is the case for every densely defined closed linear operator on a 
Hilbert space, see \cite[Thm~VII.2.3]{Yos95}.

Since it is obvious from \eref{e:exprPthat} that every measure of the type $\nu \star \mu_ \infty$ with 
$\nu$ invariant for $S$ is also invariant for $\CP_t$, it remains to show that the converse also holds.
Let $\mu$ be invariant for $\CP_t$ and define $\mu_t$ as the push-forward of $\mu$ under the map $S(t)$.
Since $\hat \mu_t(x) = \hat \mu(S^*(t)x)$, it follows from \eref{e:exprPthat} and the invariance of $\mu$ that there exists a function
$\psi \colon \CH \to \R$ such that $\hat \mu_t(x) \to \psi(x)$ uniformly on bounded sets, $\psi \circ S(t)^* = \psi$, and such that
$\hat \mu(x) = \psi(x) \exp(-\hf \scal{x,Q_\infty x})$. It therefore only remains to show that there exists a probability measure
$\nu$ on $\CH$ such that $\psi = \hat \nu$. 

In order to show this, it suffices to show that the family of measures $\{\mu_t\}$ is tight, that is for every $\eps>0$
there exists a compact set $K$ such that $\mu_t(K) \ge 1-\eps$ for every $t$. Prokhorov's theorem \cite[p.~37]{Bill68} then ensures
the existence of a sequence $t_n$ increasing to $\infty$ and a measure $\nu$ such that $\mu_{t_n} \to \nu$
weakly. In particular, $\hat \mu_{t_n}(x) \to \hat \nu(x)$ for every $x \in \CH$, thus concluding the proof.

To show tightness, denote by $\nu_t$ the centred Gaussian measure on $\CH$ with covariance $Q_t$ and
note that one can find a sequence of bounded linear operators $A_n \colon \CH \to \CH$ with the following properties:
\begin{claim}
\item[$a.$] One has $\|A_{n+1} x\| \ge \|A_n x\|$ for every $x\in \CH$ and every $n \ge 0$.
\item[$b.$] The set $B_R = \{x\,:\, \sup_n \|A_n x\| \le R\}$ is compact for every $R>0$.
\item[$c.$] One has $\sup_n \tr A_n Q_\infty A_n^* < \infty$.
\end{claim}
(By diagonalising $Q_\infty$, the construction of such a family of operators is similar to the construction, given
a positive sequence $\{\lambda_n\}$ with $\sum_n \lambda_n < \infty$, of a positive sequence $a_n$ with
$\lim_{n\to \infty} a_n = +\infty$ and $\sum_n a_n \lambda_n < \infty$.) Let now $\eps>0$ be arbitrary. It follows from
Prokhorov's theorem that there exists a compact set $\hat K \subset \CH$ such that $\mu(\CH \setminus \hat K) \le {\eps\over 2}$.
Furthermore, it follows from property $c.$ above and the fact that $Q_\infty \ge Q_t$ that there exists $R>0$ such that 
$\nu_t(\CH \setminus B_R) \le {\eps \over 2}$. Define a set $K\subset \CH$ by
\begin{equ}
K = \{z-y\,:\, z\in \hat K\;,\, y \in B_R\}\;.
\end{equ}
It is straightforward to check, using the Heine-Borel theorem, that $K$ is precompact.

If we now take $X$ and $Y$ to be independent $\CH$-valued random variables
with laws $\mu_t$ and $\nu_t$ respectively, then it follows from the definition of a mild solution and the invariance
of $\mu$ that $Z = X+Y$ has law $\mu$. Since one has the obvious implication
$\{Z \in \hat K\} \& \{Y \in B_R\} \Rightarrow \{X \in K\}$, it follows that
\begin{equ}
\mu_t(\CH \setminus K) = \P(X \not \in K) \le \P(Z \not \in \hat K) + \P(Y \not \in B_R) \le \eps\;,
\end{equ}
thus showing that the sequence $\{\mu_t\}$ is tight as requested.
\end{proof}

It is clear from Theorem~\ref{theo:char} that if \eref{e:linearSPDE} does have a solution in some Hilbert space $\CH$ and
if $\|S(t)\| \to 0$ as $t \to \infty$ in that same Hilbert space, then it also possesses a unique invariant measure on $\CH$.
It turns out that as far as the ``uniqueness'' part of this statement is concerned, it is sufficient to have
$\lim_{t \to \infty} \|S(t)x\| = 0$ for every $x \in \CH$:

\begin{proposition}\label{prop:weakconv}
If $\lim_{t \to \infty} \|S(t)x\| = 0$ for every $x \in \CH$, then \eref{e:linearSPDE} can have at most one invariant measure.
Furthermore, if an invariant measure $\mu_\infty$ exists in this situation, then one has $\CP_t^*\nu \to \mu_\infty$ 
weakly for every probability measure $\nu$ on $\CH$.
\end{proposition}
\begin{proof}
In view of Theorem~\ref{theo:char}, the first claim follows if we show that $\delta_0$ is the only measure that is invariant under the action
of the semigroup $S$. Let $\nu$ be an arbitrary probability measure on $\CH$ such that $S(t)^*\nu = \nu$ for every $t>0$
and let $\phi \colon \CH \to \R$ be a bounded continuous function. On then has indeed
\begin{equ}[e:Ptnu]
\int_\CH \phi(x) \nu(dx) = \lim_{t \to \infty} \int_\CH \phi(S(t)x) \nu(dx) = \phi(0)\;, 
\end{equ}
where we first used the invariance of $\nu$ and then the dominated convergence theorem.

To show that $\CP_t^*\nu \to \mu_\infty$ whenever an invariant measure exists we use the fact that in this case, by Theorem~\ref{theo:char},
one has $Q_t \uparrow Q_\infty$ in the trace class topology. Denoting by $\mu_t$ the centred Gaussian measure with covariance
$Q_t$, the fact that $L^2$ convergence implies weak convergence then implies that there exists a measure
$\hat \mu_\infty$ such that $\mu_t \to \hat \mu_\infty$ weakly.
Furthermore, the same reasoning as in \eref{e:Ptnu} shows that $S(t)^*\nu \to \delta_0$ weakly as $t \to\infty$.
The claim then follows from the fact that $\CP_t^*\nu = (S(t)^*\nu)\star \mu_t$ and from the fact that convolving two probability measures
is a continuous operation in the topology of weak convergence.
\end{proof}

Note that the condition $\lim_{t \to \infty} \|S(t)x\| = 0$ for every $x$ is \textit{not} sufficient in general to guarantee
the existence of an invariant measure for \eref{e:linearSPDE}. This can be seen again with the aid of Example~\ref{ex:longtime}.
Take an increasing function $V$ with $\lim_{x \to \infty} V(x) = \infty$, but such that $\int_0^\infty e^{-V(x)}\,dx = \infty$.
Then, since $\exp \bigl(V(x) - V(x+t)\bigr) \le 1$ and
$\lim_{t \to \infty}\exp \bigl(V(x) - V(x+t)\bigr) = 0$ for every $x \in \R$, it follows from \eref{e:normStu} and the dominated 
convergence theorem that $\lim_{t \to \infty} \|S(t)u\| = 0$ for every $u \in \CH$. However, the fact that 
$\int_0^\infty e^{-V(x)}\,dx = \infty$ prevents the random process $\tilde u$ defined in \eref{e:solstat} from belonging to $\CH$,
so that \eref{e:translweight} has no invariant measure in this particular situation.

\begin{exercise}
Show that if \eref{e:linearSPDE}  has an invariant measure $\mu_\infty$ as in Theorem~\ref{theo:char} 
but there exists $x \in \CH$ such that
$\limsup_{t \to \infty} \|S(t)x\| > 0$, then one cannot have $\CP_t^*\delta_x \to \mu_\infty$ weakly. In this sense, the statement
of Proposition~\ref{prop:weakconv} is sharp.
\end{exercise}

\begin{exercise}
Let $\CB$ be the space of continuous functions $\phi \colon \R \to \R$ with $\lim_{|x| \to \infty} \phi(x) = 0$ and
$\|\phi\| = \sup_x |\phi(x)|$. 
Let $S$ be the semigroup on $\CB$ given $(S(t)\phi)(x) = \phi(e^t x)$. Show that even though one has
$\|S(t)\phi\| = \|\phi\|$ for every $t>0$ and $\phi\in \CB$, the only probability measure $\mu$ on $\CB$ such that $S(t)^*\mu = \mu$
for all $t>0$ is given by $\mu = \delta_0$.
\end{exercise}

\subsection{Convergence in other topologies}

Proposition~\ref{prop:weakconv} shows that if \eref{e:linearSPDE} has an invariant measure $\mu_\infty$, one can in many cases 
expect to have $\CP_t^*\nu \to \mu_\infty$ weakly for every initial measure $\nu$. It is however not clear \textit{a priori} whether such
a convergence also holds in some stronger topologies on the space of probability measures. If we consider the finite-dimensional
case (that is $\CH = \R^n$ for some $n>0$), the situation is clear: the condition $\lim_{t \to \infty} \|S(t)x\| = 0$ for every $x\in \CH$
then implies that $\lim_{t \to \infty} \|S(t)\| = 0$, so that $L$ has to be a matrix whose eigenvalues all have strictly negative real parts.
One then has:

\begin{proposition}\label{prop:convstrong}
In the finite-dimensional case, assume that all eigenvalues of $L$ strictly negative real parts and that $Q_\infty$ has full rank.
Then, there exists $T>0$ such that $\CP_t^*\delta_x$ has a smooth density $p_{t,x}$ with respect to Lebesgue measure
for every $t > T$. Furthermore, $\mu_\infty$ has a smooth density $p_\infty$ with respect to Lebesgue measure and there exists $c>0$
such that, for every $\lambda>0$, one has
\begin{equ}
\lim_{t\to \infty} e^{ct} \sup_{y \in \R^n} e^{\lambda |y|} |p_\infty(y) - p_{t,x}(y)| = 0\;.
\end{equ}
In other words, $p_{t,x}$ converges to $p_\infty$ exponentially fast in any weighted norm with exponentially increasing weight.
\end{proposition}

The proof of Proposition~\ref{prop:convstrong} is left as an exercise. It follows in a straightforward way from the explicit expression
for the density of a Gaussian measure.

In the infinite-dimensional case, the situation is much less straightforward. The reason is that there exists no natural
reference measure (the equivalent of the Lebesgue measure) with respect to which one could form densities.

In particular, even though
 one always has $\|\mu_t - \mu_\infty\|_\infty \to 0$ in the finite-dimensional case (provided that $\mu_\infty$ exists and that all
 eigenvalues of $L$ have strictly negative real part), one cannot expect this to
 be true in general. Consider for example the SPDE
\begin{equ}
dx = -x\, dt + Q\, dW(t)\;,\quad x(t)\in \CH\;,
\end{equ}
where $W$ is a cylindrical process on $\CH$ and $Q\colon \CH \to \CH$ is a Hilbert-Schmidt operator.
One then has 
\begin{equ}
Q_t = {1- e^{-2t}\over 2} Q Q^*\;, \qquad Q_\infty = {1\over 2}Q Q^*\;.
\end{equ}
Combining this with Proposition~\ref{prop:singDilat} (dilates of an infinite-dimensional Gaussian measure are mutually singular) 
shows that if $QQ^*$ has infinitely many non-zero eigenvalues, then
$\mu_t$ and $\mu_\infty$ are mutually singular in this case.

Recall from \eqref{e:TVCoupl} that the total variation
distance between two probability measures $\mu$ and $\nu$ on a separable Banach space $\CB$ is given by
\begin{equ}[e:TVCouplRecall]
\|\mu-\nu\|_{\TV} = 2\inf_{\pi \in \cC(\mu,\nu)} \pi(\{x \neq y\})\;,
\end{equ}
where the infimum runs over the set $\cC(\mu,\nu)$ of all probability measures $\pi$ on $\CB \times \CB$ with marginals $\mu$ and $\nu$. 
This yields a straightforward interpretation to the total variation convergence  $\CP_t^\nu \to \mu_\infty$: for large times, 
a sample drawn from the invariant distribution is with high probability indistinguishable from a sample drawn from the 
Markov process at time $t$. Compare this with the notion of \index{weak convergence}{weak convergence} which relies on the topology
of the underlying space and only asserts that the two samples are close with high probability in the sense determined by the
topology in question.
For example, $\|\delta_x - \delta_y\|$ is always equal to $2$ if $x \neq y$,
whereas $\delta_x \to\delta_y$ weakly if $x \to y$.


We also recall that an alternative characterisation of the total variation norm is as the dual norm to the supremum norm on the space
$\B_b(\CB)$ of bounded Borel measurable functions on $\CB$:
\begin{equ}[e:charTV]
\|\mu - \nu\|_\TV = \sup \Bigl\{\int \phi(x)\mu(dx) - \int \phi(x)\nu(dx)\,:\, \sup_{x \in \CB}|\phi(x)| \le 1\Bigr\}\;.
\end{equ}
The following generalisation of the Perron--Frobenius theorem (sometimes also called Doeblin's criterion)
is left as an exercise.

\begin{exercise}
Let $\CX$ be a Polish space, let $P \colon \CX \to \cP(\CX)$ be a measurable map, and let $\CP^*$ be the
operator on $\cP(\CX)$ given by
\begin{equ}
\big(\CP^* \mu\big)(A) = \int_{\CX}P(x,A)\,\mu(dx)\;.
\end{equ}
Assume that there exists $\delta > 0$ and a probability measure $\nu \in \cP(\CX)$ such that the bound
$P(x,\cdot) \ge \delta \nu$ holds uniformly over $x \in \CX$. Show that there then exists a unique
$\mu_\infty \in \cP(\CX)$ such that $\CP^*\mu_\infty = \mu_\infty$.
\end{exercise}

Doeblin's condition is extremely restrictive and tends to be satisfied mainly for Markov processes
with compact state space. A much more useful criterion is given by a ``local'' version of Doeblin's condition,
combined with a suitable Lyapunov criterion. 
For this, instead of showing directly that $\CP_t^* \nu \to \mu_\infty$ in the total variation norm, we aim for conditions
under which one has $\CP_t^* \nu \to \mu_\infty$ in a weighted total variation norm as in Exercise~\ref{ex:TVWeight}, 
which is actually slightly \textit{stronger} than the
usual total variation norm. Given a weight function $V\colon \CB \to \R_+$, we define a weighted
supremum norm on measurable functions by
\begin{equ}
\|\phi\|_V = \sup_{x \in \CB}{|\phi(x)| \over 1+V(x)}\;,
\end{equ}
as well as the dual norm on measures by
\begin{equ}[e:defTVV]
\|\mu - \nu\|_{\TV,V} = \sup \Bigl\{\int \phi(x)\mu(dx) - \int \phi(x)\nu(dx)\,:\, \|\phi\|_V \le 1\Bigr\}\;.
\end{equ}
Since we assumed that $V > 0$, it is obvious from \eqref{e:charTV} that one has the relation $\|\mu - \nu\|_\TV \le \|\mu - \nu\|_{\TV,V}$, so that
convergence in the weighted norm immediately implies convergence in the usual total variation norm.
By considering the Jordan decomposition of $\mu - \nu = \rho_+ - \rho_-$, it is clear that the supremum in \eref{e:defTVV} is attained 
at functions $\phi$ such that $\phi(x) = 1+V(x)$ for $\rho_+$-almost every $x$ and $\phi(x) = -1-V(x)$ for $\rho_-$-almost 
every $x$. In other words, an alternative expression for the weighted total variation norm is given by
\begin{equ}[e:exprTVV]
\|\mu - \nu\|_{\TV,V} = \int_\CX \bigl(1+V(x)\bigr) \,|\mu-\nu|(dx)\;,
\end{equ}
just like the total variation norm is given by $\|\mu - \nu\|_\TV = |\mu - \nu|(\CX)$.

The reason why it turns out to be easier to work in a weighted norm is the following: For a suitable choice of $V$,
we are going to see that in a large class of examples, one can construct a weight function $V$ and find constants $c < 1$ and $T>0$
such that 
\begin{equ}[e:contraction]
\|\CP_T^* \mu - \CP_T^*\nu\|_{\TV,V} \le c \|\mu - \nu\|_{\TV,V}\;,
\end{equ}
for any two probability measures $\mu$ and $\nu$. This implies that the map $\CP_T$ is a contraction on the space
of probability measures, which must therefore have exactly one fixed point, yielding both the existence of an invariant measure $\mu_\infty$
and the exponential convergence of $\CP_t^* \nu$ to $\mu_\infty$ for every initial probability measure $\nu$ 
which integrates $V$.

This argument is based on the following abstract result that works for arbitrary Markov semigroups on Polish (that is 
separable, complete, metric) spaces:\index{Harris's theorem}

\begin{theorem}[Harris]\label{theo:Harris} 
Let $\CP_t$ be a Markov semigroup on a Polish space $\CX$ such that there exists a time $T_0>0$ and a 
function $V \colon \CX \to \R_+$  such that:
\begin{claim}
\item The exist constants $\gamma < 1$ and $K > 0$ such that $\CP_{T_0} V(x) \le \gamma V(x) + K$ for every $x \in \CX$.
\item For every $K' > 0$, there exists $\delta > 0$ such that $\|\CP_{T_0}^*\delta_x - \CP_{T_0}^*\delta_y\|_\TV \le 2-\delta$
for every pair $x,y$ such that $V(x) + V(y) \le K'$.
\end{claim}
Then, there exists $T>0$ such that \eref{e:contraction} holds for some $c < 1$.
\end{theorem}

In a nutshell, the argument for the proof of  Theorem~\ref{theo:Harris} is the following. There are two mechanisms that 
allow to decrease the weighted total variation distance between two probability measures:
\begin{claim}
\item[2.] The mass of the two measures moves into regions where the weight $V(x)$ becomes smaller.
\item[1.] The two measures ``spread out'' in such a way that there is an increase in the overlap between them.
\end{claim}
The two conditions of Theorem~\ref{theo:Harris} are tailored such as to combine these two effects in order to 
obtain an exponential convergence of $\CP_t^* \mu$ to the unique invariant measure for $\CP_t$ as $t \to \infty$.

\begin{exercise}
Let $\xi_n$ be an i.i.d.\ sequence of real-valued random variables such that $\E |\xi_n|^\alpha < \infty$ for some
$\alpha > 0$ and such that the law of $\xi_n$ has a continuous density with respect to Lebesgue measure.
Consider the Markov process $X_n$ given by $X_{n+1} = \f12 X_n + \xi_n$. Show that its transition 
probabilities do satisfy the conditions of Harris' theorem but do not satisfy Doeblin's condition.
\end{exercise}

\begin{remark}
The condition that there exists $\delta > 0$ such that $\|\CP_{T_0}^*\delta_x - \CP_{T_0}^*\delta_y\|_\TV \le 2-\delta$
for any $x,y \in A$ is sometimes referred to in the literature as the set $A$ being a \textit{small set}. 
\end{remark}

\begin{remark}
Traditional proofs of Theorem~\ref{theo:Harris} as given for example in \cite{MT93} tend to make use of \index{coupling}{coupling} arguments
and estimates of return times of the Markov process described by $\CP_t$ to level sets of $V$. 
The basic idea is to make use of \eref{e:TVCoupl} to get a bound on the total variation between $\CP_T^*\mu$ and
$\CP_T^*\nu$ by constructing an explicit coupling between two instances $x_t$ and $y_t$ 
of a Markov process with transition 
semigroup $\{\CP_t\}$. Because of the second assumption in Theorem~\ref{theo:Harris}, one can construct this coupling
in such a way that every time the process $(x_t, y_t)$ returns to some sufficiently large level set of $V$,
there is a probability $\delta$ that $x_{t'} = y_{t'}$ for $t' \ge t+T_0$. The first assumption then guarantees that
these return times have exponential tails and a renewal-type argument allows to conclude.

Such proofs are quite involved
at a technical level and are by consequent not so easy to follow, especially if one wishes to get a
\index{spectral gap}{spectral gap} bound
like \eref{e:contraction} and not ``just'' an exponential decay bound like
\begin{equ}
\|\CP_T^*\delta_x - \CP_T^*\delta_y\|_\TV \le C e^{-\gamma T}\;,
\end{equ}
with a constant $C$ depending on $x$ and $y$. Furthermore, they require more background in advanced
probability theory than what is assumed for the scope of these notes.

The elementary proof
given here is taken from \cite{HaiMat08} and is based on the arguments first exposed in \cite{HaiMat06SGW}.
It has the disadvantage of being less intuitively appealing than proofs based on coupling 
arguments, but this is more than
offset by the advantage of fitting into less than two pages without having to appeal to advanced mathematical concepts.
It also has the advantage of being generalisable to situations where level sets of the Lyapunov function
are not small sets, see \cite{HMS09}.
\end{remark}

Before we turn to the proof of Theorem~\ref{theo:Harris}, we define for every
$\beta > 0$ the distance function 
\begin{equ} 
d_\beta(x,y) = \left\{\begin{array}{cl} 0 & \text{if $x = y$} \\ 2 + \beta V(x) + \beta
V(y) & \text{if $x \neq y$.} \end{array}\right. 
\end{equ} 
One can check that the positivity of $V$ implies that this is
indeed a distance function, albeit a rather strange one. 
We define the corresponding ``Lipschitz'' seminorm on functions $\phi\colon \CX \to \R$ by 
\begin{equ} \|\phi\|_{\Lipb} = \sup_{x \neq y} {|\phi(x)
- \phi(y)| \over d_\beta(x,y)}\;. 
\end{equ} 
We are going to make use of the following lemma:

\begin{lemma}\label{lem:equiv} With the above notations, one has
$\|\phi\|_\Lipb = \inf_{c \in \R} \|\phi + c\|_{\beta V}$. 
\end{lemma}

\begin{proof} 
It is obvious that $\|\phi\|_\Lipb \le \|\phi + c\|_{\beta V}$ for every $c \in \R$. 
On the other hand, if $x_0$ is any fixed point in $\CX$, one has 
\begin{equ}[e:aprb] |\phi(x)| \le
|\phi(x_0)| + \|\phi\|_\Lipb \bigl(2 + \beta V(x) + \beta
V(x_0)\bigr)\;, 
\end{equ} 
for all $x \in \CX$. Set now \begin{equ} c = - \sup_{x \in \CX}
\bigl(\phi(x) - \|\phi\|_\Lipb \bigl(1+\beta V(x)\bigr)\bigr)\;.
\end{equ} It follows from \eref{e:aprb} that $c$ is finite. Furthermore,
one has \begin{equ} \phi(y) + c \le \phi(y) - \bigl(\phi(y) -
\|\phi\|_\Lipb \bigl(1+\beta V(y)\bigr)\bigr) = \|\phi\|_\Lipb
\bigl(1+\beta V(y)\bigr)\;, \end{equ} and \begin{equs} \phi(y) + c &=
\inf_{x \in \CX} \bigl(\phi(y) - \phi(x) + \|\phi\|_\Lipb \bigl(1+\beta
V(x)\bigr)\bigr) \\ &\ge \inf_{x \in \CX} \|\phi\|_\Lipb \bigl(1+\beta
V(x) - d_\beta(x,y)\bigr) = - \|\phi\|_\Lipb \bigl(1+\beta V(y)\bigr)\;,
\end{equs} which implies that $\|\phi + c\|_{\beta V} \le
\|\phi\|_\Lipb$.
\end{proof}

\begin{proof}[Proof of Theorem~\ref{theo:Harris}] 
During this proof, we use the notation $\CP \eqdef \CP_{T_0}$ for simplicity.
We are going to show that there exists a choice of
$\beta\in (0,1]$ such that there is $\alpha < 1$ satisfying the bound 
\begin{equ}[e:requested]
|\CP\phi(x) - \CP \phi(y)| \le \alpha d_\beta(x,y) \|\phi\|_\Lipb\;,
\end{equ} 
uniformly over all measurable functions $\phi\colon \CX \to \R$ and all pairs $x,y \in \CX$. 
Note that this is equivalent to the bound $\|\CP\phi\|_\Lipb \le \alpha \|\phi\|_\Lipb$.
Combining this with Lemma~\ref{lem:equiv} and \eref{e:exprTVV}, we obtain that, for $T = n T_0$, one has the bound
\begin{equs}
\|\CP_T^* \mu - \CP_T^* \nu\|_{\TV,V} &= \inf_{\|\phi\|_V \le 1} \int_\CX \bigl(\CP_T\phi\bigr)(x)\, \bigl(\mu - \nu\bigr)(dx) \\
& = \inf_{\|\phi\|_V \le 1} \inf_{c \in \R} \int_\CX \bigl(\bigl(\CP_T\phi\bigr)(x) + c\bigr)\, \bigl(\mu - \nu\bigr)(dx) \\
&\le \inf_{\|\phi\|_V \le 1} \inf_{c \in \R} \|\CP_T\phi + c\|_V \int_\CX \bigl(1+V(x)\bigr)\, \bigl|\mu - \nu\bigr|(dx) \\
& \le  \inf_{\|\phi\|_V \le 1} \beta^{-1} \inf_{c \in \R} \|\CP_T\phi + c\|_{\beta V}  \bigl\|\mu - \nu\bigr\|_{\TV,V} \\
& =  \beta^{-1} \inf_{\|\phi\|_V \le 1} \|\CP_T\phi\|_{\Lipb}  \bigl\|\mu - \nu\bigr\|_{\TV,V} \\
& \le  {\alpha^n \over \beta} \inf_{\|\phi\|_V \le 1} \|\phi\|_{\Lipb}  \bigl\|\mu - \nu\bigr\|_{\TV,V} 
\le {\alpha^n \over \beta} \bigl\|\mu - \nu\bigr\|_{\TV,V}\;.
\end{equs}
Since $\alpha < 1$, the result \eref{e:contraction}  then follows at once by choosing $n$ sufficiently large.

Let us turn now to \eref{e:requested}. If $x = y$, there is nothing to prove, so we assume that $x \neq y$. Fix
an arbitrary non-constant function $\phi$ and assume without loss of
generality that $\|\phi\|_\Lipb = 1$. It follows from Lemma~\ref{lem:equiv}
that, by adding a constant to it if necessary, we can assume that
$|\phi(x) + c| \le \bigl(1 + \beta V(x)\bigr)$.

This immediately implies the bound 
\begin{equs} 
|\CP\phi(x) - \CP\phi(y)|
&\le \bigl(2 + \beta \CP V(x) + \beta \CP V(y)\bigr) \\ &\le 2 +
2\beta K + \beta \gamma V(x) + \beta\gamma V(y)\;. 
\end{equs} 
Suppose now that $x$ and $y$ are such that $V(x) + V(y) \ge {2K + 2\over
1-\gamma}$. A straightforward calculation shows that in this case, for
every $\beta>0$ there exists $\alpha_1 < 1$ such that \eref{e:requested}
holds. One can choose for example 
\begin{equ} 
\alpha_1 = 1-{1\over 2}{\beta \over 1-\gamma + \beta K + \beta}\;. 
\end{equ}

Take now a pair $x,y$ such that $V(x) + V(y) \le {2K + 2\over
1-\gamma}$. Note that we can write $\phi = \phi_1 + \phi_2$ with
$|\phi_1(x)| \le 1$ and $|\phi_2(x)| \le \beta V(x)$. (Set $\phi_1(x) = (\phi(x) \wedge 1) \vee (-1)$.) It follows from
the assumptions on $V$ that there exists some $\delta>0$ such that
\begin{equs}
|\CP\phi(x) - \CP\phi(y)| &\le |\CP\phi_1(x) - \CP\phi_1(y)| + |\CP\phi_2(x) - \CP\phi_2(y)| \\ 
&\le \|\CP^*\delta_x - \CP^*\delta_y\|_\TV + \beta \bigl(\CP V\bigr)(x) + \beta \bigl(\CP V\bigr)(y) \\ 
&\le 2-\delta + \beta \bigl(2K + \gamma
V(x) + \gamma V(y)\bigr) \le 2-\delta + 2\beta K{1+\gamma \over
1-\gamma}\;. 
\end{equs} 
If we now choose $\beta < {\delta\over 4K} {1-\gamma \over
1+\gamma}$, \eref{e:requested} holds with $\alpha = 1- \hf \delta < 1$. Combining this estimate with the
one obtained previously shows that one
can indeed find $\alpha$ and $\beta$ such that \eref{e:requested} holds for all
$x$ and $y$ in $\CX$, thus concluding the proof of Theorem~\ref{theo:Harris}.
\end{proof}

One could argue that this theorem does not guarantee the existence of 
an invariant measure since the fact that $\CP_T^* \mu = \mu$ does not guarantee
that $\CP_t \mu = \mu$ for every $t$. However, one has:
\begin{lemma}
If there exists a probability measure such that $\CP_T^* \mu = \mu$ for some fixed time $T>0$, 
then there also exists
a probability measure $\mu_\infty$ such that $\CP_t^* \mu_\infty = \mu_\infty$ for all $t>0$.
\end{lemma}

\begin{proof}
Define the measure $\mu_\infty$ by
\begin{equ}
\mu_\infty(A) = {1\over T} \int_0^T \CP_t^* \mu (A)\, dt\;.
\end{equ}
It is then a  straightforward exercise to check that it does have the requested property.
\end{proof}

We are now able to use this theorem to obtain the following result on the convergence of the solutions
to \eref{e:linearSPDE} to an invariant measure in the total variation topology:

\begin{theorem}\label{theo:linTV}
Assume that \eref{e:linearSPDE} has a solution in some Hilbert space $\CH$ and that there exists a
time $T$ such that $\|S(T)\| < 1$ and such that $S(T)$ maps $\CH$ into the range of $Q_T^{1/2}$.
Then  \eref{e:linearSPDE} admits a unique invariant measure $\mu_\infty$ and there exists $\gamma>0$ such that
\begin{equ}
\|\CP_t^* \nu - \mu_\infty\|_\TV \le C(\nu) e^{-\gamma t}\;,
\end{equ}
 for every probability measure $\nu$ on $\CH$ with finite second moment.
\end{theorem}

\begin{proof}
Let $V(x) = \|x\|$ and denote by $\mu_t$ the centred Gaussian measure with covariance $Q_t$.
We then have
\begin{equ}
\CP_t V (x) \le \|S(t)x\| + \int_\CH \|x\| \,\mu_t(dx)\;,
\end{equ}
which shows that the first assumption of Theorem~\ref{theo:Harris} is satisfied. A simple variation of
Exercise~\ref{ex:GaussK} (use the decomposition $\CH = \tilde \CH \oplus \ker K$) shows that 
the Cameron--Martin space of $\mu_T$ is given by $\range Q_T^{1/2}$ endowed with the norm
\begin{equ}
\|h\|_T = \inf \{\|x\|\,:\, h = Q_T^{1/2} x\}\;.
\end{equ}
Since we assumed that $S(T)$ maps $\CH$ into the image of $Q_T^{1/2}$, it follows 
from the closed graph theorem that there exists a constant $C$ such that $\|S(T)x\|_T \le C \|x\|$ 
for every $x \in \CH$. It follows from the Cameron--Martin formula that the total variation
distance between $\CP_T^*\delta_x$ and $\CP_T^*\delta_y$ is equal to the total variation
distance between $\CN(0,1)$ and $\CN(\|S(T)x - S(T)y\|_T,1)$, so that the 
second assumption of Theorem~\ref{theo:Harris} is also satisfied.

Both the existence of $\mu_\infty$ and the exponential convergence of $\CP_t^* \nu$ towards
it is then a consequence of Banach's fixed point theorem in the dual to the space of measurable
functions with $\|\phi\|_V < \infty$.
\end{proof}

\begin{remark}
The proof of Theorem~\ref{theo:linTV} shows that if its assumptions are satisfied, then the
map $x \mapsto \CP_t^* \delta_x$ is continuous in the total variation distance for $t \ge T$.
\end{remark}

\begin{remark}
Since $\range  S(t)$ decreases with $t$ and $\range  Q_t^{1/2}$ increases with $t$, it follows that if 
$\range  S(t) \subset \range  Q_t^{1/2}$ for some $t$, then this also holds for any subsequent time.
This is consistent with the fact that Markov operators are contraction operators in the supremum norm,
so that if $x \mapsto \CP_t^* \delta_x$ is continuous in the total variation distance for some $t$,
the same must be true for all subsequent times.
\end{remark}

While Theorem~\ref{theo:linTV} is very general, it is sometimes not straightforward at all to verify
its assumptions for arbitrary linear SPDEs. In particular, it might be complicated \textit{a priori}
to determine the image of $Q_t^{1/2}$. The task of identifying this subspace can be made easier by
the following result:

\begin{proposition}\label{prop:ImQt}
The image of $Q_t^{1/2}$ is equal to the image of the map $A_t$ given by
\begin{equ}
A_t\colon L^2([0,t],\CK) \to \CH\;,\qquad A_t\colon h \mapsto \int_0^t S(s)Qh(s)\, ds\;.
\end{equ}
\end{proposition}

\begin{proof}
Since $Q_t = A_t A_t^*$, we can use polar decomposition \cite[Thm~VI.10]{ReedSimon} to find 
an isometry $J_t$ of $(\ker A_t)^\perp \subset \CH$ (which extends to $\CH$ by setting $J_t x = 0$
for $x \in \ker A_t$) such that $Q_t^{1/2} = A_t J_t$.

Alternatively, one can show that, in the situation of Theorem~\ref{theo:imageGaussian}, the Cameron--Martin space
of $\tilde \mu = A^*\mu$ is given by the image  under $A$ of the Cameron--Martin space of $\mu$.
This follows from Proposition~\ref{prop:defRmu} since, as a consequence of the definition
of the push-forward of a measure, the composition with $A$ yields an isometry between $L^2(\CB,\mu)$
and $L^2(\CB, \tilde \mu)$. 
\end{proof}


One case where it is straightforward to check whether $S(t)$ maps
$\CH$ into the image of $Q_t^{1/2}$ is the following:

\begin{example}\label{ex:selfadj}
Consider the case where $\CK = \CH$,
$L$ is selfadjoint, and there exists a function $f \colon \R \to \R_+$ such that $Q = f(L)$.
(This identity should be interpreted in the sense of the functional calculus already mentioned in
Theorem~\ref{theo:SD}.) 

If we assume furthermore that $f(\lambda)>0$ for every $\lambda \in \R$, then the existence of an invariant 
measure for \eqref{e:linearSPDE} is equivalent to the existence of a constant $c > 0$ such that $\scal{x,Lx} \le -c\|x\|^2$ for every
$x \in \CH$. Using functional calculus, we see that the operator $Q_T$ is then given by
\begin{equ}
Q_T = -{f^2(L)\over 2L} \bigl(1 - e^{2LT}\bigr)\;,
\end{equ}
and, for every $T>0$, the Cameron--Martin norm for $\mu_T$ is equivalent to the norm
\begin{equ}
\|x\|_f = \Bigl\|{\sqrt {-L} \over f(L)}x\Bigr\|\;.
\end{equ}
In order to obtain convergence $\CP_t^* \nu \to \mu_\infty$ in the total variation topology, it is therefore
sufficient that there exist constants $c,C>0$ such that $f(\lambda) \ge C e^{-c \lambda}$ for $\lambda \ge 0$.
\end{example}

\begin{exercise}
In the context of Example~\ref{ex:selfadj}, show that if $\CH$ is infinite-dimensional, $f$ is decreasing, 
\eqref{e:linearSPDE} admits an invariant measure
$\mu_\infty$ on $\CH$, and $\CP_t^*\mu \to \mu_\infty$ in total variation for every $\mu \in \cP(\CH)$,
then there necessarily exist constants $c,C>0$ such that $f(\lambda) \ge C e^{-c \lambda}$ for all $\lambda \ge 0$.
\textbf{Hint:} Show first that the assumptions guarantee that $L$ has a sequence of eigenvectors
with eigenvalues $\lambda_n \to -\infty$.
\end{exercise}

This shows that one cannot expect convergence in the total variation topology to take place under 
similarly weak conditions as in Proposition~\ref{prop:weakconv}. In particular, convergence in the total
variation topology requires some non-degeneracy of the driving noise which was not the case
for \index{weak convergence}{weak convergence}.

\begin{exercise}
Consider again the case $\CK = \CH$ and $L$ selfadjoint with $\scal{x,Lx} \le -c\|x\|^2$ for some $c>0$.
Assume furthermore that $Q$ is selfadjoint and that $Q$ and $L$ commute, so that there exists a space $L^2(\CM,\mu)$ isometric to $\CH$ and
such that both $Q$ and $L$ can be realised as multiplication operators (say $f$ and $g$ respectively) on that space. Show that:
\begin{claim}
\item In order for there to exist solutions in $\CH$, the set $A_Q \eqdef \{\lambda \in \CM\,:\, f(\lambda) \neq 0\}$ must be ``essentially countable'' in the
sense that it can be written as the union of a countable set and a set of $\mu$-measure $0$.
\item If there exists $T>0$ such that $\range  S(T) \subset \range  Q_T^{1/2}$, then $\mu$ is purely atomic and there exists some possibly different
time $t>0$ such that $S(t)$ is trace class.
\end{claim}
\end{exercise}

Example~\ref{ex:selfadj} suggests that there are many cases where, if $S(t)$ maps $\CH$ to $\range Q_t^{1/2}$ for some 
$t>0$, then it does so for all $t>0$. It also shows that, in the case where $L$ and $Q$ are selfadjoint and commute,
$Q$ must have an orthonormal basis of eigenvectors with all eigenvalues non-zero. Both statements are 
certainly not true in general. We see from the following example  that there can be infinite-dimensional
situations where $S(t)$ maps $\CH$ to $\range Q_t^{1/2}$ even though $Q$ is of rank one!

\begin{example}\label{ex:delay}
Consider the space $\CH = \R \oplus L^2([0,1],\R)$ and denote elements of $\CH$ by $(a,u)$ with $a \in \R$.
Consider the semigroup $S$ on $\CH$ given by
\begin{equ}
S(t) \bigl(a,u\bigr) = (a, \tilde u)\;,\quad \tilde u(x) = \left\{\begin{array}{rl} a & \text{for $x \le t$} \\ u(x-t) & \text{for $x>t$.} \end{array}\right.
\end{equ}
It is easy to check that $S$ is indeed a strongly continuous semigroup on $\CH$ and we denote its generator
by $(0,\d_x)$. We drive this equation by adding noise only on the first component of $\CH$. In other words,
we set $\CK = \R$ and $Q 1 = (1,0)$ so that, formally, we are considering the equation
\begin{equ}
da = dW(t)\;,\qquad du = \d_x u\, dt\;.
\end{equ}
Even though, at a formal level, the equations for $a$ and $u$ look decoupled, they are actually coupled
via the domain of the generator of $S$ which enforces the identity $u(0,t) = a(t)$ for all $t > 0$. 
In order to check whether $S(t)$ maps $\CH$ into $\CH_t \eqdef \range  Q_t^{1/2}$,
we make use of Proposition~\ref{prop:ImQt}. This shows that $\CH_t$ consists of elements
of the form
\begin{equ}
\int_0^t h(s) \chi_s\, ds\;,
\end{equ}
where $h \in L^2([0,t])$ and $\chi_s$ is the image of $(1,0)$ under $S(s)$, which is given by $(1,\one_{[0,s\wedge 1]})$.
On the other hand, the image of $S(t)$ consists of all elements $(a,u)$ such that $u(x) = a$ for $x \le t$.
Since one has $\chi_s(x) = 0$ for $x > s$, it is obvious that $\range  S(t) \not\subset \CH_t$ for $t < 1$.

On the other hand, for $t > 1$, given any $a>0$, we can find a function $h \in L^2([0,t])$ such that
$h(x) = 0$ for $x \le 1$ and $\int_0^t h(x)\, dx = a$. Since, for $s \ge 1$, one has $\chi_s(x) = 1$ for every $x \in [0,1]$,
it follows that one does have $\range  S(t) \subset \CH_t$ for $t < 1$.
\end{example}

\section{Semilinear SPDEs}

Now that we have a good working knowledge of the behaviour of solutions to linear stochastic PDEs, we
are prepared to turn to nonlinear SPDEs. In these notes, we will restrict ourselves to the study
of \textit{semilinear} SPDEs with \textit{additive} noise. 

In this context, a \textit{semilinear} SPDE is one such that the nonlinearity can be treated as a perturbation
of the linear part of the equation. The word \textit{additive} for the noise refers to the fact that, as in \eref{e:linearSPDE},
we will only consider noises described by a fixed operator $Q \colon \CK \to \CB$, rather than by an operator-valued
function of the solution. We will therefore consider equations of the type
\begin{equ}[e:SPDE]
dx = Lx\, dt + F(x)\, dt + Q\, dW(t)\;,\quad x(0) = x_0\in \CB\;,
\end{equ}
where $L$ is the generator of a strongly continuous semigroup $S$ on a separable Banach space $\CB$, $W$ is a cylindrical
Wiener process on some separable Hilbert space $\CK$, and $Q\colon \CK \to \CB$ is bounded. Furthermore, 
$F$ is a measurable function from some linear subspace $\CD(F) \subset \CB$ into $\CB$. We will say that
a process $t \mapsto x(t) \in \CD(F)$ is a \index{mild solution}\textit{mild solution} to \eref{e:SPDE} if the identity
\begin{equ}[e:mildsol]
x(t) = S(t)x_0 + \int_0^t S(t-s)F(x(s))\, ds + \int_0^t S(t-s)Q\,dW(s)\;.
\end{equ}
holds almost surely for every $t>0$. 
Throughout this section, we will make the standing assumption that the linearisation
to \eref{e:SPDE} (that is the corresponding equation with $F=0$) does have a continuous
solution with values in $\CB$. In order to simplify notations, we are going to write
\begin{equ}
\WL(t) \eqdef \int_0^t S(t-s)Q\,dW(s)\;.
\end{equ}
This process is sometimes called the \index{stochastic convolution}stochastic
convolution.

The simplest case (which is however not very relevant in the context of PDEs)
is that when $\CD(F) = \CB$ and $F$ is globally Lipschitz continuous on $\CB$. In this case,
we have the following result.

\begin{theorem}\label{theo:existunique}
Consider \eref{e:SPDE} on a Banach space $\CB$ and assume that $\WL$ is a continuous $\CB$-valued process.
Assume furthermore that $F\colon \CB \to \CB$ is Lipschitz continuous.
Then, there exists a unique mild solution to \eref{e:SPDE}.
\end{theorem}

\begin{proof}
Given any realisation $\WL \in \CC(\R_+, \CB)$ of the stochastic
convolution, we are going to show that there exists a unique
continuous function $x \colon \R_+\to \CB$ such that
\eref{e:mildSC} holds for every $t$. In this sense, our solution theory is completely pathwise, there is
no real need for stochastic integration.

As in the classical case of the solution theory for ordinary differential equations, the proof relies on 
the Banach fixed point theorem. Given a terminal time $T>0$ and a continuous function 
$g \colon \R_+ \to \CB$, we define the map $M_{g,T}\colon \CC([0,T],\CB)\to \CC([0,T],\CB)$ by
\begin{equ}[e:defMgT]
\bigl(M_{g,T} u\bigr)(t) = \int_0^t S(t-s)F(u(s))\, ds + g(t)\;.
\end{equ}
Note that we can assume without loss of generality that
the semigroup $S$ is bounded, since we can always subtract a constant to $L$ and add it back to $F$ (see Exercise~\ref{ex:uniqueRep} below).
Using the fact that $\|S(t)\| \le M$ for some
constant $M$, one has for any $T>0$ the bound
\begin{equ}[e:boundM1]
\sup_{t \in [0,T]} \|M_{g,T} u(t) - M_{g,T} v(t)\| \le M T \sup_{t \in [0,T]} \|F(u(t)) - F(v(t))\| \;.
\end{equ}
Writing $\Lip F$ for its Lipschitz constant, it follows that
if one chooses $T$ such that $MT\Lip F < 1$, then $M_{g,T}$ is a contraction on $\CC([0,T],\CB)$ and therefore
admits a unique fixed point there by the Banach fixed point theorem. 
If we choose $g(t) = S(t)x_0 + W_L(t)$, then this solves indeed \eqref{e:mildsol} on $[0,T]$.

In order to construct a solution for all times, we note that for $t>T$ solutions to \eqref{e:mildsol} can be written as
\begin{equs}
x(t) &= S(t-T)\Big(S(T)x_0 + \int_0^T S(T-s)F(x(s))\, ds + \int_0^T S(T-s)Q\,dW(s)\Big) \\
&\quad + \int_T^t S(t-s)F(x(s))\, ds + \int_T^t S(t-s)Q\,dW(s) \label{e:iterationSol} \\
&= S(t-T)x(T) + \int_T^t S(t-s)F(x(s))\, ds + \int_T^t S(t-s)Q\,dW(s)\;.
\end{equs}
In other words, $s \mapsto x(T+s)$ solves the exact same fixed point problem as $x$, but with $x_0$ 
replaced by $x(T)$ and the noise $W$ shifted by $T$ (which does not change its law). This allows us to
concatenate solutions in essentially the same way as for ODEs and completes the proof.
\end{proof}

\subsection{Local solutions}

In the nonlinear case, there are situations where solutions explode after a finite 
(but possibly random) time interval. In order to be able to account for such a situation, we introduce the notion of
a \index{local solution}\textit{local solution}. Recall first that, given a cylindrical Wiener process $W$ defined on some probability space $(\Omega,\P)$,
we can associate to it the natural filtration\index{filtration} $\{\F_t\}_{t \ge 0}$ generated by the increments of $W$. In other words,
for every $t>0$, $\F_t$ is the smallest $\sigma$-algebra such that the random variables $W(s) - W(r)$ for $s,r \le t$ are
all $\F_t$-measurable.

In this context, a \index{stopping time}\textit{stopping time} is a positive random variable $\tau$ such that the event 
$\{\tau \le T\}$ is $\F_T$-measurable for every $T \ge 0$. The feature of stopping times that we are going to use here is that 
if we restart a Wiener process at a stopping time, we do again get a Wiener process. Here, we always assume that Wiener processes are defined on
positive times. Indeed, we have the following result:

\begin{proposition}\label{prop:strongMarkov}
Let $W$ be a Wiener process, let $\F$ be the filtration generated by its increments, and let $\tau$ be a 
stopping time. Then, the process $W^\tau(t) = W(t+\tau)- W(\tau)$ is again a Wiener process.
\end{proposition}

\begin{proof}
We assume that $W$ is real-valued since this makes no difference. We also recall that, as a consequence of the independence of the
increments of $W$, it follows that for any $t > 0$ and any 
continuous bounded function $\Phi\colon \CC(\R_+,\R) \to \R$, one has the almost sure identity
\begin{equ}
\E \big(\Phi(W^t)\,|\,\F_t\bigr) = \E \Phi(W)\;,
\end{equ}
where we set $W^t(s) = W(t+s) - W(t)$. 

We first assume that $\tau$ only takes finitely 
many values $t_1,\ldots t_n$. In this case, we have
\begin{equs}
\E \Phi(W^\tau) &= \E \sum_i \one_{\tau = t_i} \Phi(W^\tau) = 
 \E \sum_i \one_{\tau = t_i} \E\big(\Phi(W^\tau) \,|\,\F_{t_i}\big)
 = \E \sum_i \one_{\tau = t_i} \E\Phi(W) = \E \Phi(W)\;,
\end{equs}
as claimed. Let now $\tau$ be arbitrary and let $\tau_n$ be the stopping time defined by
\begin{equ}
\tau_n = \sum_{k=1}^{n^2} \frac {k}n \one_{n\tau \in (k-1,k]} + n \one_{\tau > n}\;.
\end{equ}
By the above, we then have $\E \Phi(W^{\tau_n}) = \E \Phi(W)$. Since $\tau_n \to \tau$ almost surely and since
$W$ has continuous sample paths, it follows from Lebesgue's dominated convergence theorem that
$\E \Phi(W^{\tau}) = \E \Phi(W)$, which is precisely the claim.
\end{proof}

\begin{exercise}
Retracing the proof of Proposition~\ref{prop:strongMarkov} to show the following stronger result. 
Define $\F_\tau$ as the $\sigma$-algebra generated by all events of the form $A \cap \{\tau \ge t\}$ with $A \in \F_t$.
Then, for every $\Phi$ as above, one has the almost sure identity $\E \big(\Phi(W^\tau)\,|\,\F_\tau\bigr) = \E \Phi(W)$.
\end{exercise}

With the definition of a stopping time at hand, we can define the following notion of a local solution.

\begin{definition}
A \textit{local mild solution} to \eref{e:SPDE} is a $\CD(F)$-valued stochastic process $x$ together with a stopping time $\tau$
such that $\tau > 0$ almost surely and such that, for every $t \ge 0$, the identity
\begin{equ}[e:mildSC]
x(t) = S(t)x_0 + \int_0^{t} S(t-s)F(x(s))\, ds + \WL(t)\;,
\end{equ}
holds almost surely on the event $\{t < \tau\}$. 
\end{definition}

\begin{remark}\label{rem:FanalSG}
In some situations, it might be of advantage to allow $F$ to be a map from $\CD(F)$ to $\CB'$ for some
superspace $\CB'$ such that $\CB \subset \CB'$ densely and such that $S(t)$ extends to a continuous linear map from $\CB'$
to $\CB$. The prime example of such a space $\CB'$ is an interpolation space with negative index in the
case where the semigroup $S$ is analytic. The definition of a mild solution carries over to this situation
without any change.
\end{remark}

It is obvious from the definition that if $(x,\tau)$ is a local mild solution and $\bar \tau$ is any stopping time
with $\bar \tau \le \tau$ almost surely, then $(x,\bar \tau)$ is also a mild solution. We furthermore consider two local mild
solutions $(x,\tau)$ and $(y, \bar \tau)$ as identical if $\tau = \bar \tau$ almost surely and, for every $t \ge 0$,
$x(t) = y(t)$ holds almost surely on the event $\{t < \tau\}$. The values $x(t)$ and $y(t)$ are however allowed to disagree on $\{t \ge \tau\}$.
A local mild solution $(x,\tau)$ is called \textit{maximal} if,
for every mild solution $(\tilde x, \tilde \tau)$, one has $\tilde \tau \le \tau$ almost surely.

\begin{exercise}\label{ex:uniqueRep}
Show that local mild solutions to \eref{e:SPDE} coincide with local mild solutions to \eref{e:SPDE}
with $L$ replaced by $\tilde L = L - c$ and $F$ replaced by $\tilde F = F + c$ for any constant $c \in \R$.
(Here the notion of coinciding is as above, i.e.\ up to the stopping time.)
\end{exercise}

Our first result on the existence and uniqueness of local mild solutions to nonlinear SPDEs makes the strong assumption
that the nonlinearity $F$ is defined on the whole space $\CB$ and that it is locally Lipschitz there:

\begin{theorem}\label{theo:existunique}
Consider \eref{e:SPDE} on a Banach space $\CB$ and assume that $\WL$ is a continuous $\CB$-valued process.
Assume furthermore that $F\colon \CB \to \CB$ is such that it restriction to every bounded set is Lipschitz continuous.
Then, there exists a unique maximal mild solution $(x,\tau)$ to \eref{e:SPDE}. Furthermore, this solution has continuous sample paths
and one has $\lim_{t \uparrow \tau} \|x(t)\| = \infty$ almost surely on the set $\{\tau < \infty\}$.

If $F$ is globally Lipschitz continuous, then $\tau = \infty$ almost surely.
\end{theorem}

\begin{proof}
Given any realisation $\WL \in \CC(\R_+, \CB)$ of the \index{stochastic convolution}stochastic
convolution, we are going to show that there exists a time $\tau > 0$ depending only on $\WL$ up to time $\tau$ and a unique
continuous function $x \colon [0,\tau)\to \CB$ such that
\eref{e:mildSC} holds on $\{t < \tau\}$. Furthermore, the construction will be such that either $\tau = \infty$, or 
one has $\lim_{t \uparrow \tau} \|x(t)\| = \infty$, thus showing that $(x,\tau)$ is maximal.

As before, the proof relies on the Banach fixed point theorem and we define $M_{g,T}$ as in \eqref{e:defMgT}.
We also assume again that the semigroup $S$ is bounded.
Besides \eqref{e:boundM1}, one then has for any $T>0$ the bound
\begin{equ}[e:boundM2]
\sup_{t \in [0,T]} \|M_{g,T} u(t) - g(t)\| \le M T \sup_{t \in [0,T]} \|F(u(t))\| \;.
\end{equ}
Fix now an arbitrary constant $R>0$. Since $F$ is locally Lipschitz, it follows from \eref{e:boundM1} and \eref{e:boundM2} that
there exists a maximal $T > 0$ such that $M_{g,T}$ maps the ball of radius $R$ around $g$ in $\CC([0,T],\CB)$ into itself and 
is a contraction with contraction constant $\hf$ there. 

Setting $g(t) = S(t)x_0 + \WL(t)$, the event $\{T \le t\}$ is clearly $\F_t$-measurable 
(since $g(s)$ is $\F_s$-measurable) so that $T$ is a stopping time. Furthermore,
the pair $(x,T)$, where $T$ is as just constructed and $x$ is the unique fixed point of $M_{g,T}$,
thus yields a local mild solution to \eref{e:SPDE}.

In order to construct the maximal solution, we iterate this construction in the same way as in the globally Lipschitz case,
noting that the decomposition \eqref{e:iterationSol} still holds on the event $\{t>T\}$ when $T$ is a stopping time.
Uniqueness and continuity in time also follows as in the finite-dimensional case.
\end{proof}

While this setting is very straightforward and did not really make use of any PDE theory, it nevertheless allows to construct
solutions for an important class of examples, since every composition operator of the form
$\bigl(N(u)\bigr)(\xi) = \bigl(f \circ u\bigr)(\xi)$ is locally Lipschitz on $\CC(K,\R^d)$ (for $K$ a compact subset of $\R^n$, say), 
provided that $f \colon \R^d \to \R^d$ is locally Lipschitz continuous.

A much larger class of examples can be found if we assume that $L$
generates an analytic semigroup:

\begin{theorem}\label{theo:bootstrap}
Let $L$ generate an analytic semigroup on $\CB$ (denote by $\CB_\alpha$, $\alpha \in \R$ the corresponding
interpolation spaces) and assume that $Q$ is such that the \index{stochastic convolution}stochastic convolution
$\WL$ has almost surely continuous sample paths in $\CB_\alpha$ for some $\alpha \ge 0$. Assume furthermore
that there exists $\gamma \ge 0$ and $\delta \in [0,1)$ such that,  for every $\beta \in [0,\gamma]$, the map $F$ extends to a locally Lipschitz continuous map from $\CB_\beta$ to $\CB_{\beta-\delta}$ that grows at most
polynomially.

Then, \eref{e:SPDE} has a unique maximal mild solution $(x,\tau)$ with $x$ taking values in $\CB_\beta$ for
every $\beta < \beta_\star \eqdef \alpha \wedge (\gamma + 1-\delta)$.
\end{theorem}

\begin{proof}
In order to show that \eref{e:SPDE} has a unique mild solution, we proceed in a way similar to the proof of Theorem~\ref{theo:existunique}
and we make use of Exercise~\ref{ex:ASGab} to bound $\|S(t-s)F(u(s))\|$ in terms of $\|F(u(s))\|_{-\delta}$. This yields instead
of \eref{e:boundM1} the bound
\begin{equ}[e:boundM1delta]
\sup_{t \in [0,T]} \|M_{g,T} u(t) - M_{g,T} v(t)\| \le M  T^{1-\delta} \sup_{t \in [0,T]} \|F(u(t)) - F(v(t))\| \;,
\end{equ}
and similarly for \eref{e:boundM2}, thus showing that \eref{e:SPDE} has a unique $\CB$-valued maximal mild solution $(x,\tau)$.
In order to show that $x(t)$ actually belongs to $\CB_\beta$ for $t < \tau$ and $\beta \le \alpha \wedge \gamma$, we make use of
a ``bootstrapping'' argument, which is essentially an induction on $\beta$.

For notational convenience, we introduce the family of processes $\WL^a(t) = \int_{at}^t
S(t-r) Q\, dW(r)$, where $a \in [0,1)$ is a parameter. Note that one has the identity
\begin{equ}
\WL^a(t) = \WL(t) - S((1-a)t)\WL(at)
\end{equ}
so that if $\WL$ is continuous with values in $\CB_\alpha$, then the same is true for $\WL^a$.
 
We are actually going to show the following stronger statement. Fix an arbitrary time $T > 0$. Then, for every $\beta
  \in [0, \beta_\star)$ there exist exponents $p_\beta \ge 1$,
  $q_ \beta \ge 0$,  and constants $a \in (0,1)$, $C > 0$ such that the bound
  \begin{equ}[e:boundugamma]
    \|x_t\|_\beta \le C t^{-q_\beta} \Bigl(1 + \sup_{s \in [at,t]}\|x_s\| + \sup_{0 \le s \le t}
    \|\WL^{a}(s)\|_\beta \Bigr)^{p_\beta}\;,
  \end{equ}
 holds almost surely for all $t \in (0,T]$.

The bound \eref{e:boundugamma} is obviously true for $\beta =
  0$ with $p_\beta = 1$ and $q_\beta = 0$.  Assume now that, for some
$\beta = \beta_0 \in
  [0,\gamma]$, the bound \eref{e:boundugamma} holds.  
  We will then argue that, for any $\eps \in (0, 1-\delta)$,
  the statement \eref{e:boundugamma} also holds for $\beta =
  \beta_0 + \eps$ (and therefore also for all intermediate values), provided that we adjust the constants appearing
  in the expression. Since it is possible to go from $\beta =
  0$ to any value of $\beta < \gamma + 1-\delta$ in a finite number of such
  steps, the claim then follows at once.

From the definition of a mild solution, we have the identity
  \begin{equ}
    x_t = S((1-a) t) x_{a t} + \int_{a t}^t
    S(t-s) F(x(s))\, ds + \WL^a(t)\;.
  \end{equ}
  Since $\beta \le \gamma$, it follows from our polynomial growth assumption on $F$ that
  there exists $n >0$ such that, for $t \in (0,T]$,
  \begin{equs}
    \|x_t\|_{\beta+\eps} &\le C t^{-\eps} \|x_{a t}\|_\beta
    + \|\WL^a(t)\|_{\beta+\eps} + C\int_{a t}^t (t-s)^{-(\eps
      + \delta)} (1+\|x_s\|_\beta)^n\, ds\\
    &\le C \bigl(t^{-\eps} + t^{1-\eps - \delta}\bigr)\sup_{{a t}
      \le s \le t}\bigl(1+\|x_s\|_\beta^n\bigr) + \|\WL^a(t)\|_{\beta+\eps} \\
    &\le C t^{-\eps}\sup_{{a t} \le s \le
      t}\bigl(1+\|x_s\|_\beta^n\bigr) + \|\WL^a(t)\|_{\beta+\eps}\;.
  \end{equs}
  Here, the constant $C$ depends on everything but $t$ and $x_0$.
  Using the induction hypothesis, this yields the bound
  \begin{equ}
    \|x_t\|_{\beta+\eps} \le C t^{-\eps - n q_\beta} \bigl(1 +
    \sup_{s \in [a^2 t,t]}\|x_s\| + \sup_{0 \le s \le
      t} \|\WL^a(s)\|_\beta \bigr)^{n p_\beta} + \|\WL^a(t)\|_{\beta+\eps} \;,
  \end{equ}
  thus showing that \eref{e:boundugamma} does indeed hold for $\beta =
  \beta_0 + \eps$, provided that we replace $a$ by $a^2$  and set $p_{\beta +
    \eps} = n p_\beta$ and $q_{\beta + \eps} = \eps +
  nq_{\beta}$. This concludes the proof of
  Theorem~\ref{theo:bootstrap}.
\end{proof}


\subsection{Reaction-diffusion equations}

\index{stochastic!reaction-diffusion equation}
This is a class of partial differential equations that model the evolution of reactants in a gel, described by a spatial domain $D$. 
They are of the type
\begin{equ}[e:SGL]
du = \Delta u\, dt + f \circ u \, dt + Q\, dW(t)\;,
\end{equ}
where $u(x,t) \in \R^d$, $x \in D$, describes the density of the various components of the reaction
at time $t$ and location $x$. The nonlinearity $f \colon \R^d \to \R^d$ describes the reaction itself and the 
 term $\Delta u$ describes the diffusion of the reactants in the gel. The noise term $Q\, dW$ should be interpreted
as a crude attempt to describe the fluctuations in the quantities of reactant due both to the discrete nature
of the underlying particle system and the interaction with the environment.\footnote{A more realistic description of these fluctuations
would result in a covariance $Q$ that depends on the solution $u$. Since we have not developed the tools necessary to treat this type
of equations, we restrict ourselves to the simple case of a constant covariance operator $Q$.}

Equations of the type \eref{e:SGL} also appear in the theory of amplitude equations, where they appear as a kind of
``normal form'' near a change of linear instability. In this particular case, one often has $d = 2$ and $f(u) = \kappa u - u |u|^2$
for some $\kappa \in \R$, see for example \cite{BloHaiPav}.
A natural choice for the Banach space $\CB$ in which to consider solutions to \eref{e:SGL} is the space
of bounded continuous functions $\CB = \CC(D, \R^d)$ since the composition operator $u \mapsto f \circ u$
(also sometimes called Nemitskii operator) then maps $\CB$ into itself and inherits the regularity properties of $f$.
If the domain $D$ is sufficiently regular then the semigroup generated by the Laplacian $\Delta$ is the Markov semigroup for a 
Brownian motion in $D$. The precise description of the domain of $\Delta$ is related to the behaviour of the corresponding 
Brownian motion when it hits the boundary of $D$.
In order to avoid technicalities, let us assume from now on that $D$ consists of the torus $\T^n$, so that
there is no boundary to consider. 

\begin{exercise}
Use the explicit form of the heat kernel to show that $\Delta$ generates an analytic semigroup on $\CB=\CC(\T^n, \R^d)$.
\end{exercise}

\begin{remark}
For $2\alpha \not \in \N$, it turns out that the interpolation spaces $\CB_\alpha$ associated with the
heat semigroup on $\CB=\CC(\T^n)$ coincide with $\CC^{2\alpha}(\T^n)$ (or rather the completion of
smooth functions under the corresponding Hölder norm). This however is \textit{not} true in general
as can be seen already with the case $\alpha = 1$. If we take a smooth function $f$ on the two-dimensional torus
minus the origin such that $f(x,y) = (x^2-y^2)\sqrt{|\log(x^2 + y^2)|}$ for $x^2 + y^2 < 1/2$ (say),
then it is not difficult to check that $f$ belongs to the domain of $\Delta$, but $f$ is of course not $\CC^2$.
\end{remark}
 
If $Q$ is such that the \index{stochastic convolution}stochastic convolution has continuous sample paths in $\CB$ almost surely
and $f$ is locally Lipschitz continuous,
we can directly apply Theorem~\ref{theo:existunique} to obtain the existence of a unique local solution to \eref{e:SGL}
in $\CC(\T^n,\R^d)$. We would like to obtain conditions on $f$ that ensure that this local solution is also a global solution,
that is the blow-up time $\tau$ is equal to infinity almost surely. 

If $f$ happens to be a globally Lipschitz continuous function, then the existence and uniqueness of global solutions follows
from Theorem~\ref{theo:existunique}.
Obtaining global solutions when $f$ is not globally Lipschitz continuous 
is slightly more tricky. The idea is to obtain some \textit{a priori} estimate
on some functional of the solution which dominates the supremum norm and 
ensures that it cannot blow up in finite time.

Let us first consider the deterministic part of the equation alone. The structure we are going to exploit is the fact that the
Laplacian generates a Markov semigroup. We have the following general result:

\begin{lemma}\label{lem:contrPtV}
Let $\CP_t$ be a \index{Feller semigroup}Feller\footnote{A Markov semigroup is Feller if it maps continuous functions into continuous functions.} Markov semigroup over a Polish space $\CX$. Extend it to $\CC_b(\CX,\R^d)$ by applying it to each
component independently. Let $V \colon \R^d \to \R_+$ be convex (that is $V(\alpha x + (1-\alpha)y) \le \alpha V(x) + (1-\alpha) V(y)$
for all $x,y \in \R^d$ and $\alpha \in [0,1]$) and define $\tilde V \colon \CC_b(\CX,\R^d) \to \R_+$ by
$\tilde V(u) = \sup_{x \in \CX} V(u(x))$. Then $\tilde V(\CP_t u) \le \tilde V(u)$ for every $t \ge 0$ and every $u \in \CC_b(\CX,\R^d)$.
\end{lemma}

\begin{proof}
Note first that if $V$ is convex, then it is continuous and, for every probability measure $\mu$ on $\R^d$, one has Jensen's
inequality \cite{Jensen}
\begin{equ}[e:convex]
V \Bigl(\int_{\R^d}x\,\mu(dx)\Bigr) \le \int_{\R^d} V(x)\, \mu(dx)\;.
\end{equ}
One can indeed check by induction that \eref{e:convex} holds if $\mu$ is a ``simple'' measure
consisting of a convex combination of finitely many Dirac measures.
The general case then follows from the continuity of $V$ and the fact that every probability measure on 
$\R^d$ can be approximated (in the topology of weak convergence) by a sequence of simple measures.

Denote now by $\CP_t(x,\cdot\,)$ the transition probabilities for $\CP_t$, so that $\CP_t u$ is given by the formula
$\bigl(\CP_t u\bigr)(x) = \int_\CX u(y)\,\CP_t(x,dy)$. One then has
\begin{equs}
\tilde V \bigl(\CP_t u\bigr) &= \sup_{x \in \CX} V\Bigl(\int_\CX u(y)\,\CP_t(x,dy)\Bigr) = \sup_{x \in \CX}  V\Bigl(\int_{\R^d} v\,\bigl(u^*\CP_t(x,\cdot\,)\bigr)(dv)\Bigr) \\
&\le \sup_{x \in \CX}  \int_{\R^d} V (v)\,\bigl(u^*\CP_t(x,\cdot\,)\bigr)(dv) = \sup_{x \in \CX}  \int_{\CX} V (u(y))\,\CP_t(x,dy) \\
&\le \sup_{y \in \CX}V (u(y)) = \tilde V(u)\;,
\end{equs}
as required.
\end{proof}

In particular, this result can be applied to the semigroup $S(t)$ generated by the Laplacian in \eref{e:SGL}, so that
$\tilde V(S(t)u) \le \tilde V(u)$ for every convex $V$ and every $u \in \CC(\T^n, \R^d)$. This is the main ingredient allowing
us to obtain a priori estimates on the solution to \eref{e:SGL}:

\begin{proposition}\label{prop:global}
Consider the setting for equation \eref{e:SGL} described above. Assume that $Q$ is such that $\WD$ has continuous sample
paths in $\CB = \CC(\T^n, \R^d)$ and that there exists a convex twice differentiable function $V \colon \R^d \to \R_+$ such that
$\lim_{|x| \to \infty} V(x) = \infty$ and such that, for every $R>0$, there exists a constant $C$ such that
$\scal{\nabla V(x), f(x+y)} \le C V(x)$ for every $x \in \R^d$ and every $y$ with $|y| \le R$. Then \eref{e:SGL} has a global solution in $\CB$.
\end{proposition}

\begin{proof}
We denote by $u(t)$ the local mild solution to \eref{e:SGL}. Our aim is to obtain \textit{a priori} bounds on $\tilde V(u(t))$ that are sufficiently
good to show that one cannot have $\lim_{t \to \tau} \|u(t)\| = \infty$ for any finite (stopping) time $\tau$.

Setting
$v(t) = u(t) - \WD(t)$, the definition of a mild solution shows that $v$ satisfies the equation
\begin{equ}
v(t+h) = e^{\Delta h} v(t) + \int_t^{t+h} e^{\Delta(t+h-s)} \bigl(f \circ (v(s) + \WD(s))\bigr)\, ds \eqdef e^{\Delta h} v(t) + \int_0^h e^{\Delta(h-s)} F(t+s)\, ds\;.
\end{equ}
Since $t \mapsto v(t)$ is continuous by Theorem~\ref{theo:existunique} and the same holds for $\WD$ by assumption, the function
$t \mapsto F(t)$ is continuous in $\CB$. Therefore, one has
\begin{equ}
\lim_{h \to 0} {1\over h}\Bigl( \int_0^h e^{\Delta(h-s)} F(t+s)\, ds - h e^{\Delta h}F(t)\Bigr) = 0\;.
\end{equ}
We therefore obtain for $\tilde V(v)$ the bound
\begin{equs}
\limsup_{h \to 0} h^{-1} \bigl(\tilde V(v(t+h)) - \tilde V(v(t))\bigr) &= \limsup_{h \to 0} h^{-1} \bigl(\tilde V\big(e^{\Delta h}(v(t) + h F(t))\big) - \tilde V(v(t))\bigr)\\
&\le \limsup_{h \to 0} h^{-1} \bigl(\tilde V\big(v(t) + h F(t)\big) - \tilde V(v(t))\bigr)\;,
\end{equs}
where we used Lemma~\ref{lem:contrPtV} to obtain the inequality.
Since $V$ belongs to $\CC^2$ by assumption, we have
\begin{equs}
\tilde V(v(t) + hF(t)) &= \sup_{x \in \T^n} \bigl(V(v(x,t)) + h \scal{\nabla V(v(x,t)), F(x,t)}\bigr) + \CO(h^2)\\
&\le \tilde V(v(t)) + \sup_{x \in \T^n} h \scal{\nabla V(v(x,t)), F(x,t)} + \CO(h^2)\;.
\end{equs}
Using the definition of $F$ and the assumptions on $V$, it follows that for every $R>0$ there exists a constant $C$ such that, 
provided that $\|\WD(t)\| \le R$, one has
\begin{equ}
\limsup_{h \to 0} h^{-1} \bigl(\tilde V(v(t+h)) - \tilde V(v(t))\bigr) \le C \tilde V(v(t))\;.
\end{equ}
By a standard comparison argument for ODEs, this implies that $\tilde V(v(t)) \le e^{Ct}\tilde V(v(0))$,
thus showing that $v$ (and a fortiori $u$) does not blow up since $\|\WD(t)\|$ is bounded on
bounded time intervals, thus concluding the proof. 
\end{proof}

\begin{exercise}
In the case $d=1$, show that the assumptions of Proposition~\ref{prop:global} are satisfied for $V(u) = u^2$ if $f$ is any polynomial
of odd degree with negative leading coefficient.
Show that in this case solutions are not only global, but such that $\sup_{t > 0}\E \sup_x |u(t,x)|^2$
is finite. \textbf{Hint:} Define $v = u-W_{\Delta - a}$ for suitable $a>0$ and derive a bound
on $\E \tilde V(v(t))$ by following the argument of Proposition~\ref{prop:global}.
\end{exercise}

\begin{exercise}
Show that in the case $d = 3$, \eref{e:SGL} has a unique global solution when we take for $f$ the right-hand side of the Lorentz
attractor:
\begin{equs}
f(u) = \begin{pmatrix} \sigma(u_2 - u_1) \cr u_1 (\rho - u_3) - u_2 \cr u_1 u_2 - \beta u_3 \end{pmatrix}\;,
\end{equs}
where $\rho$, $\sigma$ and $\beta$ are three arbitrary positive constants.
\end{exercise}

\subsection{Interpolation inequalities and Sobolev embeddings}

\index{Sobolev!embedding}The kind of bootstrapping arguments used in the proof of Theorem~\ref{theo:bootstrap} above
 are extremely useful to obtain regularity properties of the solutions
to semilinear parabolic stochastic PDEs. However, they rely on obtaining bounds on
the regularity of $F$ from one interpolation space into another. In many important situations, the interpolation spaces
turn out to be given by \index{Sobolev!space (fractional)}fractional Sobolev spaces. For the purpose of these notes, we are going to restrict ourselves
to the analytically easier situation where the space variable of the stochastic PDE under consideration takes
values in the $d$-dimensional torus $\T^d$. In other words, we restrict ourselves to situations where the operator describing the linearised
evolution is endowed with periodic boundary conditions.

This will make the proof of the embedding theorems presented in these notes technically more straightforward.
For the corresponding embeddings with more general boundary conditions or even on more general manifolds
or unbounded domains, we refer for example the comprehensive series of monographs \cite{Tri83,Tri92,Tri06}.

Recall that, given a distribution
$u \in L^2(\T^d)$, we can decompose it as a Fourier series:
\begin{equ}
u(x) = \sum_{k \in \Z^d} u_k e^{i\scal{k,x}}\;,
\end{equ}
where the identity holds for (Lebesgue) almost every $x \in \T^d$. Furthermore, the $L^2$ norm of $u$ is given by Parseval's 
identity $\|u\|^2 = \sum |u_k|^2$. We have

\begin{definition}
The fractional Sobolev space $H^s(\T^d)$ for $s \ge 0$  is given by the subspace of functions $u \in L^2(\T^d)$
such that
\begin{equ}[e:defHs]
\|u\|_{H^s}^2 \eqdef \sum_{k \in \Z^d} (1+|k|^2)^s |u_k|^2 < \infty\;.
\end{equ}
Note that this is a separable Hilbert space and that $H^0 = L^2$. For $s < 0$, we define $H^s$ as the closure of $L^2$
under the norm \eref{e:defHs}.
\end{definition}

\begin{remark}
By the spectral decomposition theorem, $H^s$ for $s>0$ is the domain of $(1-\Delta)^{s/2}$
and we have $\|u\|_{H^s} = \|(1-\Delta)^{s/2}u\|_{L^2}$.
\end{remark}

A very important situation is
the case where $L$ is a differential operator with constant coefficients (formally $L = P(\d_x)$ for some polynomial
$P\colon \R^d \to \R$)
and $\CH$ is either an $L^2$ space or some Sobolev
space. In this case, one has

\begin{lemma}
Assume that $P\colon \R^d \to \R$ is a polynomial of degree $2m$ such that there exist positive constants $c,C$ such that the bound
\begin{equ}
(-1)^{m+1} c |k|^{2m} \le P(k) \le (-1)^{m+1} C |k|^{2m}\;,
\end{equ}
holds for all $k$ outside of some compact set. Then, the operator $P(\d_x)$ generates an analytic semigroup on $\CH = H^s$ for every $s \in \R$
and the corresponding interpolation spaces are given by $\CH_\alpha = H^{s+2m\alpha}$.
\end{lemma}

\begin{proof}
By inspection, noting that $P(\d_x)$ is conjugate to the multiplication operator by $P(ik)$ via the Fourier decomposition.
\end{proof}

Note first that for any two positive real numbers $a$ and $b$ and any pair of positive conjugate exponents $p$ and $q$, 
one has Young's inequality\index{Young inequality}
\begin{equ}[e:Young]
ab \le {a^p \over p} + {b^q \over q}\;,\qquad {1\over p} + {1 \over q} = 1\;.
\end{equ}
As a corollary of this elementary bound, we obtain H\"older's inequality\index{H\"older inequality}, which can be viewed as a generalisation
of the Cauchy-Schwartz inequality:

\index{H\"older inequality}
\begin{proposition}[H\"older's inequality]
Let $(\CM, \mu)$ be a measure space and let $p$ and $q$ be a pair of positive conjugate exponents. Then, for any pair
of measurable functions $u,v \colon \CM \to \R$, one has
\begin{equ}
\int_{\CM} |u(x) v(x)|\,\mu(dx) \le \|u\|_p \, \|v\|_q\;,
\end{equ}
for any pair $(p,q)$ of conjugate exponents.
\end{proposition}

\begin{proof}
It follows from \eref{e:Young} that, for every $\eps>0$, one has the bound
\begin{equ}
\int_{\CM} |u(x) v(x)|\,\mu(dx) \le {\eps^p \|u\|_p^p \over p} +  {\|v\|_q^q \over q \eps^q}\;,
\end{equ}
Setting $\eps = \|v\|_q^{1\over p} \|u\|_p^{{1\over p}-1}$ concludes the proof.
\end{proof}

One interesting consequence of H\"older's inequality is the following interpolation inequality for powers of selfadjoint operators:
\begin{proposition}\label{prop:Holderselfadj}
Let $A$ be a positive definite selfadjoint operator on a separable Hilbert space $\CH$ and let $\alpha \in [0,1]$. Then, the bound
$\|A^\alpha u\| \le \|Au\|^\alpha \|u\|^{1-\alpha}$ holds for every $u \in \CD(A^\alpha) \subset \CH$.
\end{proposition}
\begin{proof}
The extreme cases $\alpha \in \{0,1\}$ are obvious, so we assume $\alpha \in (0,1)$. By the spectral theorem, we
can assume that $\CH = L^2(\CM,\mu)$ and that $A$ is  the multiplication operator by some positive function $f$.
Applying H\"older's inequality with $p = 1/\alpha$ and $q = {1/(1-\alpha)}$, one then has
\begin{equs}
\|A^\alpha u\|^2 &= \int f^{2\alpha}(x) u^2(x)\,\mu(dx) = \int |fu|^{2\alpha}(x) \, |u|^{2-2\alpha}(x)\,\mu(dx) \\
&\le \Bigl(\int f^2(x) u^2(x) \,\mu(dx)\Bigr)^{\alpha}\,\Bigl(\int u^2(x) \,\mu(dx)\Bigr)^{1-\alpha}\;,
\end{equs}
which is exactly the bound we wanted to show.
\end{proof}

An immediate corollary is:
\begin{corollary}
For any $t>s$ and any $r \in [s,t]$,  the bound
\begin{equ}[e:interp]
\|u\|_{H^r}^{t-s} \le \|u\|_{H^t}^{r-s} \|u\|_{H^s}^{t-r}
\end{equ}
is valid for every $u \in H^t$.
\end{corollary}

\begin{proof}
Apply Proposition~\ref{prop:Holderselfadj} with $\CH = H^s$, $A = (1-\Delta)^{t-s\over 2}$, and $\alpha = (r-s)/(t-s)$.
\end{proof}

\begin{exercise}\label{ex:multipleHolder}
As a consequence of H\"older's inequality, show that for any collection of $n$ measurable functions
and any exponents $p_i > 1$ such that $\sum_{i=1}^n p_i^{-1} = 1$, one has the bound
\begin{equ}
\int_{\CM} |u_1(x)\cdots u_n(x)|\,\mu(dx) \le \|u_1\|_{p_1}\cdots \|u_n\|_{p_n}\;.
\end{equ}
\end{exercise}

Following our earlier discussion regarding \index{Sobolev!space (fractional)}fractional Sobolev spaces, it would be convenient to
be able to bound the $L^p$ norm of a function in terms of one of the fractional Sobolev norms.
It turns out that bounding the $L^\infty$ norm is rather straightforward:

\begin{lemma}\label{lem:Linfty}
For every $s > {d \over 2}$, the space $H^s(\T^d)$ is contained in the space of continuous functions
and there exists a constant $C$ such that $\|u\|_{L^\infty} \le C \|u\|_{H^s}$.
\end{lemma}

\begin{proof}
It follows from Cauchy-Schwarz that
\begin{equ}
\sum_{k\in \Z^d} |u_k| \le \Bigl(\sum_{k\in \Z^d} (1+|k|^2)^s |u_k|^2\Bigr)^{1/2} \Bigl(\sum_{k\in \Z^d} (1+|k|^2)^{-s}\Bigr)^{1/2}\;.
\end{equ}
Since the sum in the second factor converges if and only if $s > {d\over 2}$, the claim follows.
\end{proof}

\begin{exercise}
In dimension $d = 2$, find an example of an unbounded function $u$ such that $\|u\|_{H^1} < \infty$.
\end{exercise}

\begin{exercise}
Show that for $s > {d \over 2}$, $H^s$ is contained in the space $\CC^\alpha(\T^d)$ for every $\alpha < s - {d\over 2}$.
\end{exercise}

As a consequence of Lemma~\ref{lem:Linfty}, we are able to obtain a more general \index{Sobolev!embedding}Sobolev embedding  for all $L^p$ spaces:

\begin{theorem}[Sobolev embeddings]
Let $p \in [2,\infty]$.
Then, for every $s > {d \over 2} - {d\over p}$, the space $H^s(\T^d)$ is contained in the space $L^p(\T^d)$
and there exists a constant $C$ such that $\|u\|_{L^p} \le C \|u\|_{H^s}$.
\end{theorem}

\begin{proof}
The case $p=2$ is obvious and the case $p=\infty$ has already been shown, so it remains to show the claim for $p \in (2,\infty)$.
The idea is to divide Fourier space into ``blocks'' corresponding to different length scales and to estimate separately the
$L^p$ norm of every block. More precisely, we define a sequence of functions $u^{(n)}$ by
\begin{equ}
u^{-1}(x) = u_0\;,\quad u^{(n)}(x) = \sum_{2^n \le |k| < 2^{n+1}} u_k e^{i\scal{k,x}}\;,
\end{equ}
so that one has $u = \sum_{n \ge -1} u^{(n)}$. For $n \ge 0$, one has
\begin{equs}[e:boundLp]
\|u^{(n)}\|_{L^p}^p \le \|u^{(n)}\|_{L^2}^2 \|u^{(n)}\|_{L^\infty}^{p-2}\;.
\end{equs}
Choose now $s' = {d\over 2} + \eps$ and note that the construction of $u^{(n)}$, together with Lemma~\ref{lem:Linfty}, guarantees that one has the bounds
\begin{equ}
\|u^{(n)}\|_{L^2} \le 2^{-ns} \|u^{(n)}\|_{H^s}\;,\quad \|u^{(n)}\|_{L^\infty} \le C\|u^{(n)}\|_{H^{s'}} \le C 2^{n(s'-s)} \|u^{(n)}\|_{H^s}\;.
\end{equ}
Inserting this into \eref{e:boundLp}, we obtain for $c = 1-2/p$,
\begin{equ}
\|u^{(n)}\|_{L^p} \le C\|u^{(n)}\|_{H^s} 2^{n \bigl((s'-s){p-2\over p} - {2s\over p}\bigr)} = C\|u^{(n)}\|_{H^s} 2^{n \bigl({d\over 2} - {d\over p} + c\eps -s\bigr)}
\le C\|u\|_{H^s} 2^{n \bigl({d\over 2} - {d\over p}+c\eps-s\bigr)}\;.
\end{equ}
It follows that $\|u\|_{L^p} \le |u_0| + \sum_{n \ge 0}\|u^{(n)}\|_{L^p} \le C\|u\|_{H^s}$, provided that the exponent appearing in this
expression is negative, which is precisely the case whenever $s > {d \over 2} - {d\over p}$, provided that $\eps>0$ is
chosen sufficiently small.
\end{proof}

\begin{remark}
For $p \neq \infty$, one actually has $H^s(\T^d) \subset L^p(\T^d)$ with $s = {d \over 2} - {d\over p}$, but this borderline
case is more difficult to obtain.
\end{remark}

Combining the Sobolev embedding theorem and H\"older's inequality, it is eventually possible to estimate in a similar way 
the fractional Sobolev norm of a product of two functions:

\begin{theorem}\label{theo:multSob}
Let $r$, $s$ and $t$ be positive exponents such that $s \wedge r > t$ and $s+r > t + {d\over 2}$. Then, if $u \in H^r$ and $v \in H^s$, the product $w = uv$ belongs
to $H^t$.
\end{theorem}

\begin{proof}
Define $u^{(n)}$ and $v^{(m)}$ as in the proof of the Sobolev embedding theorem and set $w^{(m,n)} = u^{(m)} v^{(n)}$. Note that one
has $w^{(m,n)}_k = 0$ if $|k| > 2^{3+(m\vee n)}$. It then follows from H\"older's inequality that if $p,q \ge 2$ are such that
$p^{-1} + q^{-1} = \hf$, one has the bound
\begin{equs}
\|w^{(m,n)}\|_{H^t}  \le C 2^{t (m\vee n)} \|w^{(m,n)}\|_{L^2}
 \le C 2^{t (m\vee n)} \|u^{(m)}\|_{L^p}\|v^{(n)}\|_{L^q}\;.
\end{equs}
Assume now that $m>n$.
The conditions on $r$, $s$ and $t$ are such that there exists a pair $(p,q)$ as above with
\begin{equ}
r > t + {d\over 2} - {d\over p}\;,\qquad
s > {d\over 2} - {d\over q}\;.
\end{equ}
In particular, we can find some $\eps>0$ such that 
\begin{equ}
\|u^{(m)}\|_{L^p} \le C \|u^{(m)}\|_{H^{r-t-\eps}} \le C 2^{-m(t+\eps)}\|u\|_{H^r}\;,\quad
\|v^{(n)}\|_{L^p} \le C \|v^{(n)}\|_{H^{s-\eps}} \le C 2^{-m\eps}\|v\|_{H^s}\;.
\end{equ}
Inserting this into the previous expression, we find that
\begin{equ}
\|w^{(m,n)}\|_{H^t} \le C 2^{-m\eps -n\eps} \|u\|_{H^r}\|u\|_{H^s}\;.
\end{equ}
Since our assumptions are symmetric in $u$ and $v$, we obtain a similar bound for the case $m \le n$, so that
\begin{equ}
\|w\|_{H^t} \le \sum_{m, n>0} \|w^{(m,n)}\|_{H^t} \le C\|u\|_{H^r}\|u\|_{H^s}\sum_{m,n>0} 2^{-m\eps -n\eps} \le C\|u\|_{H^r}\|u\|_{H^s}\;,
\end{equ}
as requested.
\end{proof}

\begin{exercise}
Show that the conclusion of Theorem~\ref{theo:multSob} still holds if $s = t = r$ is a positive integer, provided that $s > {d\over 2}$.
\end{exercise}

\begin{exercise}
Similarly to Exercise~\ref{ex:multipleHolder}, show that one can iterate this bound so that if $s_i > s \ge 0$ are exponents such that
$\sum_i s_i > s + {(n-1)d\over 2}$, then one has the bound
\begin{equ}
\|u_1\cdots u_n\|_s \le C \|u_1\|_{s_1}\cdots \|u_n\|_{s_n}\;.
\end{equ}
\textbf{Hint:} The case $s \ge {d\over 2}$ is simple, so it suffices to consider the case $s < {d\over 2}$.
\end{exercise}

The functional inequalities from the previous section allow to check that the assumptions of Theorems~\ref{theo:existunique}
and \ref{theo:bootstrap} are verified
by a number of interesting equations.

\subsection{The stochastic Navier--Stokes equations}
\index{stochastic!Navier--Stokes equations}

The Navier--Stokes equations govern the motion of an idealised incompressible fluid and are one
of the most studied models in the theory of partial differential equations, as well as in theoretical 
and mathematical physics. We are going to use the symbol $u(x,t)$ to denote the instantaneous
velocity of the fluid at position $x \in \R^d$ and time $t$, so that $u(x,t) \in \R^d$. With these notations, the deterministic Navier--Stokes
equations are given by
\begin{equ}[e:NSD]
\d_t u = \nu \Delta u - (u\cdot \nabla)u - \nabla p \;, \qquad \div u = 0\;,
\end{equ}
where the (scalar) pressure $p$ is determined implicitly by the incompressibility condition $\div u = 0$ and $\nu>0$ denotes
the kinematic viscosity of the fluid. In principle, these equations make sense for any value of the dimension $d$. However,
even the deterministic equations \eref{e:NSD} are known to have global smooth solutions
for  arbitrary smooth initial data only in dimension $d = 2$. We are therefore going to restrict ourselves to the two-dimensional 
case in the sequel. As we saw already in the introduction, solutions to \eref{e:NSD} tend to $0$ as times goes to $\infty$,
so that an external forcing is required in order to obtain an interesting stationary regime.

One natural way of adding an external forcing is given by a stochastic force that is white in time and admits a translation
invariant correlation function in space. In this way, it is possible to maintain the translation invariance of the equations
(in a statistical sense), even though the forcing is not constant in space. We are furthermore going to restrict ourselves
to solutions that are periodic in space in order to avoid the difficulties arising from partial differential equations in
unbounded domains.
The incompressible stochastic Navier--Stokes equations on the torus $\R^2$ are given by
\begin{equ}[e:SNS]
du = \nu \Delta u\, dt - (u\cdot \nabla)u \, dt - \nabla p\, dt + Q\, dW(t)\;, \qquad \div u = 0\;,
\end{equ}
where $p$ and $\nu>0$ are as above. In order to 
put these equations into the more familiar form \eref{e:SPDE}, we denote by $\Pi$ the orthogonal
projection onto the space of divergence-free vector fields. In Fourier components, $\Pi$ is given by
\begin{equ}[e:defPi]
(\Pi u)_k =  u_k - {k \scal{k,u_k}\over |k|^2}\;.
\end{equ}
(Note here that the Fourier coefficients of a vector field are themselves vectors.) With this notation,
one has 
\begin{equ}
du = \nu \Delta u\, dt - \Pi (u\cdot \nabla)u \, dt + Q\, dW(t) \eqdef \Delta u\, dt + F(u) \, dt + Q\, dW(t) \;.
\end{equ}
It is clear from \eref{e:defPi} that $\Pi$ is a contraction in any fractional Sobolev space. For $t \ge 0$, it 
therefore follows that
\begin{equ}[e:boundF]
\|F(u)\|_{H^t} \le \|u\|_{H^s} \|\nabla u\|_{H^r} \le C\|u\|_{H^s}^2\;,
\end{equ}
provided that $s \ge 1$ and $s + (s-1) > 2t \vee \bigl(t + {d\over 2}\bigr)$, which in dimension $2$
yields the constraint $s > (t+ \f 12) \vee \bigl({t\over 2} + 1\bigr)$.
In particular, this bound holds for $s = t+1$, provided that $t>0$.

Furthermore, in this setting, since $L$ is just the Laplacian, if we choose $\CH = H^s$, then the
interpolation spaces $\CH_\alpha$ are given by $\CH_\alpha = H^{s + 2\alpha}$.
This allows us to apply Theorem~\ref{theo:bootstrap} to show that the stochastic Navier--Stokes equations
admit local solutions for any initial condition in $H^s$, provided that $s > 1$, and that the stochastic convolution takes
values in that space. Furthermore, these solutions
will immediately lie in any higher order Sobolev space, all the way up to the space in which the
stochastic convolution lies.

This line of reasoning does however not yield any \textit{a priori} bounds on the solution, so that it may blow
up in finite time. The Navier--Stokes nonlinearity satisfies $\scal{u,F(u)} = 0$ (the scalar product is the $L^2$ 
scalar product), so one may expect bounds in $L^2$, but we do not know at this stage whether initial conditions
in $L^2$ do indeed lead to local solutions. We would therefore like to obtain bounds on $F(u)$ in
negative Sobolev spaces. In order to do this, we exploit the fact that $H^{-s}$ can naturally be identified with the
dual of $H^s$, so that 
\begin{equ}
\|F(u)\|_{H^{-s}} = \sup \Bigl\{\int F(u)(x)\, v(x)\, dx\;,\quad v \in \CC^\infty \;,\quad \|v\|_{H^s} \le 1\Bigr\}\;.
\end{equ}
Making use of the fact that we are working with divergence-free vector fields, one has
(using Einstein's convention of summation over repeated indices):
\begin{equ}
\int F(u)\, v\, dx = -\int v_j u_i \d_i u_j\, dx  \le \|v\|_{L^p} \|\nabla u\|_{L^2} \|u\|_{L^q} \;,
\end{equ}
provided that $p,q > 2$ and ${1\over p} + {1\over q} = {1\over 2}$. We now make use of the fact that
$\|u\|_{L^q} \le C_q \|\nabla u\|_2$ for every $q \in [2,\infty)$ (but $q = \infty$ is excluded) to conclude that
for every $s > 0$ there exists a constant $C$ such that
\begin{equ}[e:boundFs]
\|F(u)\|_{-s} \le C \|\nabla u\|_{L^2}^2\;.
\end{equ}

In order to get \textit{a priori} bounds for the solution to the 2D stochastic Navier--Stokes equations, 
 one can then make use of the following trick: introduce the vorticity
$w = \nabla \wedge u = \d_1 u_2 - \d_2 u_1$. Then, provided that $\int u\, dx = 0$ (which, provided that the range of $Q$
consists of vector fields with mean $0$, is a condition that is preserved under the solutions to \eref{e:SNS}),
the vorticity is  sufficient to describe $u$ completely by making use of the incompressibility assumption $\div u = 0$.
Actually, the map $w \mapsto u$ can be described explicitly by
\begin{equ}
u_k = \bigl(Kw\bigr)_k = {k^\perp w_k\over |k|^2}\;,\qquad (k_1, k_2)^\perp = (-k_2, k_1)\;.
\end{equ}
This shows in particular that $K$ is actually a bounded operator from $H^s$ into $H^{s+1}$ for every $s$.
It follows that one can rewrite \eref{e:SNS}  as
\begin{equ}[e:SNSvort]
dw = \nu \Delta w\, dt + (Kw\cdot \nabla)w \, dt + \tilde Q\, dW(t) \eqdef \Delta w\, dt + \tilde F(w) \, dt + \tilde Q\, dW(t) \;.
\end{equ}
Since $\tilde F(w) = \nabla \wedge F(Kw)$, it follows from \eref{e:boundFs} that one has the bounds
\begin{equ}
\|\tilde F(w)\|_{-1-s} \le C \|w\|_{L^2}^2\;,
\end{equ}
so that $\tilde F$ is a locally Lipschitz continuous map from $L^2$ into $H^{s}$ for every $s<-1$.
This shows that \eref{e:SNSvort} has unique local solutions for every initial condition in $L^2$ and that these
solutions immediately become as regular as the corresponding stochastic convolution.

Denote now by $\tilde \WL$ the stochastic convolution
\begin{equ}
\tilde \WL(t) = \int_0^t e^{\Delta(t-s)} \tilde Q\,dW(s)\;,
\end{equ}
and define the process $v(t) = w(t) - \WL(t)$. With this notation, $v$ is the unique solution to the random PDE
\begin{equ}
\d_t v = \nu \Delta v + \tilde F(v + \tilde \WL)\;.
\end{equ}
It follows from \eref{e:boundF} that $\|\tilde F(w)\|_{H^{-s}} \le C \|w\|_{H^s}^2$, provided that $s > 1/3$.
For the sake of simplicity, assume from now on that $\tilde \WL$ takes values in $H^{1/2}$ almost surely.
Using the fact that $\scal{v, \tilde F(v)} = 0$ (which is only the case in dimension $2$!), 
we then obtain for the $L^2$-norm of $v$ the following
\textit{a priori} bound:
\begin{equs}
\d_t \|v\|^2 &= -2 \nu \|\nabla v\|^2 - 2\scal{\tilde \WL, \tilde F(v+\tilde \WL)}\\
&\le -2 \nu \|\nabla v\|^2 + 2 \|\tilde \WL\|_{H^{1/2}} \|v+\tilde \WL\|_{H^{1/2}}^2 \\
&\le -2 \nu \|\nabla v\|^2 + 4 \|\tilde \WL\|_{H^{1/2}} \bigl(\|v\|_{H^{1/2}}^2 + \|\tilde \WL\|_{H^{1/2}}^2\bigr) \\
&\le -2 \nu \|\nabla v\|^2 + 4 \|\tilde \WL\|_{H^{1/2}} \bigl(\|v\| \|\nabla v\| + \|\tilde \WL\|_{H^{1/2}}^2\bigr) \\
&\le {8 \over \nu} \|\tilde \WL\|_{H^{1/2}}^2 \|v\| ^2  + 2\|\tilde \WL\|_{H^{1/2}}^3\;, \label{e:aprioriNS}
\end{equs}
so that global existence of solutions then follows from Gronwall's inequality.

This calculation is only formal, since it is not known in general whether the $L^2$-norm of $v$ is differentiable
as a function of time. The bound that one obtains from \eref{e:aprioriNS} can however be made rigorous in a very similar way as
for the example of the stochastic reaction-diffusion equation, so that we will not reproduce this argument
here.
\newpage

\bibliographystyle{Martin}
\bibliography{./refs}

\printindex

\end{document}